\documentclass[a4paper,11pt]{amsart}	
\usepackage[foot]{amsaddr}

\usepackage{geometry}
\usepackage[T1]{fontenc}
\usepackage{lmodern}
\usepackage{microtype}

\usepackage{mathtools}
\usepackage{amssymb}
\usepackage{mathrsfs}
\usepackage{bm}
\usepackage{nccmath}
\allowdisplaybreaks[1]

\usepackage{graphicx, float}
\graphicspath{{Images/}}
\usepackage[
  font=small,
  labelfont=bf
]{caption}

\usepackage{enumitem}
\usepackage{array,booktabs,longtable}

\usepackage[sort,nocompress]{cite}
\usepackage{xcolor}

\definecolor{darkcyan}{RGB}{0,75,75}
\definecolor{citationred}{RGB}{125,70,70}
\definecolor{doiplum}{RGB}{130,70,115}

\usepackage[
  colorlinks=true,
  linkcolor=darkcyan,
  citecolor=citationred,
  urlcolor=doiplum
]{hyperref}

\theoremstyle{plain}
\newtheorem{theorem}{Theorem}[section]
\newtheorem*{theorem*}{Theorem}
\newtheorem{lemma}[theorem]{Lemma}
\newtheorem{corollary}[theorem]{Corollary}
\newtheorem{proposition}[theorem]{Proposition}

\theoremstyle{definition}
\newtheorem{definition}[theorem]{Definition}

\theoremstyle{remark}

\theoremstyle{plain}
\newtheorem{maintheorem}{Theorem}

\newtheorem{maincorollary}[maintheorem]{Corollary}

\usepackage{xcolor}
\usepackage[
  capitalize,
  nameinlink
]{cleveref}
\numberwithin{equation}{section}

\newcommand{\blackcrefname}[3]{%
  \crefname{#1}{\textcolor{black}{#2}}{\textcolor{black}{#3}}%
  \Crefname{#1}{\textcolor{black}{#2}}{\textcolor{black}{#3}}%
}

\blackcrefname{appendix}{Appendix}{Appendices}

\blackcrefname{theorem}{Theorem}{Theorems}
\blackcrefname{lemma}{Lemma}{Lemmas}
\blackcrefname{algorithm}{Algorithm}{Algorithms}
\blackcrefname{corollary}{Corollary}{Corollaries}
\blackcrefname{proposition}{Proposition}{Propositions}
\blackcrefname{definition}{Definition}{Definitions}
\blackcrefname{remark}{Remark}{Remarks}

\blackcrefname{maintheorem}{Theorem}{Theorems}
\blackcrefname{maincorollary}{Corollary}{Corollaries}

\blackcrefname{section}{Section}{Sections}

\newcommand{\norm}[1]{\left\lVert #1\right\rVert}
\newcommand{\tnorm}[1]{\lVert #1\rVert}

\newcommand{\abs}[1]{\left \lvert #1\right\rvert}
\newcommand{\tabs}[1]{\lvert #1\rvert}
\newcommand{\scalar}[2]{\left \langle #1,#2\right\rangle}
\newcommand{\hnorm}[1]{\left[#1 \right]}

\DeclareMathOperator{\dist}{dist}

\let\Re\undefined
\let\Im\undefined
\DeclareMathOperator{\Re}{Re}
\DeclareMathOperator{\Im}{Im}

\renewcommand{\d}{\mathrm d}
\newcommand{\A}{\mathrm A}
\newcommand{\cA}{\mathcal A}

\newcommand{\sL}{\mathscr L}

\DeclareMathOperator{\Herg}{H}

\title[Asymptotics of Bergman Polynomials under Finite Regularity]{Asymptotics of Weighted Bergman Polynomials under Finite Regularity}

\author{Lennart A. Hübner}

\address{KU Leuven, Department of Mathematics,
Celestijnenlaan 200B, 3001 Leuven, Belgium}

\email{lennart.huebner@kuleuven.be}

\begin{document}

\begin{abstract}
We study strong asymptotics of planar orthogonal polynomials with respect to fixed weighted area measures on Jordan domains. The aim is to understand how the regularity of the domain and the weight enter the large-degree asymptotic behavior of these polynomials. 

For domains with real-analytic boundaries, Suetin established leading-order asymptotics for Hölder continuous weights, while Hedenmalm--Wennman obtained a full asymptotic expansion for real-analytic weights. We interpolate between these results while allowing both the boundary and the weight to have finite smoothness. The resulting expansions have a length determined by the available regularity and hold in the exterior region and, under sufficient regularity conditions, in a shrinking neighborhood of the boundary.
\end{abstract}
\newgeometry{margin=32mm,top=18mm,bottom=20mm,head=20pt}

\clearpage\maketitle
\thispagestyle{empty}

\makeatletter
\renewcommand{\l@subsection}{%
  \@tocline{2}{0pt}{3pc}{3pc}{}}
\renewcommand{\l@subsubsection}{%
  \@tocline{3}{0pt}{5pc}{3pc}{}}
\makeatother

\vspace{-.5cm}
\begingroup
  \renewcommand{\bm}[1]{#1}
  \tableofcontents
\endgroup

\restoregeometry
\clearpage

\section{Introduction and Main Results}\label{sec:intro}
Let \(\mathscr D\) be a Jordan domain in the complex plane and \(\omega\not\equiv 0\) a non-negative continuous weight function on \(\overline{\mathscr D}\). We write \(L^2(\mathscr D,\omega)\) for the standard \(L^2\)-space induced by the measure \(1_\mathscr D \omega\d \A\), where \(\d\A\) denotes the area measure normalized by \(1/\pi\). For two functions \(f,g\in L^2(\mathscr D,\omega)\) the corresponding inner product is defined by
\[\scalar{f}{g}_{\mathscr D,\omega}\coloneqq \int _\mathscr D f(z) \overline{g(z)}\, \omega(z)\,\d \A(z).\]
We denote by \(P_N\) the \(N\)-th orthogonal polynomial, that is the unique polynomial of degree \(N\) which has positive leading coefficient, and satisfies \(\tnorm{P_N}_{L^2(\mathscr D,\omega)}=1\) as well as
\[\scalar{\smash{z^k}}{P_N}_{\mathscr D,\omega}=0,\qquad k=0,\dots, N-1.\]
These polynomials are commonly referred to as \textit{weighted Bergman polynomials}.

The simplest example is obtained by taking \(\mathscr D\) to be the unit disk \(\mathbb D\) and \(\omega\equiv1\). In this case,
\[P_N(z)=\sqrt{N+1}\,z^N.\]
For general domains and weights, explicit formulas of this kind are rarely available. A notable exception is provided by \cite{Akemann2021Gegenbauer}, where the Bergman polynomials on ellipses for a certain class of weight functions are represented explicitly in terms of Gegenbauer polynomials. Nevertheless, the large-degree behavior of \(P_N\) can often be described in terms of the geometry of \(\mathscr D\) and the boundary behavior of \(\omega\).

For now, let us consider the case in which \(\mathscr D\) has a real-analytic boundary and \(\omega\equiv 1\). We denote by \(\phi\) the exterior conformal mapping  from \(\mathscr D^e\coloneqq \mathbb C \setminus \overline{\mathscr D}\) onto  \(\mathbb D^e\coloneqq \mathbb C \setminus \overline{\mathbb D}\),  normalized by \(\phi(\infty)=\infty\) and \(\phi'(\infty)>0\). Since \(\partial \mathscr D\) is real-analytic, the conformal map \(\phi\) extends holomorphically across the boundary. The classical work by Carleman \cite{carleman1922classic} shows that the orthogonal polynomials can be approximated on some neighborhood \(U\) of \(\mathscr D^c\coloneqq \mathbb C\setminus \mathscr D\). In fact, there exists \(c>0\) such that, uniformly for \(z\in U\),
\[P_N(z)=\sqrt{N+1}\phi'(z)\phi^N(z)\Bigl(1+\mathcal O\Bigl(e^{-cN}\Bigr)\Bigr).\]
In essence, the asymptotic behavior of \(P_N\) is obtained by a change of variables via the exterior conformal map \(\phi\). For a modern reference, see \cite[p. 12]{gaier1987complex}.

Motivated by this picture, our aim is to determine how the regularity of \(\mathscr D\) and the weight \(\omega\) influence the asymptotic behavior of \(P_N\) as \(N\to\infty\). More precisely, we fix \(\varkappa\in\mathbb Z_{\geq 0}\) and \(\alpha\in(0,1)\), and consider a Jordan domain \(\mathscr D\) with \(C^{\varkappa+1,\alpha}\)-smooth boundary and a weight \(\omega\) which belongs to \(C^{\varkappa,\alpha}\) on a neighborhood of \(\overline{\mathscr D}\) and is strictly positive on \(\partial\mathscr D\). Under these regularity assumptions, we derive a finite asymptotic expansion for \(P_N\), together with explicit error estimates,  in the exterior of \(\mathscr D\) and in a shrinking neighborhood of its boundary, see \cref{thm:Main1,thm:Main2} below.

\subsection{Historical Background}\label{sec:history}
Motivated by an extension of Carleman's theorem due to Korovkin \cite{korovkin1947asymptotic} to a special class of analytic weights, Suetin studied the asymptotic behavior of Bergman polynomials associated with Hölder continuous weights on real-analytic Jordan domains. By introducing a multiplicative Szeg\H{o}-type correction function, he extended the existing asymptotic representation to positive weights of Hölder class \(C^{0,\alpha}\). More precisely, in \cite{suetin1974polynomials}, he established
\begin{equation}\label{eq:Suetin-asymp}
P_N(z)=\sqrt{N+1}\phi'(z)\phi^N(z)e^{V(z)}\biggl(1+\mathcal O\biggl(\Bigl(\mfrac{\log N}{N}\Bigr)^{\alpha/2}\biggr)\biggr),
\qquad z\in K\subset\widehat{\mathbb C}\setminus \overline{\mathscr D},
\end{equation}
where \(K\) is compact in the Riemann sphere \(\widehat{\mathbb C}\), and \(V\) is holomorphic in the exterior of \(\mathscr D\). The correction function \(V\) is characterized by the boundary condition \(2\Re V=-\log\omega\) on \(\partial\mathscr D\), together with the normalization \(V(\infty)\in\mathbb R\). Its role is to compensate for the effect of the weight \(\omega\) near the boundary.
The proof of \eqref{eq:Suetin-asymp} relies on \(L^2\)-estimates over a shrinking annular neighborhood of the boundary. The main difficulty concerns the domain on which the correction function is defined. Since \(\omega\) is only Hölder continuous, \(V\) need not extend across \(\partial\mathscr D\) into \(\mathscr D\). To overcome this difficulty, Suetin introduced the dilated correction term
\[V_N(z)=V\bigl(\phi^{-1}(\phi(z)/r_N)\bigr),\]
where \(r_N<1\) and \(1-r_N=\mathcal O(\frac{\log N}{N})\). However, this dilation argument does not by itself yield substantial improvements under stronger regularity assumptions on the weight. 

More recently, Hedenmalm and Wennman considered the case of real-analytic boundaries and weights. In \cite{hedenmalm2024riemann}, they analyzed the asymptotic behavior using the foliation flow method developed in \cite{hedenmalm2017planar}. This approach allowed them to construct a full asymptotic expansion of the orthogonal polynomials. More precisely, if \(\partial\mathscr D\) and the positive weight \(\omega\) are real-analytic, then there exist bounded holomorphic functions \(V_j\), defined on a neighborhood of \(\mathscr D^c\), such that, for every \(\varkappa\in\mathbb Z_{\geq 0}\) and any \(R>0\),
\begin{equation}\label{eq:HW-asymp}
P_N(z)=\sqrt{N+1}\phi'(z)\phi^N(z)\exp\Bigl({\textstyle\sum_{j=0}^{\varkappa}N^{-j}V_j(z)}\Bigr)\Bigl(1+\mathcal O\bigl(N^{-\varkappa-1}\bigr)\Bigr),
\end{equation}
holds uniformly for all \(z\) with \(\dist(z,\mathscr D^c)\leq\frac{R\log N}{N}\).
Here, the leading correction function \(V_0\) coincides with Suetin's function \(V\), whereas the higher-order correction functions are determined by analogous boundary value problems. In contrast to \eqref{eq:Suetin-asymp}, this expansion is not merely locally uniform in the exterior of \(\mathscr D\), but remains valid in a shrinking boundary layer extending slightly into \(\mathscr D\). This is a consequence of the real-analyticity of \(\partial\mathscr D\) and \(\omega\), which ensures that the exterior conformal map and the correction functions \(V_j\) extend holomorphically across the boundary.

The preceding results primarily illustrate how the regularity of the weight influences the asymptotic behavior of the orthogonal polynomials. A separate line of research concerns the effect of reduced boundary regularity. In the unweighted case \(\omega\equiv1\), several boundary regularity classes have been considered. Suetin \cite{Suetin1969AreallyOrthogonal,suetin1974polynomials} established corresponding asymptotic formulas for \(C^{\varkappa,\alpha}\)-smooth boundaries, whereas Stylianopoulos \cite{Stylianopoulos2013StrongAsymptotics} treated piecewise analytic boundaries without cusps. Further refinements of Carleman-type asymptotics and their applications, for example to the distribution of zeros, were obtained by Beckermann, Dragnev, Gustafsson, Mi\~na-Díaz, Putinar, Saff, and Stylianopoulos \cite{GustafssonEtAl2009Archipelago, BeckermannStylianopoulos2018Grunsky, DragnevMinaDiaz2010Carleman, MinaDiaz2008AsymptoticIntegral, MinaDiaz2023Bergman}.

\subsection{Main Results}\label{sec:main-result}
Whereas Suetin and Hedenmalm--Wennman treat the two extreme cases of Hölder continuous and real-analytic weights, respectively, no analogous results appear to be available for intermediate degrees of regularity. Moreover, existing results generally assume that either the boundary is real-analytic or the weight is constant. To the best of our knowledge, no finite higher-order expansion with quantitative remainder estimates has previously been established when both the boundary and the weight have only finite regularity.

The aim of this paper is to develop an asymptotic expansion for the orthogonal polynomials \(P_N\) when both the boundary and the weight have finite regularity. In \cref{thm:Main1,thm:Main2} as well as \cref{coro:main} we summarize the main results established in \cref{theo:Expansion1,theo:Expansion2}. The corresponding proofs can be found at the end of \cref{sec:asymp expansion}
.

We fix \(\varkappa\in\mathbb Z_{\geq 0}\) and \(\alpha\in(0,1)\), and assume that \(\mathscr D\) is a Jordan domain with \(C^{\varkappa+1,\alpha}\)-smooth boundary. Further, let \(\phi\) denote the exterior conformal mapping normalized at infinity. That is, \(\phi\) is the conformal mapping from the exterior \(\mathscr D^e=\mathbb C\setminus\overline{\mathscr D}\) onto the exterior of the unit disk \(\mathbb D^e=\mathbb C\setminus\overline{\mathbb D}\), satisfying \(\phi(\infty)=\infty\) and \(\phi'(\infty)>0\). This mapping is unique by the Riemann mapping theorem and extends \(C^{\varkappa+1,\alpha}\)-smoothly to the boundary \(\partial\mathscr D\) by the Kellogg--Warschawski theorem \cite[Thm.~3.6]{pommerenke2013boundary}. Further details are given in \cref{sec:conf-approx}.

\begin{maintheorem}\label{thm:Main1}
Let \(\varkappa\in\mathbb Z_{\geq 0}\) and \(\alpha\in(0,1)\). Let \(\mathscr D\) be a \(C^{\varkappa+1,\alpha}\)-smooth Jordan domain, and let \(\omega\) be a non-negative \(C^{\varkappa,\alpha}\)-smooth weight on a neighborhood of \(\overline{\mathscr D}\), positive on \(\partial\mathscr D\). Then there exist bounded functions \(\nu_0,\dots,\nu_\varkappa\), holomorphic on \(\mathbb D^e\), which extend \(C^{\varkappa-j,\alpha}\)-smoothly on \(\mathbb T\) and are real-valued at infinity.
Further, for every \(\beta\in(0,\alpha)\) and every closed set \(K\subset\mathscr D^c\), uniformly for \(z\in K\),
\begin{equation}\label{eq:into-asymp-exp1}
P_N(z) = \sqrt{N+1}\phi'(z)\phi^N(z)\exp\Bigl({\textstyle \sum_{j=0}^\varkappa N^{-j} \nu_{j}(\phi(z))}\Bigr) \Bigl(1+\mathcal O\bigl(N^{\varepsilon-\varkappa-\beta}\bigr)\Bigr),
\end{equation}
Here \(\varepsilon=0\) if \(K\cap\partial\mathscr D=\varnothing\), and \(\varepsilon=\frac12\) otherwise.
\end{maintheorem}

 In particular, for \(\varkappa=0\), this result improves Suetin's previous error estimate \eqref{eq:Suetin-asymp}. Moreover, if \(\varkappa+\alpha>1/2\), then the asymptotic expansion \eqref{eq:into-asymp-exp1} remains valid up to the boundary \(\partial\mathscr D\).

The asymptotic approximation \eqref{eq:into-asymp-exp1} of \(P_N\) is valid only in the complement of \(\mathscr D\). Because the boundary and the weight have only finite regularity, the functions \(\phi\) and \(\nu_j\) need not extend holomorphically across their respective boundaries. To overcome this difficulty, we replace them by suitable analytic approximations.

Set \(L_N\coloneqq\lfloor\frac{N}{\log N}\rfloor\) and denote by \(\mathcal F_N\) the \(N\)-th Fourier--Laurent truncation, that is
\[\mathcal F_N[f](z) = \sum_{\abs{k}\leq N}a_k z^k,\qquad f(z) = \sum_{k\in \mathbb Z}a_k z^k.\]
For large \(N\) and every fixed \(A>0\), the \(L_N\)-truncation \(\mathcal F_{L_N}[\phi^{-1}]\) of the inverse conformal mapping \(\phi^{-1}\) remains conformal on  \(\mathbb D_{r_N}^e\), the exterior of the disk of radius \(r_N = e^{-\frac{A\log N}{N}}\). Therefore we get the approximation 
\[\phi_N =(\mathcal F_{L_N}[\phi^{-1}])^{-1}\] of the exterior conformal map \(\phi\), see \cref{sec:conf-approx} for details.

\begin{maintheorem}\label{thm:Main2}
Under the assumptions of \cref{thm:Main1}, fix \(R>0\) and define the region
\[\Omega_{N,R}\coloneqq\left\{z\in\mathbb C\,\left|\,\dist(z,\mathscr D^c)\leq\mfrac{R}{\norm{\phi'}_{L^\infty(\partial\mathscr D)}} \mfrac{\log N }{N}\right.\right\}.\]
For all sufficiently large \(N\), the function \(\phi_N\) is well-defined on \(\Omega_{N,R}\), and
\[\nu_{N,j}\coloneqq\mathcal F_{L_N}[\nu_j],\qquad 0\leq j\leq\varkappa,\]
is well-defined on \(\phi_N(\Omega_{N,R})\). Further, for every \(\beta\in(0,\alpha)\), one has uniformly on \(\Omega_{N,R}\)
\[P_N(z)=\sqrt{N+1}\phi_N'(z)\phi_N^N(z)\exp\Bigl({\textstyle\sum_{j=0}^\varkappa N^{-j}\nu_{N,j}(\phi_N(z))}\Bigr)\Bigl(1+\mathcal O\bigl(N^{1/2+R-\varkappa-\beta}\bigr)\Bigr).\]
\end{maintheorem}

In addition to the asymptotic behavior we also obtain information about the location of the zeros of the orthogonal polynomials. If \(z_N\) is a zero of \(P_N\), then
\[\dist(z_N,\overline{\mathscr D})\longrightarrow 0, \qquad	N\to\infty.\]
In particular, if \(\varkappa+\alpha>1/2\), then all zeros of \(P_N\) lie in \(\overline{\mathscr D}\) for sufficiently large \(N\).  The following corollary gives a more precise description of this behavior.

\begin{maincorollary}\label{coro:main}
Under the assumptions of \cref{thm:Main1}, the following statements hold. If \(\varkappa+\alpha\leq\frac12\), then, for every \(\gamma\in(0,\alpha)\),
\[\max_{P_N(z)=0}\dist(z,\overline{\mathscr D})=\mathcal O(N^{-2\gamma}).\]
If \(\varkappa+\alpha>\frac12\), set \(d_N=\min\{\dist(z,\mathscr D^c)\,|\,P_N(z)=0\}\), then
\[\liminf_{N\to \infty}\frac{d_N\,N}{\log N} \geq \frac{\varkappa+\alpha-1/2}{\norm{\phi'}_{L^\infty(\partial \mathscr D)}}.\]
\end{maincorollary}

\subsection{Related Work}
Closely related to our setting is the asymptotic analysis of planar orthogonal polynomials with exponentially varying weights \cite{Hedenmalm2024softRHP,Hedenmalm2020spectral}, which arise in random matrix theory and two-dimensional Coulomb gases \cite{hedenmalm2017planar,Hedenmalm2013Coulomb,Ameur2011fluctuations,Zabrodin2006Dyson}. At inverse temperature \(2\), these gases form determinantal point processes whose correlation kernels can be expressed in terms of orthogonal polynomials.

For the fixed-weight ensemble associated with \(\omega\), the particles are confined to \(\mathscr D\), producing a hard edge at its boundary. The correlation kernel is
\[
K_N(z,w)=\sqrt{\omega(z)\omega(w)}\sum_{j=0}^{N-1}P_j(z)\overline{P_j(w)},\qquad z,w\in\mathscr D,
\]
where \(P_j\) are the orthogonal polynomials in \(L^2(\mathscr D,\omega)\). Its diagonal \(K_N(z,z)\) gives the particle intensity with respect to \(\d A\).

Boundary scaling limits describe the particle distribution at microscopic distances from \(\partial\mathscr D\). Under mild assumptions, Lubinsky established a universal limit for the rescaled kernel \cite{Lubinsky2010universality}. After suitable rescaling and normalization, this limit is independent of the weight and the domain. Under suitable regularity assumptions, the asymptotic expansion in \cref{thm:Main2} yields a quantitative version of this result, valid on a growing microscopic scaling window. The same limiting kernel appears in recent work on hard-edge universality for varying weights \cite{cronvall2025direct}.

\subsection{Proof Idea and Outline}\label{sec:outline}
We adapt the foliation flow approach of Hedenmalm--Wennman \cite{hedenmalm2017planar,hedenmalm2024riemann} to the present setting. At its core, this approach reduces orthogonality with respect to an area measure to orthogonality along a family of closed curves. Let \(\varphi\colon\mathbb D\to\mathscr D\) be a diffeomorphism and set \(\Gamma_t=\varphi(\mathbb T_t)\), where \(\mathbb T_t\) is the circle of radius \(t\). Then the curves \((\Gamma_t)_t\) foliate the domain \(\mathscr D\) and for suitably regular function \(F\), we get
\begin{equation}\label{eq:Idea1}
\int_{\mathscr D}F(z)\,\omega(z)\,\d\A(z)
=
\mfrac{1}{\pi}\int_0^1\int_{\Gamma_t}
F(\zeta)\,\omega(\zeta)v_\varphi(\zeta)
\abs{\d\zeta}\d t,
\end{equation}
where \(v_\varphi(\zeta)\) is the outward normal velocity of the curve passing through \(\zeta\).
Suppose now that \(F\) is a polynomial of degree \(N\) satisfying
\begin{equation}\label{eq:Idea2}
\int_{\Gamma_t}Q(\zeta)\overline{F(\zeta)}
\,\omega(\zeta)v_\varphi(\zeta)\abs{\d\zeta}=0,
\end{equation}
for every polynomial \(Q\) of degree less than \(N\) and every \(0<t<1\). Applying \eqref{eq:Idea1} to \(Q\overline F\), we see that \(F\) agrees with \(P_N\), up to normalization.

Ideally, we would find a polynomial \(F\) and a foliation satisfying \eqref{eq:Idea2} exactly. Since no general construction is available, we instead seek an approximate solution near the boundary. This is motivated by the fact that the orthogonal polynomials concentrate their \(L^2\)-mass in a shrinking one-sided neighborhood of the boundary as \(N\) grows.
For each \(N\), we construct a holomorphic function \(F_N\) and a foliation \((\Gamma_{N,t})_t\) of a shrinking boundary neighborhood that approximately satisfy \eqref{eq:Idea2}.
These holomorphic approximations \(F_N\) are given in terms of the exterior conformal mapping \(\phi\) and certain Szeg\H{o} functions.
Afterwards, we obtain the asymptotic estimates for \(P_N\) from the approximate orthogonality of \(F_N\).

The main difficulty in adapting these methods to finite regularity is the restriction on the domains of definition. In the real-analytic setting, the exterior conformal map and the Szeg\H{o} functions involved in \(F_N\) extend holomorphically across the boundary, whereas such extensions are generally unavailable under our assumptions. Consequently, the corresponding constructions cannot be applied directly and require suitable modifications.  We overcome this difficulty by replacing these functions with their Fourier--Laurent truncations.

We develop the approximation framework in \cref{sec:analytic-app-frame}. In particular, \cref{sec:Fourier,sec:conf-approx} establish Fourier approximation estimates on shrinking annuli and apply them to construct approximations of the exterior conformal map that extend across \(\partial\mathscr D\). We then introduce \textit{admissible Laurent families} in \cref{sec:AL}. These families satisfy uniform estimates that are needed to adapt the foliation flow method of Hedenmalm--Wennman to the finite regularity setting and play a central role throughout the paper.

The foliation flow construction is carried out in \cref{sec:foliation-technique}. We first introduce the holomorphic approximation \(F_N\) of \(P_N\) and the foliation in \cref{sec:Foliation and Model}. Both are expressed in terms of initially unspecified admissible Laurent families, which are determined in \cref{sec:asymp-flow-decomp} by a recursive algorithm derived from the foliation flow.

We turn to the orthogonal polynomials in \cref{sec:strong-asymptotics}. We first show in \cref{sec:quasi-ortho} that \(F_N\) is approximately normalized and quasi-orthogonal on a shrinking boundary layer. In \cref{sec:faber}, we approximate \(F_N\) by its associated generalized Faber polynomial \(B_N\). These polynomials inherit the approximate normalization and quasi-orthogonality of \(F_N\) and extend both properties to the whole domain \(\mathscr D\). Finally in \cref{sec:asymp expansion}, we compare the polynomials \(B_N\) with \(P_N\) and derive the pointwise asymptotic expansions stated in \cref{thm:Main1,thm:Main2}.

\subsection{Notation and Conventions}\label{sec:notation}
We denote by \(\mathbb D\) the open unit disk and by \(\mathbb T\coloneqq\partial\mathbb D\) the unit circle. More generally, for any positive number \(r>0\) we write \(\mathbb D_r\) and \(\mathbb T_r\) for the disk and circle of radius \(r\), respectively. Furthermore, for a given domain \(D\subset \mathbb C\) we denote by \(D^e\coloneqq \mathbb C\setminus \overline D\) and \(D^c\coloneqq \mathbb C\setminus D\) its exterior and complement, respectively. 

We write \(\d\A\) for the area measure normalized by \(1/\pi\). The space of complex polynomials of degree at most \(N\) is denoted by \(\mathrm{Pol}_N\). Further, we use \(\mathcal O\) and \(\mathrm{o}\) in their usual Landau senses, and write \(f\asymp g\) when \(f=\mathcal O(g)\) and \(g=\mathcal O(f)\).

Unless stated otherwise, all estimates involving \(N\) are understood to hold for sufficiently large \(N\). The constants implicit in the notation \(\mathcal O(\cdot)\) may depend on \(\mathscr D\), \(\omega\), \(\varkappa\), \(\alpha\), and any fixed auxiliary parameters under consideration, but are independent of \(N\).

\section{Analytic Approximation Framework}\label{sec:analytic-app-frame}

\subsection{Hölder Spaces}\label{sec:hoelder}
We denote by \(C^\varkappa(\mathbb T)\) the space of functions \(f\colon \mathbb T\to \mathbb C\) which are \(\varkappa\)-times continuously differentiable with respect to the angular variable. Moreover, \(C^{\varkappa,\alpha}(\mathbb T)\) denotes the space of all \(f\in C^\varkappa(\mathbb T)\) such that the norm
\[\norm{f}_{C^{\varkappa,\alpha}(\mathbb T)}\coloneqq\sum_{s=0}^\varkappa \norm{D_\tau^s f}_{L^\infty(\mathbb T)}+\hnorm{D_\tau^\varkappa f}_{\mathbb T,\alpha}\]
is finite, where
\[\hnorm{D_\tau^\varkappa f}_{\mathbb T,\alpha}\coloneqq\sup_{\substack{\zeta,\eta\in \mathbb T\\\zeta \neq \eta}}\frac{\abs{D_\tau^\varkappa f(\zeta)-D_\tau^\varkappa f(\eta)}}{\abs{\zeta-\eta}^\alpha},\]
and, for \(\zeta=e^{i\theta}\),
\[D_\tau f(\zeta)\coloneqq
\frac{d}{d\theta} f(e^{i\theta}).\]
With this norm, \(C^{\varkappa,\alpha}(\mathbb T)\) is a Banach space. Since the Euclidean and angular distances are equivalent on \(\mathbb T\), the resulting Hölder norm is equivalent to the standard Hölder norms used in the literature.

We use analogous notation for functions defined on an annulus \(\mathcal A\). For \(z=\rho e^{i\theta}\in \mathcal A\), we define
\[D_\tau f(z)\coloneqq\frac{d}{d\theta}f(\rho e^{i\theta}),\qquad D_\rho f(z)\coloneqq\frac{d}{d\rho}f(\rho e^{i\theta}).\]
For \(\varkappa\in \mathbb N\) and \(\alpha\in(0,1)\), we define \(C^{\varkappa,\alpha}(\mathcal A)\) as the space of all functions \(f\colon \mathcal A\to\mathbb C\) for which the derivatives \(D_\rho^mD_\tau^n f\) with \(m+n\leq \varkappa\)
exist, and for which the norm
\[\norm{f}_{C^{\varkappa,\alpha}(\mathcal A)} \coloneqq \sum_{m+n\leq \varkappa} \norm{D_\rho^mD_\tau^n f}_{L^\infty(\mathcal A)} + \sum_{m+n=\varkappa} \hnorm{D_\rho^mD_\tau^n f}_{\cA,\alpha}\]
is finite. Here
\[\hnorm{D_\rho^mD_\tau^n f}_{\cA,\alpha}\coloneqq\sup_{\substack{z,w\in \mathcal A\\ z\neq w}}\frac{\abs{D_\rho^mD_\tau^n f(z)-D_\rho^mD_\tau^n f(w)}}{\abs{z-w}^\alpha}.\]

If \(f\in C^1(\cA)\), then the polar derivatives introduced above are related to the Wirtinger derivatives
\[\partial \coloneqq \mfrac{1}{2}(\partial_x-i\partial _y),\qquad\bar \partial \coloneqq \mfrac{1}{2}(\partial_x+i \partial_y),\]
through the relations
\begin{equation}\label{eq:polar-to-wirtinger}
D_\tau f(z)=iz\,\partial f(z)-i\overline z\,\bar\partial f(z),\qquad D_\rho f(z)=\frac{z}{\abs z}\,\partial f(z)+\frac{\overline z}{\abs z}\,\bar\partial f(z).
\end{equation}
Standard product estimates for Hölder spaces show that the norm defined above is equivalent to the corresponding Hölder norm obtained from Wirtinger derivatives,
\[\sum_{m+n\leq \varkappa}\norm{\partial^m\bar\partial^n f}_{L^\infty(\mathcal A)}+\sum_{m+n=\varkappa}\hnorm{\partial^m\bar\partial^n f}_{\cA,\alpha}.\]
This equivalence is uniform for annuli contained in a fixed annular neighborhood of \(\mathbb T\).
Unlike the polar formulation, the Wirtinger formulation extends naturally without change to arbitrary domains \(U\subset\mathbb C\). In this case we can also write \(\tnorm{f}_{C^{\varkappa,\alpha}(U)}\).

The spaces \(C^{\varkappa,\alpha}(\mathbb T)\) and \(C^{\varkappa,\alpha}(\mathcal A)\) are Banach spaces, and the equivalent norm formulations described above may be used interchangeably.  Finally, these Hölder spaces are Banach algebras under pointwise multiplication. More precisely, if \(X\) denotes either \(\mathbb T\) or \(\mathcal A\), then there exists a constant \(C_{\varkappa,\alpha}>0\) such that, for any \(f,g\in C^{\varkappa,\alpha}(X)\), one has \(fg\in C^{\varkappa,\alpha}(X)\) and \(\tnorm{fg}_{C^{\varkappa,\alpha}(X)}\leq C_{\varkappa,\alpha}\tnorm{f}_{C^{\varkappa,\alpha}(X)}\tnorm{g}_{C^{\varkappa,\alpha}(X)}\).

\subsection{Fourier Approximation on Shrinking Annuli}\label{sec:Fourier}
Given a function \(f\in C^{\varkappa,\alpha}(\mathbb T)\), we define its Fourier coefficients \(c_k(f)\) by
\[c_k(f) \coloneqq \mfrac{1}{2\pi}\int_\mathbb T f(\zeta) \zeta^{-k} \abs{\d \zeta}.\]
Further, we define the \(N\)-th Fourier truncation operator \(\mathcal F_N\) by
\[\mathcal F_N[f](\zeta) \coloneqq \sum_{\abs{k}\leq N} c_k(f) \zeta^k.\]
Although \(\mathcal F_N[f]\) is initially defined on the unit circle, the same formula defines a Laurent polynomial, and hence a holomorphic function on \(\mathbb C\setminus\{0\}\). If only Fourier coefficients corresponding to non-negative or only non-positive indices occur, then this Laurent polynomial extends holomorphically to \(0\) or \(\infty\), respectively. In the same way, it is reasonable to consider the derivative of \(\mathcal F_N[f]\) as
\[\mathcal F_N[f]'(z) = \sum_{\abs{k}\leq N}kc_k(f) z^{k-1},\quad z\in \mathbb C\setminus\{0\},\]
when \(\mathcal F_N[f]\) is regarded as a holomorphic function. Note that, in order to view \(\mathcal F_N[f]\) as a holomorphic function it is important to interpret the negatively indexed part of \(\mathcal F_N[f](z)\) as negative powers of \(z\) (i.e. \(z^{-k}\) for \(k>0\)) and not as powers of the conjugate (i.e. \(\bar z^k\)). Although both definitions coincide on the circle, the latter will not describe a holomorphic function.

The following result collects the Fourier approximation estimates used throughout the paper. It provides control of  the uniform approximation error of the Fourier truncations \(\mathcal F_N[f]\) and their Hölder norms. Although these estimates are classical, the precise formulation required here is not commonly stated as a single result, so we include a proof for completeness in \cref{sec:App-add-proofs}.
\begin{theorem}\label{theo:Bernstein0}
Let \(f\in C^{\varkappa,\alpha}(\mathbb T)\) for \(\varkappa\in\mathbb Z_{\geq 0}\) and \(\alpha\in (0,1)\). Then there exists  a constant \(c_\alpha>0\) independent of \(f\) and \(N\) such that 
\begin{equation}\label{eq:theo-Bernstein-uni1}
\norm{f-\mathcal F_N[f]}_{L^\infty(\mathbb T)}\leq c_\alpha \norm{f}_{C^{\varkappa,\alpha}(\mathbb T)}\mfrac{\log N}{N^{\varkappa+\alpha}} ,\qquad N\geq 2.
\end{equation}
Further, for every \(\beta\in(0,\alpha)\) there exists a constant \(c_{\alpha,\beta}>0\), independent of \(f\) and \(N\), such that
\begin{equation}\label{eq:theo-Bernstein-uni2}
\norm{\mathcal F_N[f]}_{C^{\varkappa,\beta}(\mathbb T)}\leq c_{\alpha,\beta} \norm{f}_{C^{\varkappa,\alpha}(\mathbb T)}.
\end{equation}
\end{theorem}

Truncating the Fourier series at order \(N\) is not suitable for our purpose. For each fixed \(A>0\), we need approximants \(f_N\) whose Hölder norms remain uniformly bounded on the shrinking annuli \(\cA_N=\mathbb D_{e^{\delta_N}}\setminus\overline{\mathbb D}_{e^{-\delta_N}}\), where \(\delta_N=\frac{A\log N}{N}\). More precisely, we require \(\tnorm{f_N}_{C^{\varkappa,\beta}(\cA_N)}=\mathcal O(1)\) for every \(\beta\in(0,\alpha)\).

As suggested by \cref{lem:Fourier-Master}, we therefore truncate at \(L_N=\lfloor \frac{N}{\log N}\rfloor\), which is inversely proportional to \(\delta_N\). This gives us control of the truncation away from the unit circle while preserving the algebraic approximation rate on the circle, up to an arbitrarily small loss in the exponent.
Indeed, let \(f\in C^{\varkappa,\alpha}(\mathbb T)\), then by \cref{theo:Bernstein0} we have for every \(\beta\in(0,\alpha)\)
\[\norm{f-\mathcal F_{L_N}[f]}_{L^\infty(\mathbb T)}\leq c_\alpha \norm{f}_{C^{\varkappa,\alpha}(\mathbb T)} \frac{\log L_N}{L_N^{\varkappa+\alpha}}=\mathcal O\Bigl(\norm{f}_{C^{\varkappa,\alpha}(\mathbb T)} N^{-\varkappa-\beta}\Bigr).\]

\begin{lemma}\label{lem:Fourier-Master}
Let \(\varkappa \in \mathbb Z_{\geq 0}\) and \(\beta\in(0,1)\), then there exists a constant \(c_\varkappa>0\), such that for any Laurent polynomial \(f\) it holds
\[\norm{f}_{C^{\varkappa,\beta}(\cA)}\leq c_\varkappa\, e^{2\delta \deg f} \norm{f}_{C^{\varkappa,\beta}(\mathbb T)},\]
where \(\cA = \mathbb D_{e^\delta}\setminus \overline{\mathbb D_{e^{-\delta}}}\) and \(\delta\in (0,1)\).
\end{lemma}

Before proving \cref{lem:Fourier-Master}, we recall Bernstein's inequality
for Laurent polynomials. Let \(f(z)=\sum_{\abs{k}\leq M}a_k z^k\) be a Laurent polynomial, then we define the degree of \(f\) by \(\deg f\coloneqq \max\{\tabs{k}\,|\, a_k \neq 0\}\). Then the Bernstein inequality reads as
\begin{equation}\label{eq:bernstein-ineq1}
\norm{D_\tau f}_{L^\infty(\mathbb T)}
\leq
\deg f \norm{f}_{L^\infty(\mathbb T)}.
\end{equation}
For a modern reference, see \cite[Thm.~1.22]{plonka2018numerical}. 
In addition, we need a Hölder variant of Bernstein's inequality.  For every \(\gamma\in(0,1)\) it  holds 
\begin{equation}\label{eq:bernstein-ineq3}
\norm{D_\tau f}_{L^\infty(\mathbb T)}\leq \pi (\deg f)^{1-\gamma}\hnorm{f}_{\mathbb T,\gamma}.
\end{equation}
This inequality follows from the Bernstein--Nikolsky--Stechkin inequality, see Stechkin \cite{Stechkin1948} and for a modern reference \cite[p.~131, Eq.~(12.1)]{DeVoreLorentz1993constapprox}. The particular case used here, already goes back to Riesz \cite{Riesz1914}.
Using \eqref{eq:polar-to-wirtinger}, one may equivalently replace
\(\tnorm{D_\tau f}_{L^\infty(\mathbb T)}\) by
\(\tnorm{f'}_{L^\infty(\mathbb T)}\) in \eqref{eq:bernstein-ineq1} and \eqref{eq:bernstein-ineq3}.

\begin{proof}
It suffices to prove the result for \(\varkappa=0\). The general case follows by applying the case \(\varkappa=0\) to \(f^{(j)}\), for \(0\le j\le\varkappa\), where the additional factors \(e^{2j\delta}\le e^{2\varkappa}\) are absorbed into the constant \(c_\varkappa\). The supremum term in the Hölder norm is controlled by the maximum principle. Considering \(z^{\deg f}f(z)\) we obtain a polynomial of degree \(2\deg f\), and using the maximum principle yields
\[e^{-\delta \deg f}\norm{f}_{L^\infty(\mathbb T_{e^{-\delta})}}=\norm{\smash{z^{\deg f} f(z)}}_{L^\infty(\mathbb T_{e^{-\delta}})}\leq \norm{f}_{L^\infty(\mathbb T)}.\]
Similarly, we obtain an upper bound for the exterior boundary of \(\cA\), and hence
\begin{equation}\label{eq:prf-lem-FM1}
\norm{f}_{L^\infty(\cA)}\leq e^{\delta\deg f}\norm{f}_{L^\infty(\mathbb T)}.
\end{equation}
To estimate the Hölder seminorm, we treat the radial and tangential contributions separately. We observe for distinct \(t\zeta,s\eta\in \cA\)  with \(\zeta,\eta\in \mathbb T\) and \(s,t\in(e^{-\delta},e^\delta)\) the relation \(\abs{t\zeta-s\eta}^2=(t-s)^2+ts\abs{\zeta-\eta}^2\), and hence
\[\frac{\abs{f(t\zeta)-f(s\eta)}}{\abs{t\zeta-s\eta}^\beta}\leq \frac{\abs{f(t\zeta)-f(t\eta)}+ \abs{f(t\eta)-f(s\eta)}}{((t-s)^2+ts\abs{\zeta-\eta}^2)^{\beta/2}}  \leq\frac{\abs{f(t\zeta)-f(t\eta)}}{(ts)^{\beta/2}\abs{\zeta-\eta}^\beta}+\frac{\abs{f(t\eta)-f(s\eta)}}{\abs{t-s}^\beta}.\]
Note that we may assume here that \(t\neq s\) and \(\eta \neq \zeta\), otherwise the corresponding fraction above vanishes.
To simplify notation we write \(M=\deg f\). Using \eqref{eq:prf-lem-FM1}, \eqref{eq:bernstein-ineq3} and \(1-e^{-x}\leq x\) for \(x\geq0\), we obtain
\begin{align*}
\frac{\abs{f(t\eta)-f(s\eta)}}{\abs{t-s}^\beta}&\leq \norm{f'}_{L^\infty(\cA)}\abs{t-s}^{1-\beta}\\
&\leq e^{\delta (1+M)}\norm{f'}_{L^\infty(\mathbb T)}\abs{t-s}^{1-\beta}\\
&\leq c e^{\delta M} (e^{\delta}-e^{-\delta})^{1-\beta} M^{1-\beta} \hnorm{f}_{\mathbb T,\beta}\\
&\leq c' e^{\delta M}(\delta M)^{1-\beta}\hnorm{f}_{\mathbb T,\beta}\\
&\leq  c' e^{\delta M}(1+\delta M)\hnorm{f}_{\mathbb T,\beta}\\
&\leq c'e^{2\delta M}\hnorm{f}_{\mathbb T,\beta},
\end{align*}
for some \(c,c'>0\). For the other term we get, analogously to \eqref{eq:prf-lem-FM1}, 
\begin{align*}
\abs{f(t\zeta)-f(t\eta)}&\leq \norm{f(\cdot \zeta)-f(\cdot \eta)}_{L^\infty (\cA)}\\
&\leq e^{\delta M} \norm{f(\cdot \zeta)-f(\cdot \eta)}_{L^\infty (\mathbb T)}\leq e^{\delta M} \hnorm{f}_{\mathbb T,\beta}\abs{\zeta-\eta}^\beta.
\end{align*}
Hence,
\begin{align*}
\frac{\abs{f(t\zeta)-f(t\eta)}}{(ts)^{\beta/2}\abs{\zeta-\eta}^\beta}&\leq e^{\beta \delta + \delta M}\hnorm{f}_{\mathbb T,\beta}.
\end{align*}
Combining these estimates gives the wanted result.
\end{proof}

These features are not specific to the Fourier truncations themselves. Any family of Laurent polynomials of degree \(\mathcal O(L_N)\) which approximates \(f\) on \(\mathbb T\) with the same accuracy as the Fourier truncation inherits the corresponding uniform boundedness on the annulus.

\begin{theorem}\label{theo:Fourier-Master}
Let \(\varkappa\in\mathbb Z_{\geq 0}\), \(\alpha\in (0,1)\) and \(f\in C^{\varkappa,\alpha}(\mathbb T)\). 
Let \(f_N\) be a sequence of Laurent polynomials with degree at most \(\mathcal O( L_N) \) such that \(\tnorm{f_N-f}_{L^\infty(\mathbb T)}=\mathcal O(N^{-\varkappa-\gamma})\) for all \(\gamma\in(0,\alpha)\). Let \(\mathcal A_N= \mathbb D_{e^{\delta_N}}\setminus \overline{ {\mathbb D}}_{e^{-\delta_N}}\) with \(\delta_N=\frac{A\log N}{N}\) for fixed \(A>0\). Then \(\tnorm{f_N-\mathcal F_{L_N}[f]}_{L^\infty(\mathcal A_N)}=\mathcal O(N^{-\varkappa-\beta})\)  and \(\tnorm{f_N}_{C^{\varkappa,\beta}(\mathcal A_N)  }=\mathcal O(1)\) for any \(\beta\in (0,\alpha)\).
\end{theorem}

\begin{proof} The first statement follows directly from the maximum principle. Similarly to \eqref{eq:prf-lem-FM1} we have
\[\norm{f_N-\mathcal F_{L_N}[f]}_{L^\infty(\cA_N)}\leq e^{c L_N \delta _N} \norm{f_N-\mathcal F_{L_N}[f]}_{L^\infty(\mathbb T)}=\mathcal O (N^{-\varkappa-\beta}),\]
for some \(c>0\) and any \(\beta \in (0,\alpha)\). Here we have used that \(L_N\delta_N \leq A\). To show the uniform Hölder boundedness, we set \(M_N=\max(\deg f_N,L_N)\) and observe
\begin{align*}
\norm{D_\tau^j f_N-D_\tau^j f}_{L^\infty(\mathbb T)}&\leq \norm{D_\tau^j f_N-D_\tau^j \mathcal F_{L_N}[f]}_{L^\infty(\mathbb T)}+\norm{D_\tau^j \mathcal F_{L_N}[f]-D_\tau^j f}_{L^\infty(\mathbb T)}\\
&\leq M_N^j\norm{f_N-\mathcal F_{L_N}[f]}_{L^\infty(\mathbb T)}+\norm{\mathcal F_{L_N}[D_\tau^jf]-D_\tau ^jf}_{L^\infty(\mathbb T)}
\end{align*}
where we have used the Bernstein inequality \eqref{eq:bernstein-ineq1}, and the commutation of \(\mathcal F_{L_N}\) and \(D_\tau\). Now using \cref{theo:Bernstein0} and uniform convergence assumption of \(f_N\), we get for any \(\beta\in(0,\alpha)\) and \(j=0,\dots,\varkappa\)
\[\norm{D_\tau^j f_N-D_\tau^j f}_{L^\infty(\mathbb T)}\leq \mathcal O(N^{j-\varkappa-\beta})\]
Hence, using \cref{lem:Fourier-Master} yields
\[\norm{f_N}_{C^{\varkappa,\beta}(\cA_N)}\leq c_\varkappa e^{2A}\norm{f_N}_{C^{\varkappa,\beta}(\mathbb T)}\leq c_\varkappa e^{2A}\hnorm{D_\tau^\varkappa f_N -D_\tau^\varkappa \mathcal F_{L_N}[f]}_{\mathbb T,\beta}+\mathcal O(1).\]
Let \(g\) be an arbitrary Laurent polynomial, we observe for \(\zeta,\eta\in\mathbb T\)
\begin{align*}
\abs{g(\zeta)-g(\eta)}&\leq\min\left(2\norm{g}_{L^\infty(\mathbb T)},\mfrac{\pi}{2}\norm{g'}_{L^\infty(\mathbb T)}\abs{\zeta-\eta}\right)\\&
\leq 2^{1-2\beta}\pi^\beta (\deg g)^\beta\norm{g}_{L^\infty(\mathbb T)}\abs{\zeta-\eta}^\beta,
\end{align*}
where we have used \(\min(a,b)\leq a^{1-\beta} b^{\beta}\) and \eqref{eq:bernstein-ineq1}. We conclude
\[\hnorm{g}_{\mathbb T,\beta}\leq c_\beta (\deg g)^\beta \norm{g}_{L^\infty(\mathbb T)}.\]
 Using this estimate together with the Bernstein inequality \eqref{eq:bernstein-ineq1} we get
\begin{align*}
e^A\hnorm{D_\tau^\varkappa f_N -D_\tau^\varkappa \mathcal F_{L_N}[f]}_{\mathbb T,\beta}&\leq \pi M_N^\beta e^A\norm{D_\tau^\varkappa f_N -D_\tau^\varkappa \mathcal F_{L_N}[f]}_{L^\infty(\mathbb T)}\\
&\leq \pi M_N^{\varkappa + \beta} e^A\norm{f_N-\mathcal F_{L_N}[f]}_{L^\infty(\mathbb T)}\\
&=\hbox{o}(1),
\end{align*}
where we used that \(\tnorm{f_N-\mathcal F_{L_N}[f]}_{L^\infty(\mathbb T)}=\mathcal O(N^{-\varkappa-\gamma})\) for any \(\gamma \in(\beta,\alpha)\).
\end{proof}

\begin{corollary}\label{coro:uni-conv-Fourier}
Let \(f\in C^{\varkappa,\alpha}(\mathbb T)\) and \(A>0\) be fixed and set \(f_{N,t}(\zeta)= \mathcal F_{L_N}[f](e^{-t}\zeta)\). Then for any \(\beta\in(0,\alpha)\) we have \(\tnorm{f-f_{N,0}}_{L^\infty(\mathbb T)}=\mathcal O(N^{-\varkappa-\beta})\) as well as
\begin{equation}\label{eq:coro-uni-conv-Fourier1}
\sup_{\abs t\leq \frac{A\log N}{N}}\norm{f-f_{N,t}}_{L^\infty(\mathbb T)}= \begin{cases}
\mathcal O\Bigl(\frac{\log N}{N}\Bigr),&\varkappa\geq 1,\\
\mathcal O \bigl(N^{-\beta}\bigr),&\varkappa=0.
\end{cases}
\end{equation} 
\end{corollary}

\begin{proof}
The uniform estimate for the special case \(t=0\) follows directly from \cref{theo:Bernstein0}. For the second estimate we decompose
\begin{align*}
\abs{f(\zeta)-f_{N,t}(\zeta)}&\leq \abs{f(\zeta)-\mathcal F_{L_N}[f](\zeta)}+\abs{\mathcal F_{L_N}[f](\zeta)-\mathcal F_{L_N}[f](e^{-t}\zeta)}\\
&=\abs{\mathcal F_{L_N}[f](\zeta)-\mathcal F_{L_N}[f](e^{-t}\zeta)}+\mathcal O(N^{-\varkappa-\beta}).
\end{align*}
If \(\varkappa=0\), then \cref{theo:Fourier-Master} yields
\[\abs{\mathcal F_{L_N}[f](\zeta)-\mathcal F_{L_N}[f](e^{-t}\zeta)}\leq \hnorm{\mathcal F_{L_N}[f]}_{\cA_N,\gamma} \abs{\zeta-e^{-t}\zeta}^\gamma=\mathcal O(N^{-\beta}),\]
for any \(\gamma \in(\beta,\alpha)\) and \(\cA_N=\mathbb D_{e^{\delta_N}}\setminus \overline{\mathbb D}_{e^{-\delta_N}}\) with \(\delta_N=\frac{A\log N}{N}\). For \(\varkappa\geq 1\), we have
\[\abs{\mathcal F_{L_N}[f](\zeta)-\mathcal F_{L_N}[f](e^{-t}\zeta)}\leq \norm{\mathcal F_{L_N}[f]}_{C^{1,\beta} (\cA_N)} \abs{\zeta-e^{-t}\zeta}=\mathcal O\Bigl(\mfrac{\log N}{N}\Bigr),\]
again by \cref{theo:Fourier-Master}.
\end{proof}

\subsection{The Exterior Herglotz Transform}\label{sec:herglotz}
In this section, we recall the exterior Herglotz transform, which provides a convenient way to construct holomorphic functions on \(\mathscr D^e\) with prescribed real boundary values. For a function \(f\in C(\mathbb T)\), the exterior Herglotz transform \(\Herg[f]\) can be interpreted, up to a constant factor, as the projection of \(f\) onto its negatively indexed Fourier coefficients.

The Herglotz transform can also be studied on more general function spaces, such as \(L^p(\mathbb T)\), and is closely related to the theory of Hardy spaces. In particular, in the exterior setting it is essentially the analytic projection onto the exterior Hardy space. For our purposes, however, it will be sufficient to work with Hölder continuous boundary data.

Let \(f\in C^{0,\alpha}(\mathbb T)\). We define its exterior Herglotz transform by
\[\Herg[f](z)\coloneqq\mfrac{1}{2\pi}\int_{\mathbb T}f(\zeta)\frac{z+\zeta}{z-\zeta}\abs{\d\zeta},\qquad z\in\mathbb D^e.\]
This defines a holomorphic function on \(\mathbb D^e\). In fact, for \(z\in\mathbb D^e\), we have
\begin{equation}\label{eq:herglotz0}
\Herg[f](z) =  c_0(f)+2\sum_{k=1}^\infty c_{-k}(f)z^{-k},
\end{equation}
where \(c_k(f)\) denotes the \(k\)-th Fourier coefficient of \(f\).

Let \(f\in C^{\varkappa,\alpha}(\mathbb T)\) for \(\varkappa\in \mathbb Z_{\geq 0}\) and \(\alpha\in(0,1)\), then the Herglotz transform \(\Herg[f]\) and its derivatives \(\Herg[f]^{(j)}\), for \(j=1,\dots,\varkappa\), extend continuously to \(\mathbb T\). Moreover, the boundary restriction of \(\Herg[f]\) belongs to \(C^{\varkappa,\alpha}\) and satisfies
\begin{equation}\label{eq:herglotz-estimate0}
\norm{\smash{\Herg[f]}}_{C^{\varkappa,\alpha}(\mathbb T)} \leq C_{\varkappa,\alpha} \norm{f}_{C^{\varkappa,\alpha}(\mathbb T)},
\end{equation}
for some constant \(C_{\varkappa,\alpha}>0\) independent of \(f\). Although this result is classical, we could not find the precise estimate \eqref{eq:herglotz-estimate0} in the literature. We therefore include a proof in \cref{app:herglotz} for completeness.

As a direct consequence of \eqref{eq:herglotz0} together with the continuous extension to \(\mathbb T\), we obtain for real-valued \(f\)
\[\Re \Herg[f](\zeta)=c_0(f)+\sum_{k=1}^\infty\Bigl(c_{-k}(f)\zeta^{-k}+c_k(f)\overline \zeta^{-k}\Bigr)= f(\zeta),\qquad \zeta\in\mathbb T,\]
where we have used that \(c_{-k}(f)=\overline{c_k(f)}\). In other words, the real part of the Herglotz transform recovers \(f\) on the boundary.

In addition to \eqref{eq:herglotz-estimate0}, we will need a bound purely in terms of the supremum norm. Such an estimate is not available in general, but it does hold for Laurent polynomials, up to a logarithmic loss. Let \(P\) be a Laurent polynomial of degree \(M\), then
\begin{equation}\label{eq:herglotz-estimate1}
\norm{\Herg[P]}_{L^\infty(\mathbb T)} \leq c\log(M+2)\norm{P}_{L^\infty(\mathbb T)},
\end{equation}
where \(c>0\) is independent of \(P\) and \(M\). To observe \eqref{eq:herglotz-estimate1} we  express \(\Herg[P]\) as a convolution with the one-sided Dirichlet-type kernel \(K_M^-(\zeta)=1+2\sum_{k=1}^M\zeta^{-k}\). Indeed,
\[\Herg[P](\zeta)  = \mfrac{1}{2\pi}\int_{\mathbb T}P(\zeta\bar \eta)K_M^-(\eta)\abs{\d\eta}.\]
Consequently,
\[\norm{\Herg[P]}_{L^\infty(\mathbb T)}\leq\norm{P}_{L^\infty(\mathbb T)}\norm{\smash{K_M^-}}_{L^1(\mathbb T)}.\]
Using standard estimates for Dirichlet-type kernels we have \(\tnorm{K_M^-}_{L^1(\mathbb T)}=\mathcal O(\log(M+2))\).

\subsection{Conformal Approximation}\label{sec:conf-approx}
In the context of this article, we say that  a function \(\phi\) is a conformal mapping if \(\phi\) is biholomorphic. In other words, \(\phi\) is holomorphic and bijective. Of special interest to us are exterior conformal mappings. Given a Jordan domain \(\mathscr D\), the associated exterior conformal mapping is the unique conformal mapping \(\phi\colon \mathscr D^e\to\mathbb D^e\) , satisfying \(\phi(\infty)=\infty\) and \(\phi'(\infty)>0\). Here \(\mathbb D^e = \mathbb C \setminus \overline{\mathbb D}\) and \(\mathscr D^e=\mathbb C\setminus \overline{\mathscr D}\) denote the exterior of \(\mathbb D\) and \(\mathscr D\), respectively. The existence and uniqueness of \(\phi\) follow from the Riemann mapping theorem. The boundary behavior of such conformal mappings is strongly related to the regularity of \(\mathscr D\). In particular, if \(\partial \mathscr D\) is \(C^{k,\alpha}\)-smooth, with \(k\geq 1\), then the theorem of Kellogg--Warschawski \cite[Thm.~3.5,3.6]{pommerenke2013boundary} guarantees that \(\phi^{(j)}\) for \(j=0,\dots,k\) extends continuously onto the boundary of \(\mathscr D\) and that \(\phi'\) never vanishes on \(\mathscr D^c\). Thus, by continuous extension to the boundary, we may regard \(\phi\) as a map \(\phi\colon \mathscr D^c\to \mathbb D^c\). In particular, the theorem of Carathéodory \cite[Thm. 2.6]{pommerenke2013boundary}, ensures that the continuous extension remains injective.

The main objective of this section is to construct a suitable approximation \(\phi_N\) of \(\phi\). Let  \(\psi=\phi^{-1}\colon \mathbb D^c\to \mathscr D^c\) denote the inverse exterior conformal map. As before, the boundary restriction of \(\psi\) belongs to \( C^{\varkappa+1,\alpha}(\mathbb T)\). We define now
\begin{equation}\label{eq:def-psi_N}
\psi_N\coloneqq \mathcal F_{L_N} [\psi]=\mathcal F_{L_N} [\phi^{-1}],
\end{equation}
where \(L_N=\lfloor \frac{N}{\log N}\rfloor\).
As we will observe in \cref{theo:conf-inverse}, the mappings \(\psi_N\) are invertible on  \(\mathbb D^e_{r_N}\) for \(r_N = e^{-\delta_N}\) and \(\delta_N=\frac{A\log N}{N}\) with \(A>0\). Here \(\mathbb D_{r_N}^e\) denotes the exterior of the disk of radius \(r_N\). This leads to the natural definition
\begin{equation}\label{eq:def-phi_N}
\phi_N\coloneqq \psi_N^{-1}= (\mathcal F_{L_N}[\phi^{-1}])^{-1}
\end{equation}
Whenever \(\psi_N\) is injective on \(\mathbb D_\rho^c\), the complement of the disk of radius \(\rho>0\), we define the regions
\begin{equation}\label{eq:def-D_(N,r)^c}
\mathscr D_{N,\rho}^c\coloneqq \psi_N(\mathbb D_{\rho}^c),\qquad \mathscr D_{N,\rho} = \mathrm {Int}(\psi_N(\mathbb T_\rho)),
\end{equation}
where \(\mathrm{Int}(\psi_N(\mathbb T_\rho))\) denotes the interior of the closed curve \(\psi_N(\mathbb T_\rho)\). In particular, by \cref{theo:conf-inverse}, these definitions apply for \(\rho \geq e^{-\delta_N}\) and all sufficiently large \(N\).
These sets can be seen as the analogue of \(\mathbb D_\rho^e\) and \(\mathbb D_\rho\) after changing the coordinates according to \(\psi_N\).
The remainder of this section is devoted to proving the required invertibility of \(\psi_N\) and to establishing that the inverse maps \(\phi_N\) approximate \(\phi\) with the desired accuracy.

The choice of the \(L_N\)-th rather than the \(N\)-th Fourier truncation of \(\psi\) may seem unnatural at first. As in \cref{sec:Fourier}, this choice is motivated by the need for uniform bounds on annular neighborhoods of \(\mathbb T\) of width \(\frac{A\log N}{N}\), for any fixed \(A>0\). With the \(N\)-th truncation, such estimates are only available when \(A<\varkappa+\alpha-1\).

\begin{lemma}\label{lem:conf-inject}
Let \(f\)  be holomorphic in a neighborhood of \(\mathbb D^c\) and assume that 
\begin{enumerate}
\item[(i)] \(f\) is injective on \(\mathbb T\) and \(t\mapsto f(e^{it})\) is positively oriented,
\item[(ii)] \(f(z)=az +\mathcal O(1)\) as \(z\to \infty\) with \(a\neq 0\).
\end{enumerate}
Then \(f\) is injective on \(\mathbb D^c\) and \(f(\mathbb D ^e)\) is the unbounded component of \(\mathbb C\setminus f(\mathbb T)\).
\end{lemma}

The result presented in \cref{lem:conf-inject}, is a standard injectivity result and may be viewed as the exterior counterpart of \cite[Lem.~1.1]{pommerenke1975univalent}. Its proof follows from the same argument-principle computation, with the simple pole at infinity accounting for the orientation assumption.

\begin{theorem}\label{theo:conf-inverse}
Let \(\varkappa\in \mathbb Z_{\geq 0}\), \(\alpha\in(0,1)\), and let \(\mathscr D\) be a \(C^{\varkappa+1,\alpha}\)-smooth Jordan domain. Let \(\psi\colon \mathbb D^e\to \mathscr D^e\) be the inverse exterior conformal map, and set \(\psi_N\coloneqq \mathcal F_{L_N}[\psi]\). For fixed \(A>0\), define \(r_N=e^{-\frac{A\log N}{N}}\). Then, for all sufficiently large \(N\), \(\psi_N\) extends conformally to \(\mathbb D_{r_N}^e\). Hence, with \(\mathscr D_{N,r_N}^e\coloneqq \psi_N(\mathbb D_{r_N}^e)\), the map \(\psi_N\colon \mathbb D_{r_N}^e\to \mathscr D_{N,r_N}^e\) is conformal. Furthermore, for every \(\beta\in(0,\alpha)\),
\[\norm{\psi_N-\psi}_{L^\infty(\mathbb D^c)}= \mathcal O(N^{-\varkappa-1-\beta}).\]
\end{theorem}

\begin{proof}
Set \(\theta_N(z)=\psi_N(r_Nz)\). By the Kellogg–Warschawski theorem, \(\psi\) extends \(C^{\varkappa+1,\alpha}\)-smoothly up to \(\mathbb T\). Hence, \cref{coro:uni-conv-Fourier} implies that \(\theta_N\to\psi\) and \(\theta_N'\to\psi'\) uniformly on \(\mathbb T\). To apply \cref{lem:conf-inject}, it remains to show that, for sufficiently large \(N\), the restriction of \(\theta_N\) to \(\mathbb T\) is injective and the curve \(\theta_N(\mathbb T)\) is positively oriented. By the uniform convergence of \(\theta_N\) and \(\theta_N'\) on \(\mathbb T\), a classical winding number argument shows that \(\theta_N(\mathbb T)\) has the same orientation as \(\psi(\mathbb T)\). Since the real Jacobian of a conformal map is positive, \(\psi(\mathbb T)\) is positively oriented, and therefore so is \(\theta_N(\mathbb T)\) for sufficiently large \(N\). It remains to prove the injectivity of \(\theta_N\) on \(\mathbb T\).

We will argue by contradiction. We may assume that \(\theta_N\) is not injective for any \(N\in\mathbb N\), otherwise we have to pick an appropriate subsequence. For each \(N\) we pick \(\zeta_{1,N},\zeta_{2,N}\in \mathbb T\) with \(\zeta_{1,N}\neq \zeta_{2,N}\) and \(\theta_N(\zeta_{1,N})=\theta_N(\zeta_{2,N})\). We may further assume that the sequences \(\zeta_{j,N}\) are converging towards some limits \(\zeta_j\in \mathbb T\), for \(j=1,2\), otherwise we can pick again suitable subsequences.
Due to the uniform convergence \(\theta_N\to \psi\) on \(\mathbb T\) we have
\[0= \lim_{N\to \infty}\abs{\theta_{N}(\zeta_{1,N})-\theta_{N}(\zeta_{2,N})}=\abs{\psi(\zeta_1)-\psi(\zeta_2)}.\]
Since \(\psi\) is injective on \(\mathbb T\), we conclude that the limits \(\zeta_{j}\) coincide. 

Let \(\zeta^* = \zeta_j\) and take \(\tau\in\mathbb R\) such that \(e^{i\tau}=\zeta^*\).
 Now we define \(u_N(t)= \Re \theta_N(e^{it})\) and \(v_N(t)= \Im \theta_N(e^{it})\). Then we can find real numbers \(p_N,q_N\in (\tau-\pi,\tau+\pi)\) such that \(e^{ip_N}\), \(e^{iq_N}\) lie on the shorter arc connecting \(\zeta_{1,N}\) and \(\zeta_{2,N}\), as well as  \(u_N'(p_N)=0\)  and \(v_N'(q_N)=0\). The existence of such points is ensured by the mean value theorem. Since \(e^{ip_N}\) and \(e^{iq_N}\) lie between \(\zeta_{1,N}\) and \(\zeta_{2,N}\) we obtain that \(p_N\) and \(q_N\) both converge to \(\tau\). We conclude that \(h_N(t)\coloneqq \theta_N(e^{it})\) satisfies \(\lim_{N\to\infty}h_N'(\tau)=0\), because of the uniform convergence of \(h'_N\to h'\). Hence,
\[ie^{i\tau}\psi'(\zeta^*)=\lim_{N\to\infty} i e^{i\tau} \theta_N'(e^{i\tau})=\lim_{N\to \infty} h_N'(\tau)=0.\]
Consequently \(\psi'(\zeta^*)=0\). However, this contradicts \(\psi'\neq 0\) which is guaranteed due to the \(C^{\varkappa+1,\alpha}\)-smooth boundary of \(\mathscr D\). We conclude \(\theta_N\) is injective on \(\mathbb D^c\) by \cref{lem:conf-inject}. 

Lastly, the  uniform convergence of \(\psi_N\to \psi\) on \(\mathbb D^c\) results immediately from the maximum principle and  \cref{theo:Bernstein0}.
\end{proof}

Before stating the counterpart to \cref{theo:conf-inverse} for the approximation of \(\phi\) by \(\phi_N\), we first establish the inclusion \(\mathscr D^c\subset\mathscr D_{N,r_N}^c\), where \(r_N=e^{-\delta_N}\) and \(\delta_N=\frac{A\log N}N\). This ensures that \(\phi_N\) is defined on \(\mathscr D^c\). The following lemma gives the inclusion needed here, together with a more precise description that will be used later in \cref{sec:asymp expansion}.

\begin{lemma}\label{lem:min-curve}
Let \(\varkappa\in \mathbb Z_{\geq 0}\), \(\alpha\in(0,1)\) and \(\mathscr D\) be a \(C^{\varkappa+1,\alpha}\)-smooth Jordan domain.  Then for every \(R> 0\) and \(A>R\) it holds for sufficiently large \(N\)
\begin{equation}\label{eq:lem-min-curve}
\left\{ z\left|\,\dist(z,\mathscr D^c) \leq \mfrac{R}{\norm{\phi'}_{L^\infty(\partial \mathscr D)}}\mfrac{\log 	N}{N} \right\}\right.\subset \mathscr D_{N,r_N}^c,
\end{equation}
where \(r_N=e^{-\frac{A\log N}{N}}\) and \(\mathscr D_{N,r_N}^c\) is defined in \eqref{eq:def-D_(N,r)^c}.
\end{lemma}

\begin{proof}
Let \(\psi_N\) and \(\phi_N\) be as in \eqref{eq:def-psi_N} and \eqref{eq:def-phi_N}, respectively. We choose \(\eta_N,\zeta_N\in\mathbb T\) such that 
\[\min_{\eta,\zeta\in \mathbb T}  \abs{\psi_N(\eta)-\psi_N(r_N\zeta)}=\abs{\psi_N(\eta_N)-\psi_N(r_N\zeta_N)}.\]
In other words, the points \(w_N=\psi_N(\eta_N)\) and \(z_N =  \psi_N(r_N\zeta_N)\) minimize the distance between the two curves \(\Gamma_{N,1}=\psi_N(\mathbb T)\) and \(\Gamma_{N,r_N}=\psi_N(\mathbb T_{r_N})\). Let \(\gamma_N\) be the line segment connecting \(z_N\) and \(w_N\), then \(\gamma_N\subset \mathscr D_{N,r_N}^c\).
Employing the maximum principle and \cref{coro:uni-conv-Fourier} we notice that for large \(N\)
\begin{align*}
\max_{z\in\mathscr D_{N,r_N}^c}\abs{\phi_N'(z)}&=\biggl (\min_{\abs{z}=r_N}\abs{\psi_N'(z)}\biggr)^{-1}\\
&=\biggl (\min_{\abs{z}=1}\abs{\psi_N'(z)}\biggr)^{-1}+\mathcal O(N^{-\beta})\\
&< (1+\varepsilon)\norm{\phi'}_{L^\infty(\partial \mathscr D)},
\end{align*}
where \(\varepsilon>0\) is arbitrary.
Hence we have
\begin{align*}
1-r_N\leq\abs{r_N\zeta_N-\eta_N}&=\abs{\phi_N(z_N)-\phi_N(w_N)}\\
&\leq \int_{\gamma_N} \abs{\phi_N'(s)}\abs{\d s}\\
&< (1+\varepsilon)\norm{\phi'}_{L^\infty(\partial \mathscr D)}\abs{z_N-w_N}.
\end{align*}
Using the classical estimate \(1-e^{-x}\geq x-x^2/2\) for \(x\geq 0\) we get for large \(N\)
\[1-r_N\geq \frac{A\log N}{N}-\frac{(1+\varepsilon) \norm{\phi'}_{L^\infty(\partial \mathscr D)}}{N},\]
and subsequently
\[\frac{A/R}{1+\varepsilon}\frac{R \log N}{\norm{\phi'}_{L^\infty(\partial \mathscr D)} N}- N^{-1}< \min_{\eta,\zeta\in \mathbb T}  \abs{\psi_N(\eta)-\psi_N(r_N\zeta)}.\]
As \(\varepsilon\) was chosen arbitrarily, we may assume it to satisfy \(1+\varepsilon<A/R\).
Using the estimates above together with \cref{theo:Bernstein0} we obtain for sufficiently large \(N\)
\begin{align*}
\min_{\eta,\zeta\in \mathbb T}\abs{\psi_N(r_N\zeta)-\psi(\eta)}
&\geq \min_{\eta,\zeta\in\mathbb T} \abs{\psi_N(r_N\zeta)-\psi_N(\eta)}-\abs{\psi_N(\eta)-\psi(\eta)}\\
&> \frac{A/R}{1+\varepsilon}\frac{R \log N}{\norm{\phi'}_{L^\infty(\partial \mathscr D)} N}-2N^{-1}\\
&\geq \frac{R\log N}{\norm{\phi'}_{L^\infty(\partial \mathscr D)} N}.
\end{align*}
Therefore, the inclusion \eqref{eq:lem-min-curve} is established.
\end{proof}

\begin{theorem}\label{theo:uni-approx-phi}
Let \(\varkappa\in \mathbb Z_{\geq 0}\), \(\alpha\in(0,1)\) and \(\mathscr D\) be a \(C^{\varkappa+1,\alpha}\)-smooth Jordan domain and  \(\phi_N\coloneqq (\mathcal F_{L_N}[\phi^{-1}])^{-1}\). Then, for large \(N\), \(\phi_N\) is well-defined on \(\mathscr D^c\) and uniformly on \(\mathscr D^c\) it holds \(\phi_N-\phi=\mathcal O(N^{-\varkappa-1-\beta})\) for any \(\beta\in(0,\alpha)\). 
\end{theorem}
\begin{proof}
From \cref{lem:min-curve} we get \(\mathscr D^c \subset \mathscr D_{N,r_N}^c\) for \(r_N= e^{\frac{-\log N}{N}}\) and large \(N\), therefore \(\phi_N\) is well-defined on \(\mathscr D^c\).  We notice that \(\phi_N'(\infty)=\phi'(\infty)\), which follows from \(\psi_N'(\infty)=\psi'(\infty)\) together with \(\phi_N'(\infty)=\psi_N'(\infty)^{-1}\) and \(\phi'(\infty) = \psi'(\infty)^{-1}\). Therefore, \(\phi_N-\phi\) is bounded at infinity and due to the maximum principle it is sufficient to establish the uniform estimate only for boundary values.

Let \(z\in \partial \mathscr D\), then there exists a unique \(\zeta\in\mathbb T\) such that \(\psi(\zeta) = z\). Let \(\mathbb D(c,r)\) denote the disk with center \(c\) and radius \(r\). For \(\rho<1-r_N\) and sufficiently large \(N\), Koebe's quarter theorem yields
\[\mathbb D\left(\psi_N(\zeta),\tfrac{1}{4}\abs{\psi_N'(\zeta)}\rho\right)\subset \psi_N(\mathbb D(\zeta,\rho)).\]
As before, there exists a constant \(c>0\) such that \(c<\abs{\psi'_N(\eta)}\) for all \(\eta\in\mathbb T\) and \(N\) large.  Then we set
\[\varepsilon_N = \max_{\zeta\in\mathbb T}\abs{\psi(\zeta)-\psi_N(\zeta)}\]
and \(\rho_N=\frac{8 \varepsilon_N}{c}\). Using \cref{theo:Bernstein0} we have \(\varepsilon_N=\mathcal O(N^{-\varkappa-1-\beta})\) for every \(\beta\in(0,\alpha)\). We observe for large \(N\)
\[z\in \mathbb D(\psi_N(\zeta),2\varepsilon_N)=\mathbb D(\psi_N(\zeta),\tfrac 14c\rho_N)\subset \mathbb D(\psi_N(\zeta),\tfrac{1}{4}\abs{\psi_N'(\zeta)}\rho_N)\subset \psi_N(\mathbb D(\zeta,\rho_N)),\]
where we have used that \(\abs{z-\psi_N(\zeta)}=\abs{\psi(\zeta)-\psi_N(\zeta)}\leq \varepsilon_N\) as well as \(\rho_N<1-r_N\) for large \(N\). We can now find the unique value \(\zeta_N\in \mathbb D(\zeta,\rho_N)\) such that \(\psi_N(\zeta_N) = z\) and  obtain
\[\abs{\phi_N(z)-\phi(z)}=\abs{\zeta_N-\zeta}\leq \rho_N=\frac{8\varepsilon_N}{c}=\mathcal O(N^{-\varkappa-1-\beta}),\]
which yields the wanted result.
\end{proof}

\subsection{Admissible Laurent Families}\label{sec:AL}
In the following we are interested in functions of the form
\begin{equation}\label{eq:AL0}
f_{N,t}(\zeta)=\sum_{\abs{k}\leq \ell L_N} a_{N,t,k}\zeta^k,\qquad L_N= \Bigl \lfloor \mfrac{N}{\log N}\Bigr\rfloor,
\end{equation}
where \(\ell \in \mathbb N\) and the coefficients \(a_{N,t,k}\) are \(C^{\varkappa,\alpha}\)-smooth for \(\abs{t}< \delta_N=\frac{A\log N}{N}\) for any \(A>0\) and large \(N\). Similarly to \(\mathcal F_N[f]\) we can view \(f_{N,t}\) as holomorphic functions away from zero or infinity. Clearly, if the Fourier coefficients \(a_{N,t,k}\) vanish for all non-positive or non-negative \(k\), then the definition remains valid at zero or infinity, respectively. In the same way, it is reasonable to consider the derivative of \(f_{N,t}\) as
\[f'_{N,t}(z) = \sum_{\abs{k}\leq \ell L_N}ka_{N,t,k} z^{k-1},\quad z\neq 0,\]
when \(f_{N,t}\) is regarded as a Laurent polynomial.

In \cref{sec:asymp-flow-decomp}, we will work with sequences \(f_{N,t,j}\) of Laurent polynomials of the form \eqref{eq:AL0}. These families are constructed iteratively, with each new term depending on the families already obtained. The recursion involves differentiation, the Herglotz transform, multiplication, and finite Taylor expansions. We therefore need a class of Laurent polynomials which is stable under these operations and which preserves the estimates required at each step.

In spirit, we want these families to behave like \(f_{N,t}(\zeta)=\mathcal F_{L_N}[f(e^{-t}\cdot)](\zeta)\), where \(f\) is a \(C^{\varkappa,\alpha}\)-smooth function on an annular neighborhood of \(\mathbb T\). By \cref{theo:Bernstein0} and \cref{coro:uni-conv-Fourier}, for every \(\beta\in(0,\alpha)\) and every \(A>0\), one has
\[\norm{f_{N,0}-f}_{L^\infty(\mathbb T)}=\mathcal O(N^{-\varkappa-\beta})\qquad\text{and}\qquad\sup_{\abs t\leq\delta_N}\norm{f_{N,t}-f}_{L^\infty(\mathbb T)}=\mathcal O(N^{-\beta}).\]
The same behavior persists after taking the \(t\)-derivative, with the expected loss of order.

\begin{definition}\label{def:AL-Families}
Let \(f_{N,t}\) be a family of Laurent polynomials of the form \eqref{eq:AL0} such that  for every \(A>0\) and large \(N\) the coefficients \(a_{N,t,k}\) are \(C^{\varkappa,\alpha}\)-smooth for \(\abs{t}< \delta_N=\tfrac{A\log N}{N}\). 
For every \(s=0,\dots,\varkappa\), we define  \(f_{N,t}^{[s]}\coloneqq \partial_t^s f_{N,t}\). We call \(f_{N,t}\) an  \textit{\((\varkappa,\alpha)\)-admissible Laurent family}, or for short \(f_{N,t}\in \sL^{\varkappa,\alpha}\), if there exist limiting  functions \(f^{[s]}\in C^{\varkappa-s,\alpha}(\mathbb T)\) such that for any \(\beta\in(0,\alpha)\) and \(A>0\)
\begin{equation}\label{eq:AL1}
\norm{f^{[s]}_{N,0}-f^{[s]}}_{L^\infty(\mathbb T)}=\mathcal O(N^{s-\varkappa-\beta})
\end{equation}
as well as
\begin{equation}\label{eq:AL2}
\sup_{\abs{t}\leq \delta_N} \norm{f^{[s]}_{N,t}-f^{[s]}}_{L^\infty(\mathbb T)}=\mathcal O(N^{-\beta}).
\end{equation}
\end{definition}

Before turning to the main properties of this class, we record a few examples. If \(f\in C^{\varkappa,\alpha}(U)\) for some neighborhood \(U\) of \(\mathbb T\), then both \(\mathcal F_{L_N}[f](e^{-t}\cdot)\) and \(\mathcal F_{L_N}[f(e^{-t}\cdot)]\) belong to \(\sL^{\varkappa,\alpha}\). The class also contains simple families which do not arise from the Fourier truncation, for instance \(f_{N,t}(\zeta)=t\zeta^k\), with \(k\in\mathbb Z\).

We next collect some elementary stability properties of \(\sL^{\varkappa,\alpha}\). If \(f_{N,t},g_{N,t}\in\sL^{\varkappa,\alpha}\), then \(f_{N,t}+g_{N,t}\in\sL^{\varkappa,\alpha}\) and \(f_{N,t}\cdot g_{N,t}\in\sL^{\varkappa,\alpha}\), where the latter one follows from the Leibniz rule. The class is also closed under conjugation. More precisely, \(\overline{f_{N,t}}\in\sL^{\varkappa,\alpha}\), and therefore \(\Re f_{N,t},\Im f_{N,t}\in\sL^{\varkappa,\alpha}\).
When the Laurent polynomials are viewed as holomorphic functions away from the unit circle, this conjugation should be interpreted through reflection. Indeed, the function \(\overline{f_{N,t}(z)}\) is not holomorphic in \(z\). Instead one considers \(\overline{f_{N,t}(\overline z^{-1})}\), which is again a Laurent polynomial and agrees with \(\overline{f_{N,t}}\) on \(\mathbb T\).
Thus, in the sense of \cref{def:AL-Families} we regard the conjugation as complex reflection. It is  natural to ask whether more complicated operations also preserve the class \(\sL^{\varkappa,\alpha}\).

\begin{lemma}\label{lem:AL-stable}
If \(f_{N,t}\in \sL^{\varkappa,\alpha}\), then its Herglotz transform \(\Herg[f_{N,t}](\zeta)\) as well as \(f_{N,0}(e^{-t}\zeta)\) belong to \(\sL^{\varkappa,\alpha}\). If further \(\varkappa\geq 1\), then \(f'_{N,t},\partial_t f_{N,t}\in \sL^{\varkappa-1,\alpha}\).
\end{lemma}

\begin{proof}
Let us first consider \(g_{N,t}= \Herg[f_{N,t}]\). We denote by \(f^{[s]}\)  the limiting functions of \(f_{N,t}\) according to \cref{def:AL-Families} and  set \(g^{[s]}= \Herg[f^{[s]}]\). Note that the Herglotz transform of \(f_{N,t}\) commutes with \(\partial _t\), hence \(g_{N,t}^{[s]} =\partial_t^s \Herg[f_{N,t}]=\Herg[f^{[s]}_{N,t}]\) and we observe
\begin{align*}
\norm{g^{[s]}_{N,t}-g^{[s]}}_{L^\infty(\mathbb T)}&=\norm{\Herg\big[f_{N,t}^{[s]}\big]-\Herg\big[f^{[s]}\big]}_{L^\infty(\mathbb T)}\\
&\leq  \norm{\Herg\big[f_{N,t}^{[s]}\big]-\Herg\big[\mathcal F_{L_N}\big[f^{[s]}\big]\big]}_{L^\infty(\mathbb T)}+ \norm{\Herg\big[\mathcal F_{L_N}\big[f^{[s]}\big]\big]-\Herg\big[f^{[s]}\big]}_{L^\infty(\mathbb T)}\\
&\eqqcolon T_1+T_2.
\end{align*}
For the first term \(T_1\), we set \(M_N\coloneqq \max(\deg f_{N,t},L_N)\) and employ \eqref{eq:herglotz-estimate1} to obtain for some \(c>0\) and any \(\beta\in(0,\alpha)\)
\begin{align*}
T_1&\leq c\log(M_N+2) \norm{f_{N,t}^{[s]}-\mathcal F_{L_N}\big[f^{[s]}\big]}_{L^\infty(\mathbb T)}\\
&\leq c\log(M_N+2) \norm{f_{N,t}^{[s]}-f^{[s]}}_{L^\infty(\mathbb T)}+c\log(M_N+2) \norm{f^{[s]}-\mathcal F_{L_N}\big[f^{[s]}\big]}_{L^\infty(\mathbb T)}\\
&=\mathcal O\left( \log N \norm{f_{N,t}^{[s]}-f^{[s]}}_{L^\infty(\mathbb T)}+ N^{s-\varkappa-\beta}\right),
\end{align*}
where we have used \(\log(M_N+2)=\mathcal O(\log N)\). For the second term we exploit that the Herglotz transform and Fourier truncation commute to achieve
\[T_2=  \norm{\mathcal F_{L_N}\big[\Herg\big[f^{[s]}\big]\big]-\Herg\big[f^{[s]}\big]}_{L^\infty(\mathbb T)}=\mathcal O(N^{s-\varkappa-\beta}).\]
By combining the estimates for \(T_1\) and \(T_2\) we establish \eqref{eq:AL1} and \eqref{eq:AL2} for \(g_{N,t}\), as for some constant \(c'>0\) we have
\[\norm{g^{[s]}_{N,t}-g^{[s]}}_{L^\infty(\mathbb T)}\leq c' \log N \norm{f_{N,t}^{[s]}-f^{[s]}}_{L^\infty(\mathbb T)}+\mathcal O(N^{s-\varkappa-\beta}).\]
This completes the proof of \(g_{N,t}\in \sL^{\varkappa,\alpha}\).

Let us turn to \(h_{N,t}= f_{N,0}(e^{-t}\zeta)\). Then \(h_{N,t}\) is again a Laurent polynomial of degree \(\mathcal O(L_N)\), while the coefficients are sufficiently smooth. Further, we observe that \(\partial_th_{N,t}(\zeta) = i D_\tau f_{N,0}(e^{-t}\zeta)\), hence for \(s\leq \varkappa\) one has
\[h_{N,t}^{[s]}(\zeta)=\partial_t^s h_{N,t}(\zeta)=i^s\bigl(D_\tau^{s}f_{N,0}\bigr)(e^{-t}\zeta).\]
We identify the limiting functions \(h^{[s]}= i^s D_\tau^{s}f\in C^{\varkappa-s,\alpha}(\mathbb T)\), where \(f= f^{[0]}\in C^{\varkappa,\alpha}(\mathbb T)\) denotes the limiting function of \(f_{N,t}\) according to \cref{def:AL-Families}. We recall that \(D_\tau\) and \(\mathcal F_{L_N}\) commute. We consider for \(\beta\in (0,\alpha)\)
\begin{align*}
\norm{h_{N,t}^{[s]}-h^{[s]}}_{L^\infty(\mathbb T)}&\leq \norm{h_{N,t}^{[s]}- \mathcal F_{L_N}\big[h^{[s]} \big](e^{-t}\cdot)}_{L^\infty(\mathbb T)}+\norm{ \mathcal F_{L_N}\big[h^{[s]} \big](e^{-t}\cdot)-h^{[s]}}_{L^\infty(\mathbb T)}\\
&\hspace{-1cm}= \norm{D_\tau^s\big[f_{N,0}(e^{-t}\cdot)-\mathcal F_{L_N}[f] (e^{-t}\cdot)\big]}_{L^\infty(\mathbb T)}+\norm{ \mathcal F_{L_N}\big[h^{[s]} \big](e^{-t}\cdot)-h^{[s]}}_{L^\infty(\mathbb T)}\\
&\hspace{-1cm}\leq M_N^s \norm{f_{N,0}(e^{-t}\cdot)-\mathcal F_{L_N}[f](e^{-t}\cdot)}_{L^\infty(\mathbb T)}+\norm{ \mathcal F_{L_N}\big[h^{[s]} \big](e^{-t}\cdot)-h^{[s]}}_{L^\infty(\mathbb T)},
\end{align*} 
where we have used the Bernstein inequality \eqref{eq:bernstein-ineq1}. We observe \(M_N^s=\mathcal O(N^s)\), hence \eqref{eq:AL1} follows from \cref{theo:Bernstein0}, while \eqref{eq:AL2} is a consequence from \cref{theo:Fourier-Master} and \cref{coro:uni-conv-Fourier}. This proves that \(h_{N,t}\in \sL^{\varkappa,\alpha}\).

We assume now \(\varkappa\geq 1\). For the \(t\)-derivative, \cref{def:AL-Families} immediately ensures that  \(\partial _t f_{N,t}\in \sL^{\varkappa-1,\alpha}\). Lastly, we consider \(p_{N,t}=f_{N,t}'\), which again is a Laurent polynomial of degree \(\mathcal O(L_N)\), where  the coefficients have the required smoothness in \(t\).  Employing \eqref{eq:polar-to-wirtinger} we get for any \(s=0,\dots, \varkappa-1\)
\[
p_{N,t}^{[s]}(\zeta)=\partial_t^s f_{N,t}'(\zeta)=\big(f_{N,t}^{[s]}\big)'(\zeta)=- i\zeta^{-1} D_\tau f_{N,t}^{[s]}(\zeta).\]
Thus, we identify the limiting functions as \(p^{[s]}(\zeta)=-i\zeta^{-1} D_\tau f^{[s]}(\zeta)\in C^{\varkappa-1-s,\alpha}(\mathbb T)\). Then similarly as before we have for any \(\beta\in(0,\alpha)\)
\begin{align*}
\norm{p_{N,0}^{[s]}-p^{[s]}}_{L^\infty(\mathbb T)}&\leq M_N \norm{f^{[s]}_{N,0}-\mathcal F_{L_N}\big [f^{[s]}\big]}_{L^\infty(\mathbb T)}+\norm{\mathcal F_{L_N}\big [D_\tau f^{[s]}\big] -D_\tau f^{[s]}}_{L^\infty(\mathbb T)}\\
&=\mathcal O\bigl(N^{s-(\varkappa-1)-\beta}\bigr),
\end{align*}
which verifies \eqref{eq:AL1}. Now turning to the uniform estimate \eqref{eq:AL2} we consider
\[D_\tau f^{[s]}_{N,t}- D_\tau f^{[s]}= D_\tau f^{[s]}_{N,0}- D_\tau f^{[s]} + \int_0^t D_\tau f_{N,u}^{[s+1]}\,\d u.\]
Hence, using the \(t=0\) estimate established above we conclude
\[\sup_{\abs{t}\leq \delta_N}\norm{p_{N,t}^{[s]}- p^{[s]}}_{L^\infty(\mathbb T)}\leq  \mfrac{A\log N}{N} \sup_{\abs{u}\leq \delta _N}\norm{D_\tau f^{[s+1]}_{N,u}}_{L^\infty(\mathbb T)}+\mathcal O\bigl(N^{s-(\varkappa-1)-\beta}\bigr).\]
Using analogous tricks as before we can estimate the remaining norm. We have for \(\abs{u}\leq \delta_N\)
\begin{align*}
\norm{D_\tau f^{[s+1]}_{N,u}}_{L^\infty(\mathbb T)}&\leq \norm{D_\tau f^{[s+1]}_{N,u} - D_\tau \mathcal F_{L_N}\big[ f^{[s+1]}\big]}_{L^\infty(\mathbb T)}+\norm{D_\tau \mathcal F_{L_N}\big[f^{[s+1]}\big]}_{L^\infty(\mathbb T)}\\
&\leq M_N \norm{ f^{[s+1]}_{N,u} -  \mathcal F_{L_N}\big[ f^{[s+1]}\big]}_{L^\infty(\mathbb T)}+\norm{D_\tau \mathcal F_{L_N}\big[f^{[s+1]}\big]}_{L^\infty(\mathbb T)}\\
&=\norm{D_\tau \mathcal F_{L_N}\big[f^{[s+1]}\big]}_{L^\infty(\mathbb T)}+\mathcal O\bigl(L_N N^{-\beta}\bigr).
\end{align*}
For the remaining term we employ the Bernstein-type inequality \eqref{eq:bernstein-ineq3} and \cref{theo:Bernstein0} to obtain for any \(\gamma \in(\beta,\alpha)\)
\begin{align*}
\norm{D_\tau \mathcal F_{L_N}\big[f^{[s+1]}\big]}_{L^\infty(\mathbb T)}&\leq 2L_N^{1-\gamma} \norm{\mathcal F_{L_N}\big[f^{[s+1]}\big]}_{C^{0,\gamma}(\mathbb T)}\\
&\leq 2 c_{\alpha,\gamma} L_N^{1-\gamma} \norm{f^{[s+1]}}_{C^{0,\alpha}(\mathbb T)}=\mathcal O\bigl(L_N N^{-\beta}\bigr),
\end{align*}
where \(c_{\alpha,\gamma}\) is the constant implied by \cref{theo:Bernstein0}.
Combining previous estimates yields
\[\sup_{\abs{t}\leq \delta_N}\norm{p_{N,t}^{[s]}- p^{[s]}}_{L^\infty(\mathbb T)}=\mathcal O( N^{-\beta}),\]
which establishes \eqref{eq:AL2} and consequently \(p_{N,t}\in \sL^{\varkappa-1,\alpha}\).
\end{proof}

As already indicated, the iterative construction in
\cref{sec:asymp-flow-decomp} will produce families of Laurent polynomials \(f_{N,t,j}\in\sL^{\varkappa-j,\alpha}\). Each \(f_{N,t,j}\) is obtained from the previously constructed families \(f_{N,t,k}\), \(k<j\), together with certain input terms \(g_{N,t,j}\in\sL^{\varkappa-j,\alpha}\) determined by \(\omega\) and \(\mathscr D\). The stability properties discussed above ensure that the recursive operations preserve admissibility, while the following proposition provides the corresponding admissibility statement for the external input terms coming from \(\omega\) and the geometry of \(\mathscr D\).

\begin{proposition}\label{prop:u_Nt-is-AL}
Let \(\varkappa\in \mathbb Z_{\geq 0}\) and \(\alpha \in(0,1)\). Let \(\mathscr D\) be a \(C^{\varkappa+1,\alpha}\)-smooth Jordan domain and \(\omega\) be a non-negative \(C^{\varkappa,\alpha}\)-smooth function defined on a neighborhood of \(\overline{\mathscr D}\) such that \(\omega>0\) on \(\partial \mathscr D\). Then for \(m,n\in\mathbb Z_{\geq 0}\) and \(m+n\leq \varkappa\) the family
\[u_{N,t,m,n}\coloneqq \mathcal F_{L_N}\bigl[\partial^{m}\bar{\partial}^n[\log \omega\circ \phi^{-1}_{N}](e^{-t}\cdot)\bigr]\]
 belongs to \(\sL^{\varkappa-m-n,\alpha}\). Moreover, for every \(A>0\), we have 
\[\tnorm{ \log \omega \circ \phi_N^{-1}}_{C^{\varkappa,\alpha}(\mathcal A_N)}=\mathcal O(1)\]
for \(\mathcal A_N=\mathbb D_{1/r_N}\setminus \overline{\mathbb D_{r_N}}\) with  \(r_N = e^{-\frac{A\log N}{N}}\).
\end{proposition}

\begin{proof}
Let us first justify that \(u_{N,t,m,n}\) is indeed well-defined. Since \(\omega>0\) on \(\partial\mathscr D\), we can find an open neighborhood \(U\) of \(\partial\mathscr D\) such that \(\omega>0\) on \(U\).  From \cref{coro:uni-conv-Fourier} we obtain that \(\phi_N^{-1}(e^{-t}\zeta)\) converges uniformly to \(\phi^{-1}\) on \(\mathbb T\) for \(\abs{t}\leq \tfrac{A\log N}{N}\) and any fixed \(A>0\). Therefore \(\phi_N^{-1}(e^{-t}\zeta)\in U\) for large \(N\) and \(\abs{t}\leq \tfrac{A\log N}{N}\). In other words, \(u_{N,t,m,n}\) is well-defined.

Next, we observe that \(u_{N,t,m,n}\) already admits the correct shape \eqref{eq:AL0}, including \(C^{\varkappa-m-n,\alpha}\)-smooth coefficients. Therefore we need to find the correct limiting functions \(u_{m,n}^{[s]}\) and establish \eqref{eq:AL1} together with \eqref{eq:AL2}.

In the following we will use the notation \(u=\log \omega\), \(\psi=\phi^{-1}\) and \(\psi_N = \phi^{-1}_N\). All computations below
are carried out for a fixed pair \((m,n)\). Accordingly, functions and constants
may depend on \(m\) and \(n\), even when this dependence is not indicated
explicitly.

We fix some \(s\leq \varkappa-m-n\) and aim to establish \eqref{eq:AL1} and \eqref{eq:AL2}.  For this, we need to transform \(u_{N,t,m,n}^{[s]}\) in a more suitable form. For some generic function \(f\colon U\to\mathbb R\) and \(z_t=e^{-t}\zeta\), the repeated use of \(\partial_t f(z_t)=-z_t(\partial f)(z_t)-\bar z_t(\bar \partial f)(z_t)\)  yields
\[
\partial_t^s[f(z_t)]= \sum_{p+q=0}^s a_{s,p,q}\, z_t^p\bar z_t^q\, \partial^p\bar \partial^qf(z_t),
\]
for some coefficients \(a_{s,p,q}\).  Since \(\partial_t\) and \(\mathcal F_{L_N}\) commute we have
\begin{align*}
u_{N,t,m,n}^{[s]}(\zeta)&=\partial_t^s[u_{N,t,m,n}](\zeta)\\
&=\mathcal F_{L_N}\left[\partial_t^s   \partial^m\bar\partial^n[u \circ \psi_N](e^{-t}\cdot)\right](\zeta)\\
&= \sum_{p+q=0}^s a_{s,p,q} e^{-(p+q)t}\mathcal F_{L_N}\left[\zeta^p\bar \zeta^q \,\partial^{m+p}\bar\partial^{n+q}[u \circ \psi_N](e^{-t}\cdot)  \right](\zeta)\\
&\eqqcolon \sum_{p+q=0}^s a_{s,p,q} \, v_{N,t,p,q}(\zeta).
\end{align*}
We set \(\sigma=\varkappa-m-n-s\) and  realize that it is sufficient to find limiting functions \(v_{p,q}\in C^{\sigma,\alpha}(\mathbb T)\) for \(v_{N,t,p,q}\) which satisfy \eqref{eq:AL1} and \eqref{eq:AL2} with the error rates \(\mathcal O(N^{-\sigma-\beta})\) and \(\mathcal O(N^{-\beta})\), respectively. For the limiting function we have the natural candidate 
\[v_{p,q}(\zeta)\coloneqq \zeta^p\bar \zeta^q \,\partial^{m+p}\bar\partial^{n+q}[u\circ \psi](\zeta).\]
Then \cref{theo:Bernstein0} yields  for \(\abs{t}\leq \tfrac{A\log N}{N}\)
\begin{equation}\label{eq:prf-u_nt-is-AL0.0}
\norm{v_{N,t,p,q}-v_{p,q}}_{L^\infty(\mathbb T)}\leq \norm{v_{N,t,p,q}-\mathcal F_{L_N}[v_{p,q}]}_{L^\infty(\mathbb T)}+\mathcal O(N^{-\sigma-\beta}) .
\end{equation}
We recall the estimate \(\tnorm{\mathcal F_N[f]}_{L^\infty(\mathbb T)}\leq c\log N\tnorm{f}_{L^\infty(\mathbb T)}\) for \(N\geq 2\), which follows from the \(L^1\)-bound for the Dirichlet kernel \(\sum_{k=-N}^N\zeta^k\). Then the term on the right side of \eqref{eq:prf-u_nt-is-AL0.0} can be estimated by
\begin{align*}
\norm{v_{N,t,p,q}-\mathcal F_{L_N}[v_{p,q}]}_{L^\infty(\mathbb T)}&\\
&\hspace{-83pt}=\norm{\mathcal F_{L_N} \left[ e^{-t(p+q)}\zeta^p\bar{\zeta}^q\, \partial^{m+p}\bar\partial^{n+q}[u \circ \psi_N](e^{-t}\cdot)-\zeta^p\bar{\zeta}^q\,\partial^{m+p}\bar\partial^{n+q}[u\circ \psi] \right]}_{L^\infty(\mathbb T)}\\
&\hspace{-83pt}\leq c\log N\norm{ e^{-t(p+q)} \,\partial^{m+p}\bar\partial^{n+q}[u \circ \psi_N](e^{-t}\cdot)-\partial^{m+p}\bar\partial^{n+q}[u\circ \psi] }_{L^\infty(\mathbb T)}\\
&\hspace{-83pt}\leq  c' \log N e^{-t(p+q)} \norm{\partial^{m+p}\bar\partial^{n+q}[u \circ \psi_N](e^{-t}\cdot)- \partial^{m+p}\bar\partial^{n+q}[u\circ \psi]}_{L^\infty(\mathbb T)}\\
&\hspace{-83pt}\quad + c'\log N \norm{\partial^{m+p}\bar\partial^{n+q}[u\circ \psi]}_{L^\infty(\mathbb T)}\bigl|e^{-t(p+q)}-1\bigr|,
\end{align*}
for some \(c,c'>0\). Combining this with \(\tabs{e^{-t(p+q)}-1}=\mathcal O(\tabs{t})\) and \eqref{eq:prf-u_nt-is-AL0.0} yields, uniformly for \(\abs{t}\leq \tfrac{A\log N}{N}\),
\begin{multline*}
\norm{v_{N,t,p,q}-v_{p,q}}_{L^\infty(\mathbb T)}\leq c'' \log N  \norm{\partial^{m+p}\bar\partial^{n+q}[u \circ \psi_N](e^{-t}\cdot)- \partial^{m+p}\bar\partial^{n+q}[u\circ \psi]}_{L^\infty(\mathbb T)}\\
{}+ \mathcal O\bigl(\abs{t}\log N +N^{-\sigma-\beta}\bigr),
\end{multline*}
for some \(c''>0\).
To estimate the remaining norm, we use a Faà di Bruno-type expansion for the derivatives of \(u\circ \psi_N\).
More precisely, for non-negative integer \(x\) and \(y\) satisfying \(x+y\leq \varkappa\) there exist multivariate polynomials \(P_{k,\ell}\), independent of \(N\), \(u\) and \(\psi_N\), such that
\begin{equation}\label{eq:prf-u_nt-is-AL1}
\partial^{x}\bar\partial^{y}[u\circ\psi_N] =
\sum_{a,b=0}^{x,y} (\partial^{a}\bar\partial^{b}u)(\psi_N)\, P_{x,a}\left(\psi_N',\dots,\psi_N^{(x)}\right)\overline{P_{y,b}\left(\psi_N',\dots,\psi_N^{(y)}\right)}.
\end{equation} 
An analogous expression holds for \(\partial^{x}\bar\partial^{y}[u\circ \psi]\). We recall \(\sigma=\varkappa-m-n-s\). It becomes clear that \(v_{N,t,p,q}\) satisfies \eqref{eq:AL1} and \eqref{eq:AL2} if for \(a\leq m+p\) and \(b\leq n+q\), and any \(\gamma\in(\beta,\alpha)\)
\begin{equation}\label{eq:prf-u_nt-is-AL1.1}
\norm{(\partial^{a}\bar\partial^{b}u)(\psi(\cdot))-(\partial^{a}\bar\partial^{b}u)(\psi_N(e^{-t}\cdot))}_{L^\infty(\mathbb T)}=\mathcal O\bigl(\abs{t}^\gamma+N^{-\sigma-\gamma}\bigr),\qquad \abs{t}\leq \mfrac{A\log N}{N},
\end{equation}
and for \(j=0,\dots,\varkappa\)
\begin{equation}\label{eq:prf-u_nt-is-AL1.2}
\norm{\psi^{(j)}-\psi_N^{(j)}(e^{-t}\cdot)}_{L^\infty(\mathbb T)}=\mathcal O\bigl(\abs{t}+ N^{j-\varkappa-1-\gamma}\bigr).
\end{equation}
The estimate \eqref{eq:prf-u_nt-is-AL1.2} follows from \cref{theo:Bernstein0} and \cref{theo:Fourier-Master}, while keeping in mind that \(\psi\in C^{\varkappa+1,\alpha}(\mathbb T)\). To establish \eqref{eq:prf-u_nt-is-AL1.1} we have to distinguish the two cases \(a+b = \varkappa\) and \(a+b<\varkappa\). In the former case we have \(\sigma=0\) as well as
\begin{align*}
\norm{(\partial^{a}\bar\partial^{b}u)(\psi(\cdot))-(\partial^{a}\bar\partial^{b}u)(\psi_N(e^{-t}\cdot))}_{L^\infty(\mathbb T)}&\leq \hnorm{\partial^{a}\bar\partial^{b}u}_{U,\alpha} \norm{\psi-\psi_N(e^{-t}\cdot)}_{L^\infty(\mathbb T)}^\alpha\\
&=\mathcal O\Bigl(\abs{t}^\alpha+N^{-\alpha(\varkappa+1+\gamma)}\Bigr),
\end{align*}
where \(U\) is a neighborhood of \(\partial \mathscr D\) on which \(u=\log \omega\) is well-defined. The second case can be solved by performing a first order Taylor expansion of \(\partial^{a}\bar\partial^{b}u\) and applying \eqref{eq:prf-u_nt-is-AL1.2}. This completes the proof that \(u_{N,t,m,n}\in \sL^{\varkappa-m-n,\alpha}\).

Let us now turn to proving \(\tnorm{ u \circ \phi_N^{-1}}_{C^{\varkappa,\alpha}(\mathcal A_N)}=\mathcal O(1)\). To show the uniform boundedness it is sufficient to show \(\tnorm{\partial^{x}\bar\partial^{y}[u\circ\psi_N]}_{L^\infty(\cA_N)}=\mathcal O(1)\) for arbitrary \(x+y\leq \varkappa\) and \([\partial^{x}\bar\partial^{y}[u\circ\psi_N]]_{\cA_N,\alpha}=\mathcal O(1)\) for \(x+y=\varkappa\).
For \(x+y\leq \varkappa\) we employ \eqref{eq:prf-u_nt-is-AL1} to obtain 
\begin{multline*}
\norm{\partial^{x}\bar\partial^{y}[u\circ\psi_N]}_{L^\infty(\cA_N)
}\\
{}\leq  \sum_{a,b=0}^{x,y}\norm{\partial^{a}\bar\partial^{b}u}_{L^\infty(U)}\norm{ P_{x,a}\left(\psi_N^{(1)},\dots,\psi_N^{(x)}\right){P_{y,b}\left(\psi_N^{(1)},\dots,\psi_N^{(y)}\right)}}_{L^\infty(\cA_N)},
\end{multline*}
which is uniformly bounded on the given domain due to the uniform boundedness of \(\psi_{N}^{(k)}\) for \(k\leq \varkappa+1\). Again, the uniform boundedness of \(\psi_N^{(k)}\) on \(\cA_N\) follows from \cref{coro:uni-conv-Fourier}. Further, for \(x+y=\varkappa\) we obtain in a similar way
\begin{align*}
\hnorm{\partial^{x}\bar\partial^{y}[u\circ\psi_N]}_{\cA_N,\alpha}&\\
&\hspace{-1.7cm}\leq  \sum_{a,b=0}^{x,y}\hnorm{(\partial^{a}\bar\partial^{b}u)\circ\psi_N}_{\cA_N,\alpha} \norm{ P_{x,a}\left(\psi_N^{(1)},\dots,\psi_N^{(x)}\right)P_{y,b}\left(\psi_N^{(1)},\dots,\psi_N^{(y)}\right)}_{L^\infty(\cA_N)}\\
&\hspace{-60pt}\quad  +\sum_{a,b=0}^{x,y}\norm{(\partial^{a}\bar\partial^{b}u)\circ\psi_N}_{L^\infty(\cA_N)} \Bigl[ P_{x,a}\left(\psi_N^{(1)},\dots,\psi_N^{(x)}\right)\overline{P_{y,b}\left(\psi_N^{(1)},\dots,\psi_N^{(y)}\right)}\Bigr]_{\cA_N,\alpha}\\
&\hspace{-1.7cm}\eqqcolon S_1+S_2,
\end{align*}
where we have used the Hölder estimate 
\begin{equation}\label{eq:prf-u_nt-is-AL2}
\hnorm{fg}_{\cA_N,\alpha}\leq \hnorm{f}_{\cA_N,\alpha} \norm{g}_{L^\infty(\cA_N)}+ \tnorm{f}_{L^\infty(\cA_N)}\hnorm{g}_{\cA_N,\alpha}
\end{equation}
for two generic functions \(f\) and \(g\). We can bound \(S_2\) using the estimate  
\[\tnorm{(\partial^{a}\bar\partial^{b}u)\circ\psi_N}_{L^\infty(\cA_N)}\leq \tnorm{\partial^{a}\bar\partial^{b}u}_{L^\infty(U)}\]
 and  \eqref{eq:prf-u_nt-is-AL2} together with 
\[\hnorm{\psi_N^{(j)}}_{\cA_N,\alpha}\leq \mfrac{\pi}{2} \norm{\psi_N^{(j+1)}}_{L^\infty(\cA_N)}=\mathcal O(1),\qquad j=1,\dots, \varkappa,\]
ensured by \cref{theo:Fourier-Master} since \(\psi^{(j+1)}\in C^{0,\alpha}(\mathbb T)\). Hence, we have \(S_2=\mathcal O(1)\). For \(S_1\) it is sufficient to bound \([(\partial^{a}\bar\partial^{b}u)\circ\psi_N]_{\cA_N,\alpha}\), which is achieved by
\[\hnorm{(\partial^{a}\bar\partial^{b}u)\circ\psi_N}_{\cA_N,\alpha} \leq \mfrac{\pi^\alpha}{2^\alpha} \hnorm{\partial^{a}\bar\partial^{b}u}_{U,\alpha} \norm{\psi_{N}'}_{L^{\infty}(\cA_N)}^\alpha= \mathcal O(1).\]
This verifies \(\tnorm{ \log \omega \circ \phi_N^{-1}}_{C^{\varkappa,\alpha}(\cA_N)}=\mathcal O(1)\) and hence completes the proof.
\end{proof}

\section{The Foliation Flow Technique}\label{sec:foliation-technique}
In this section we will recall the key ideas behind the \textit{foliation flow} technique and adapt it to our situation. This technique was developed in \cite{hedenmalm2017planar} to obtain an asymptotic expansion of planar orthogonal polynomials for varying weights, and was later used in \cite{hedenmalm2024riemann} to derive an analogous expansion for fixed weights. In both works, the weights are assumed to be real-analytic.

Let us consider a family of conformal mappings \(\psi_t\) defined on some neighborhood of \(\mathbb D^c\), normalized by \(\psi_t(\infty)=\infty\), and \(C^1\)-smooth in \(t\) for \(\abs{t}< \delta\), where \(\delta\in (0,1)\).
We further assume that the curves \(\psi_t(\mathbb T)\) are nested monotonically as \(t\) varies.
These curves then foliate the annular domain
\[\mathcal E\coloneqq \bigcup_{\abs t<\delta}\psi_t(\mathbb T).\]
Associated with the foliation \(\psi_t\), we define the annular realization
\[\Psi(z)\coloneqq \psi_{-\log \abs z}\left(\frac{z}{\abs z}\right),\qquad z\in \mathcal A_\delta\coloneqq \mathbb D_{e^\delta}\setminus \overline{\mathbb D}_{e^{-\delta}}.\]
Thus, \(\Psi\) maps the reference annulus \(\mathcal A_\delta\) onto the foliated region \(\mathcal E\).

Let \(f\) be integrable on \(\mathcal E\) with respect to \(\d\A\), the area measure normalized by \(1/\pi\). Assuming the Jacobian \(J_\Psi\) of \(\Psi\) is strictly positive, we get 
\begin{align*}
\int_{\mathcal E} f(z)\,\d\A(z)&=\int_{\Psi(\cA_\delta)}f(z)\,\d\A(z)\\
&=\int _{\cA_\delta} f( \Psi(z)) \,J_\Psi(z)\, \d\A(z)\\
&=\mfrac{1}{\pi} \int_{-\delta}^\delta \int _\mathbb T e^{-2t}f(\psi_t(\zeta))\, J_\Psi(e^{-t}\zeta)\,\abs{\d\zeta}\,\d t.
\end{align*}
The Jacobian can be computed explicitly.
Let us write \(\Psi(z)=\psi\circ g\) for \(\psi(z,t)\coloneqq \psi_t(z)\) and \(g(z)\coloneqq(g_1(z),g_2(z))=(\frac{z}{\abs{z}},-\log\abs{z})\). Then we can represent the Jacobian as
\[J_\Psi(z)=\abs{\partial \Psi(z)}^2-\abs{\overline \partial  \Psi(z)}^2=\abs{\partial [\psi\circ g](z)}^2-\abs{\overline \partial[\psi\circ g](z)}^2.\]
A direct computation reveals
\[J_\Psi(z) =-\frac{1}{\abs{z}^3}\mathrm{Re}\left(\overline{z} \,(\partial_t \psi_{t})\left(\tfrac{z}{\abs{z}} \right)\, \overline{\psi_t'\left(\tfrac{z}{\abs{z}}\right)}\right)_{\big|t =-\log \abs{z}}. \]
Thus, inserting this expression for the Jacobian into the preceding calculation, we obtain the simple disintegration formula
\begin{equation}\label{eq:simple disintegration}
\int_\mathcal E f(z)\,\d\A(z) =  -\mfrac{1}{\pi}\int_{-\delta}^\delta \int_\mathbb Tf(\psi_t(\zeta))\Re\bigl(\overline\zeta\,(\partial_t\psi_t)(\zeta)\overline{\psi'_t(\zeta)} \bigr)\,\abs{\d\zeta}\,\d t.
\end{equation}
From the disintegration formula \eqref{eq:simple disintegration}, we obtain a useful criterion for identifying the orthogonal polynomials. In fact, let \(P\) be a polynomial of degree \(N\) which satisfies
\begin{equation}\label{eq:simple constant flow}
\Upsilon_t(\zeta)\coloneqq -\mfrac{1}{\pi}\abs{P(\psi_t(\zeta)) }^2\mathrm{Re}\,\big(\overline\zeta\,(\partial_t\psi_t)(\zeta)\overline{\psi'_t(\zeta)} \big)\equiv c_t,\qquad \zeta\in\mathbb T.
\end{equation}
In other words, we assume that \(\Upsilon_t(\zeta)\) is constant in \(\zeta\). If \(P\) has no zeros in \(\mathcal E\) or in the unbounded component of \(\mathbb C\setminus \mathcal E\), then \eqref{eq:simple disintegration} yields for every \(k=0,\dots,N-1\)
\[\int_\mathcal Ez^k \overline{P(z)}\d \A(z)= \int_{-\delta}^\delta\int_\mathbb T\frac{\psi_t(\zeta)^k}{P(\psi_t(\zeta))}\Upsilon_t(\zeta)\abs{\d\zeta}\d t=\int_{-\delta}^\delta c_t \int_\mathbb T \frac{\psi_t(\zeta)^k}{P(\psi_t(\zeta))} \abs{\d\zeta}\d t=0,\]
where the last equality follows from the Cauchy integral theorem. The inner integral is \(2\pi\) times the value of the integrand at infinity, which is zero since \(P\)  grows faster than \(z^k\) at infinity. It follows that \(P\) is, up to normalization, the \(N\)-th orthogonal polynomial of \(L^2(\mathcal E)\).

In the following subsections, we aim to transfer the orthogonality criterion \eqref{eq:simple constant flow} to the weighted space \(L^2(\mathscr D,\omega)\). Ideally, one would start with a polynomial candidate \(P\) together with a foliation of the whole domain \(\mathscr D\) for which the associated density \(\Upsilon_t\) is constant. In general, however, it is not known whether such a global construction always exists.
We therefore first carry out the construction in conformal coordinates, near the unit circle. More precisely, we introduce a holomorphic model function \(f_N\) and a foliation \(\psi_{N,t}\) near \(\mathbb T\). These objects are then transferred to \(\mathscr D\), producing a corresponding model function \(F_N\) and a shrinking annular neighborhood of \(\partial\mathscr D\). Both \(f_N\) and \(\psi_{N,t}\) are encoded by admissible Laurent families. The main objective is to choose these families so that the associated density \(\Upsilon_{N,t}\) is almost constant, in analogy with \eqref{eq:simple constant flow}.

In \cref{sec:Foliation and Model}, we introduce the model functions, the foliation near \(\mathbb T\), and the procedure which transfers the construction from conformal coordinates to \(\mathscr D\). In \cref{sec:flow-r=0}, we determine the model function explicitly in the case \(\varkappa=0\). The higher-order construction is carried out in \cref{sec:asymp-flow-decomp}, where an asymptotic decomposition of the logarithmic density \(\log \Upsilon_{N,t}\) leads to an iterative description of the admissible Laurent families involved.

\subsection{The Foliation and the Model Function}\label{sec:Foliation and Model}
For the rest of this section, let \(A>0\) be a fixed parameter, which will be specified later on. In addition, we define \(\delta_N=\frac{A\log N}{N}\) and \(r_N=e^{-\delta_N}\) as in the previous section.
Inspired by \cite{hedenmalm2024riemann}, we assume \(\psi_{N,t}\) to be of the form
\begin{equation}\label{eq:def foliation mapping}
\psi_{N,t}(\zeta) \coloneqq \zeta\exp\left({\textstyle \sum_{j=0}^{\varkappa} N^{-j} \eta_{N,t,j}(\zeta)} \right),\qquad \eta_{N,t,0}(\zeta)\coloneqq -t,
\end{equation}
where the functions \(\eta_{N,t,j}\in \sL^{\varkappa+1-j,\alpha}\) admit only Fourier coefficients corresponding to non-positive indices and satisfy \(\eta_{N,0,j}\equiv 0\). 
As before we define the foliated domain
\[\mathcal E_N \coloneqq \bigcup_{\abs{t}< \delta_N} \psi_{N,t}(\mathbb T), \qquad\mathcal A_N\coloneqq \mathbb D_{e^{\delta_N}} \setminus \overline{\mathbb D}_{e^{-\delta_N}}.\]
The condition \(\eta_{N,0,j}\equiv 0\) ensures that \(\mathcal E_N\) is located around \(\mathbb T\).
It can be observed that \(\abs{\psi_{N,t}(\zeta)}=e^{-t+\mathcal O(N^{-1})}\), uniformly for \(\abs{t}\leq\delta_N\) and \(\zeta\in\mathbb T\), which yields the inclusion \(\mathcal E_N\subset\mathbb D_{e^{2\delta_N}}\setminus\overline{\mathbb D}_{e^{-2\delta_N}}\) for sufficiently large \(N\). Applying \cref{theo:conf-inverse}, we conclude that \(\phi_N^{-1}\) is conformal on \(\mathcal E_N\). We may therefore define
\[\mathscr E_N\coloneqq\phi_N^{-1}(\mathcal E_N).\]

\begin{lemma}
The mapping \(\Psi_N(z)\coloneqq\psi_{N,-\log\abs{z}}(\tfrac{z}{\abs{z}})\) is a \(C^1\)-diffeomorphism from \(\mathcal A_N\) onto \(\mathcal E_N\). Moreover, for large \(N\) and \(\abs{t}\leq\delta_N\), the mapping \(\psi_{N,t}\) is conformal on \(\mathbb D^c\).
\end{lemma}

\begin{proof}
By admissibility and \cref{lem:AL-stable}, the functions \(\eta_{N,t,j}\), \(\partial_t\eta_{N,t,j}\), and \(D_\tau\eta_{N,t,j}\) are uniformly bounded for \(\abs{t}\leq\delta_N\). For convenience, we decompose 
\[a_{N,t}(\theta)+ib_{N,t}(\theta)\coloneqq\sum_{j=1}^{\varkappa}N^{-j}\eta_{N,t,j}(e^{i\theta}),\]
into the real-valued families \(a_{N,t}\) and \(b_{N,t}\).
 It follows that \(a_{N,t}\), \(b_{N,t}\), and their first derivatives with respect to \(t\) and \(\theta\) are of order \(\mathcal O(N^{-1})\), uniformly for \(\abs{t}\leq\delta_N\).

We first verify the conformality of \(\psi_{N,t}\). On \(\mathbb T\), we have
\[\psi_{N,t}(e^{i\theta})=e^{-t+a_{N,t}(\theta)}e^{i(\theta+b_{N,t}(\theta))}.\]
Thus, \(\theta+b_{N,t}(\theta)\) gives a continuous choice of the argument of this curve. Its derivative is \(1+\partial_\theta b_{N,t}(\theta)=1+\mathcal O(N^{-1})>0\), so it is strictly increasing for large \(N\). Moreover, since \(b_{N,t}\) is \(2\pi\)-periodic, the function \(\theta+b_{N,t}(\theta)\) increases by exactly \(2\pi\) as \(\theta\) runs from \(0\) to \(2\pi\). Hence, the argument runs once around the circle without repetitions. We conclude that \(\psi_{N,t}\) is injective on \(\mathbb T\) and that \(\psi_{N,t}(\mathbb T)\) is positively oriented.

Since the Fourier coefficients of \(\eta_{N,t,j}\) vanish for all positive indices, these functions are bounded near infinity. Consequently, \(\psi_{N,t}\) is holomorphic in a neighborhood of \(\mathbb D^c\) and satisfies \(\psi_{N,t}(z)=c_{N,t}z+\mathcal O(1)\) as \(z\to\infty\), where \(c_{N,t}=1+\mathcal O(\frac{\log N}{N})\neq0\). We may therefore apply \cref{lem:conf-inject} and conclude that \(\psi_{N,t}\) is conformal on \(\mathbb D^c\).

It remains to compare the curves corresponding to different values of \(t\). Fix an angle \(\vartheta\in\mathbb R\). By the strict monotonicity established above, there is a unique \(\theta(t)\) such that \(\vartheta=\theta(t)+b_{N,t}(\theta(t))\).
Differentiating this identity with respect to \(t\) gives
\[\theta'(t)=-\frac{\partial_tb_{N,t}(\theta(t))}{1+\partial_\theta b_{N,t}(\theta(t))}=\mathcal O(N^{-1}).\]
The curve \(\psi_{N,t}(\mathbb T)\) intersects the ray of angle \(\vartheta\) at the radius \(r_{N,t}(\vartheta)=e^{-t+a_{N,t}(\theta(t))}\).
We therefore obtain
\[\partial_t\log r_{N,t}(\vartheta)=-1+\partial_ta_{N,t}(\theta(t))+\partial_\theta a_{N,t}(\theta(t))\theta'(t)=-1+\mathcal O(N^{-1}).\]
For large \(N\), this derivative is strictly negative, uniformly in \(t\) and \(\vartheta\). Hence, the radius at every fixed angle decreases strictly as \(t\) increases. The curves \(\psi_{N,t}(\mathbb T)\) are therefore strictly nested and, in particular, mutually disjoint.

The injectivity of \(\psi_{N,t}\) on \(\mathbb T\), together with the strict nesting of the curves, shows that \(\Psi_N\) is injective on \(\mathcal A_N\). Its image is \(\mathcal E_N\) by definition. Finally, admissibility of \(\eta_{N,t,j}\) ensures that \(\Psi_N\) is \(C^1\), while the preceding estimates give
\[J_{\Psi_N}(e^{-t}e^{i\theta})=e^{2t}\Re\Bigl(-e^{-i\theta}(\partial_t\psi_{N,t})(e^{i\theta})\overline{\psi_{N,t}'(e^{i\theta})}\Bigr)=1+\mathcal O(N^{-1})>0.\]
The inverse function theorem now shows that the inverse of \(\Psi_N\) is \(C^1\), which completes the proof.
\end{proof}

We turn now to the construction of the holomorphic approximation \(F_N\) of \(P_N\). We first define 
\begin{equation}\label{eq:def model f_N}
f_N(z) \coloneqq \sqrt{N+1}\exp \left({\textstyle \sum_{j=0}^\varkappa N^{-j}\nu_{N,j}(z)}\right),
\end{equation}
where \(\nu_{N,j}\) are assumed to be holomorphic on \(\mathbb{C}\setminus \{0\}\) and finite at infinity. More precisely, we assume that \(\nu_{N,j}=\nu_{N,0,j}\) for some functions \(\nu_{N,t,j}\in\sL^{\varkappa-j,\alpha}\).
We then define our model function \(F_N\coloneqq \Lambda_N[f_N]\) through the operator
\[\Lambda_N[f](z)\coloneqq \phi'_N(z)\phi_N^N(z) f(\phi_N(z)),\]
hence we have
\begin{equation}\label{eq:def F_N}
F_N = \sqrt{N+1} \phi_N'\phi_N^N \exp \left({\textstyle \sum_{j=0}^\varkappa N^{-j}\nu_{N,j}\circ \phi_N}\right).
\end{equation}
For large \(N\) and every \(\rho>r_N = e^{-\delta_N}\) for which \(\mathscr D_{N,\rho}\setminus \overline{\mathscr D_{N,r_N}}\) lies within the domain of definition of \(\omega\), this operator is an isometric isomorphism
\[\Lambda_N \colon L^2\bigl(\mathbb D_\rho\setminus \overline{\mathbb{D}_{r_N}},\tabs{\cdot}^{2N}(\omega\circ \phi^{-1}_N)\bigr)\to L^2\bigl(\mathscr D_{N,\rho}\setminus \overline{\mathscr D_{N,r_N}},\omega\bigr).\]
Indeed, a straightforward calculation reveals
\begin{align*}
\int_{\mathscr D_{N,\rho}\setminus \mathscr D_{N,r_N}}\abs{\Lambda_N[f]}^2 \omega\,\d\A&=\int_{\phi^{-1}_N(\mathbb D_\rho\setminus \mathbb D_{r_N})}\abs{\phi_N'}^2 \abs{\phi_N}^{2N} \abs{f\circ \phi_N}^2  \omega\,\d\A\\
&=\int_{\mathbb D_\rho\setminus\mathbb D_{r_N}} \abs{z}^{2N}\abs{f(z)}^2\omega(\phi^{-1}_N(z))\,\d\A(z).
\end{align*}
As in \eqref{eq:simple disintegration}, we can use the foliation \(\Psi_N\) to obtain the disintegration identity
\begin{equation}\label{eq: disintegration}
\int_{\mathscr E_N} \Lambda_N[q](z) \overline{F_N(z)}\omega(z)\,\d\A(z) = \int_{-\delta_N}^{\delta_N}\int_\mathbb T \frac{q(\psi_{N,t}(\zeta))}{f_N(\psi_{N,t}(\zeta))}\Upsilon_{N,t}(\zeta)\abs{\d\zeta}\d t,
\end{equation}
where \(q\) is a suitable function and \(\Upsilon_{N,t}\) denotes the density of the foliation flow, given by
\begin{equation}\label{eq: foliation flow}
\Upsilon_{N,t}(\zeta) = \mfrac{1}{\pi}\abs{f_N(\psi_{N,t}(\zeta))}^2\abs{\psi_{N,t}(\zeta)}^{2N} \omega(\phi_N^{-1}(\psi_{N,t}(\zeta)))\Re\Bigl(-\overline \zeta \partial_t\psi_{N,t}(\zeta) \overline{\psi'_{N,t}(\zeta)} \Bigr),
\end{equation}
for \(\zeta\in\mathbb T\). Our goal is an appropriate choice of \(\psi_{N,t}\) such  that for every \(\beta\in(0,\alpha)\)
\begin{equation}\label{eq: flow goal}
\Upsilon_{N,t}(\zeta) = \mfrac{1}{\pi}(N+1)e^{-2(N+1)t}\Bigl(1+\mathcal O\bigl(N^{-\varkappa-\beta}\bigr)\Bigr),\qquad \abs{t}\leq  \delta_N, \zeta\in\mathbb T.
\end{equation} 
As it turns out, it is more convenient to consider the logarithmic density. We recall that \(\omega\) is defined on a neighborhood of \(\overline{\mathscr D}\) and is strictly positive on \(\partial\mathscr D\). Since \(\mathscr E_N\) shrinks towards \(\partial\mathscr D\), the preceding lemma ensures that \(\Upsilon_{N,t}>0\) for all sufficiently large \(N\). We may therefore define
\begin{equation}\label{eq: app flow}
\Pi_{N,t}(\zeta)\coloneqq\log\Upsilon_{N,t}(\zeta)+2(N+1)t-\log(N+1)+\log(\pi).
\end{equation}
Instead of establishing \eqref{eq: flow goal} directly, we aim to choose the functions \(\eta_{N,t,j}\) and \(\nu_{N,j}\) such that, for every \(\beta\in(0,\alpha)\),
\begin{equation}\label{eq: app flow goal}
\Pi_{N,t}(\zeta)=\mathcal O(N^{-\varkappa-\beta}),\qquad \abs{t}\leq\delta_N,\quad \zeta\in\mathbb T,
\end{equation}
as \(N\to\infty\). Once this estimate is established, \eqref{eq: flow goal} follows immediately by exponentiating \(\Pi_{N,t}\). Thus, it remains to determine appropriate functions \(\eta_{N,t,j}\) and \(\nu_{N,j}\) for which \eqref{eq: app flow goal} holds.

\subsection[The Logarithmic Density for the Special Case: \texorpdfstring{\(\varkappa=0\)}{k=0}]{The Logarithmic Density for the Special Case: \texorpdfstring{\(\bm{\varkappa=0}\)}{k=0}}\label{sec:flow-r=0}

Before we tackle the general case it may be useful to first consider the simplest case. For now, let \(\varkappa=0\), then \(\psi_{N,t}(\zeta)= e^{-t}\zeta\) and \(f_N(z)= \sqrt{N+1} e^{\nu_{N,0}(z)}\) for some family \(\nu_{N,0}\in\sL^{0,\alpha}\) independent of \(t\). In fact, we have \(\nu_{N,0} = \nu_{N,0,0}\) for some \(\nu_{N,t,0}\in \sL^{0,\alpha}\).
Then using the notation \(u_{N,t}=u_{N,t,0,0} = \mathcal F_{L_N}[\log \omega \circ \phi^{-1}_N(e^{-t}\cdot)]\), as in \cref{prop:u_Nt-is-AL}, we obtain from \cref{theo:Bernstein0} 
\[\log(\omega(\phi_N^{-1}(e^{-t}\zeta)))-u_{N,t}(\zeta)=\mathcal O\biggl(\norm{\log(\omega\circ \phi_N^{-1}(e^{-t}\cdot))  }_{C^{0,\alpha}(\mathbb T)} \mfrac{(\log N)^{1+\alpha}}{N^{\alpha}}\biggr).\]
Using \cref{coro:uni-conv-Fourier}, we obtain
\begin{align*}
\norm{\log(\omega\circ\phi_N^{-1}(e^{-t}\cdot))}_{C^{0,\alpha}(\mathbb T)}
&\leq \norm{\log\omega}_{L^\infty(U)}+\mfrac{\pi^\alpha}{2^\alpha}e^{\alpha\abs{t}}\hnorm{\log\omega}_{U,\alpha}\norm{(\phi_N^{-1})'(e^{-t}\cdot)}_{L^\infty(\mathbb T)}^\alpha\\
&=\mathcal O(1),
\end{align*}
uniformly for \(\abs{t}\leq\delta_N\), where \(U\) is an open neighborhood of \(\partial\mathscr D\) on which \(\omega\) is bounded away from zero. We conclude
\[\log(\omega(\phi_N^{-1}(e^{-t}\zeta)))-u_{N,t}(\zeta) = \mathcal O(N^{-\beta}).\]
Using \eqref{eq:AL2}, the logarithmic density \eqref{eq: app flow} simplifies to
\begin{align*}
\Pi_{N,t}(\zeta)&=2\Re(\nu_{N,0}(e^{-t}\zeta))+ \log(\omega(\phi_N^{-1}(e^{-t}\zeta)))\\
&=2\Re(\nu_{N,0}(e^{-t}\zeta))+ u_{N,t}(\zeta)+\mathcal O(N^{-\beta})\\
&=2\Re(\nu_{0}(\zeta))+ u_{0}(\zeta)+\mathcal O(N^{-\beta}),
\end{align*}
where \(\nu_0\) and \(u_0\) denote, in the sense of \cref{def:AL-Families}, the limiting functions of \(\nu_{N,0}\) and \(u_{N,t}\), respectively. In particular, we have \(u_0=\log\omega\circ \phi^{-1}\).
We conclude that it remains to find \(\nu_{N,0}\) such that the real part of its limiting function \(\nu_0\) coincides with \(-u_0/2\) on \(\mathbb T\).

We recall the Herglotz transform presented in \cref{sec:herglotz}, which satisfies \(\Re\Herg[f]=f\) on \(\mathbb T\) for any real-valued \(f\in C^{0,\alpha}(\mathbb T)\). Hence, the natural choice is \(\nu_{0}\coloneqq -\frac{1}{2}\Herg[u_0]\) together with
\[\nu_{N,0}\coloneqq\mathcal F_{L_N}[\nu_0]=-\tfrac{1}{2}\mathcal F_{L_N}[\Herg[\log \omega \circ \phi^{-1} ]]\in \sL^{0,\alpha}.\]
Thus, the logarithmic density satisfies \(\Pi_{N,t}(\zeta) = \mathcal O(N^{-\beta})\), since \(2\Re(\nu_0(\zeta))=-u_0(\zeta)\). We may identify \(\nu_0\) with the Szeg\H{o} function associated with \(\omega\), expressed in the exterior disk coordinates. Indeed, if \(V\) denotes the Szeg\H{o} function normalized by \(2\Re V=-\log\omega\) on \(\partial\mathscr D\) and \(V(\infty)\in\mathbb R\), then \(\nu_0=V\circ\phi^{-1}\). Its appearance here is consistent with the historical development of the asymptotics of planar orthogonal polynomials described by Suetin and Hedenmalm--Wennman \cite{suetin1974polynomials,hedenmalm2024riemann}.

\subsection{Asymptotic Decomposition of the Logarithmic Density}\label{sec:asymp-flow-decomp}
In the following, let \(u_N=\log(\omega \circ \phi_N^{-1})\), which is well-defined on any annular neighborhood of \(\mathbb T\) with width \(\mathcal O(\tfrac{\log N}{N})\). Further, we consider for \(j=0,\dots,\varkappa\) the  families \(\eta_{N,t,j}\in\sL^{\varkappa+1-j,\alpha}\), which only admit  Fourier coefficients corresponding to non-positive indices. Consequently, these families are finite at infinity. In addition we demand that \(\eta_{N,t,0}(\zeta)=-t\) and \(\eta_{N,0,j}\equiv 0\) for all \(j\).
Inspired by the special case \(\varkappa=0\), we assume that \(\nu_{N,j}=\mathcal F_{L_N}[\nu_j]\) for some \(\nu_j\in C^{\varkappa-j,\alpha}(\mathbb T)\), where each \(\nu_j\) admits only Fourier coefficients with non-positive indices. With this choice, it follows that \(\nu_{N,j}(e^{-t}\,\cdot)\in\sL^{\varkappa-j,\alpha}\).

Using \eqref{eq:def foliation mapping} and \eqref{eq:def model f_N}  the logarithmic density \eqref{eq: app flow} reads as
\begin{equation}\label{eq: app flow2}
\begin{split}
\Pi_{N,t}(\zeta)=&\ 2t+2  \sum_{j=0}^{\varkappa-1}\frac{\Re \eta_{N,t,j+1}(\zeta)}{N^j}+2\sum_{j=0}^\varkappa \frac{\Re \nu_{N,j}( \psi_{N,t}(\zeta))}{N^j}\\
&+ u_N( \psi_{N,t}(\zeta)) + \log \Re \big(-\overline \zeta \partial_t\psi_{N,t}(\zeta)\overline{\psi'_{N,t}(\zeta)}\big).
\end{split}
\end{equation}
We recall \(u_{N,t,m,n}= \mathcal F_{L_N}[\partial ^m \bar \partial ^n[u_N](e^{-t}\cdot)]\). Let \(S_j\) be the set which contains the trivial Laurent families \(t\), \(e^{-t}\), \(\zeta\) and \(\bar{\zeta}\) as well as all of the following functions together with their complex conjugates,
\begin{center}
\begin{tabular}{lcl}
\(u_{N,t,m,n},\)&& \(m+n = 0,\dots, j,\) \\
\(\nu_{N,n}^{(m)}(e^{-t}\cdot),\)&&\(m+n = 0,\dots, j,\, m\geq 1,\)\\
\(\eta_{N,t,k},\eta_{N,t,k}',\partial_t\eta_{N,t,k},\)&&\( k=0,\dots,j.\)
\end{tabular}
\end{center}

We observe that \(S_j\subset\sL^{\varkappa-j,\alpha}\). The following theorem provides an asymptotic decomposition of the logarithmic density in \eqref{eq: app flow2} which will later allow us to determine the families \(\eta_{N,t,j}\) and \(\nu_{N,j}\). Before stating the result, we clarify two points. We call \(Q_{N,t}\) a \textit{polynomial expression} in a collection \(S\subset \sL^{k,\alpha}\) if there exist a multivariate polynomial \(Q\) and \(f_{N,t,1},\dots,f_{N,t,n}\in S\) such that \(Q_{N,t}=Q(f_{N,t,1},\dots,f_{N,t,n})\). Moreover, the sets \(S_j\) possess a triangular structure in the sense that \(S_j\) depends on \(\eta_{N,t,k}\) only for \(k\leq j\) and \(\nu_{N,k}\) only for \(k\leq j-1\). This structure will allow us to determine the desired families recursively in \cref{col: flow goal}.

\begin{theorem}\label{theo: asymp flow decomp}
Set \(\eta_{N,t,0}\coloneqq-t\). For \(j=1,\dots,\varkappa\), let \(\eta_{N,t,j}\in\sL^{\varkappa+1-j,\alpha}\) satisfy \(\eta_{N,0,j}\equiv0\), and for \(j=0,\dots,\varkappa\), let \(\nu_{N,j}\in\sL^{\varkappa-j,\alpha}\). Suppose that all these families admit only Fourier coefficients corresponding to non-positive indices. Then there exist polynomial expressions \(Q_{N,t,j}\) in functions belonging to \(S_j\) such that, for every \(\beta\in(0,\alpha)\),
\begin{align*}
\Pi_{N,t}(\zeta) &=    2\Re\sum_{j=0}^{\varkappa-1} \frac{ \eta_{N,t,j+1}(\zeta)+\nu_{N,j}(e^{-t}\zeta)+Q_{N,t,j}(\zeta)}{N^j}\\
&\quad +2\Re \frac{ \nu_{N,\varkappa}(e^{-t}\zeta)+ Q_{N,t,\varkappa}(\zeta)}{N^{\varkappa}}+\mathcal O(N^{-\varkappa-\beta}),
\end{align*}
uniformly for \(\abs{t}\leq \delta_N\). In particular, \(Q_{N,t,j}\in \sL^{\varkappa-j,\alpha}\).
\end{theorem}

\begin{proof}
We recall \(u_N = \log(\omega\circ \phi^{-1}_N)\) and  set \(\mathcal A_N'=\mathbb D_{e^{2\delta_N}} \setminus \overline{\mathbb D_{e^{-2\delta_N}}}\) for \(\delta_N=\frac{A\log N}{N}\) and some fixed \(A>0\). For sufficiently large \(N\) this annulus satisfies \(\mathcal E_N\subset \cA'_N\), since \(\tabs{\psi_{N,t}(\zeta)} = e^{-t+\mathcal O(N^{-1})}\) using the uniform boundedness of \(\eta_{N,t,j}\). Applying \cref{prop:u_Nt-is-AL} and using \cref{theo:Fourier-Master}, we conclude that \(u_N\) and \(\nu_{N,j}\) are well-defined and uniformly bounded on \(\mathcal A_N'\). Since \(\mathcal E_N\subset\mathcal A_N'\), the compositions \(u_N\circ\psi_{N,t}\) and \(\nu_{N,j}\circ\psi_{N,t}\) are well-defined on \(\mathbb T\) for \(\abs{t}\leq\delta_N\).

Let us begin by finding a decomposition of \(u_N\circ \psi_{N,t}\) in terms of admissible Laurent families. Using the notation \(v_{N,t}(\zeta)\coloneqq \psi_{N,t}(\zeta)-e^{-t}\zeta\) we get  by applying Taylor's theorem at the point \(e^{-t}\zeta\)
\[
u_N(\psi_{N,t}(\zeta)) =\sum_{m+n =0}^\varkappa\frac{\partial^{m}\bar\partial^{n}u_N(e^{-t}\zeta)}{m!n!} v^m_{N,t}(\zeta)\,\overline{v^n_{N,t}(\zeta)}+R_{\varkappa}(e^{-t}\zeta ,\psi_{N,t}(\zeta)),
\]
where \(R_{\varkappa}(e^{-t}\zeta ,\cdot)\) denotes  the corresponding error term associated to the base point \(e^{-t}\zeta\). If \(\varkappa = 0\), then the remainder vanishes. For \(\varkappa >0\) we have
\begin{multline*}
R_{\varkappa}(e^{-t}\zeta ,\psi_{N,t}(\zeta))
=\sum_{m+n=\varkappa}\frac{ \varkappa \,v_{N,t}^m(\zeta)\,\overline{v_{N,t}^n(\zeta)}}{m!n!}\\\times \int_0^1\frac{\partial^{m}\bar\partial^{n}u_N(e^{-t}\zeta+xv_{N,t}(\zeta))-\partial^{m}\bar\partial^{n}u_N(e^{-t}\zeta)}{(1-x)^{1-\varkappa}}\d x.
\end{multline*}
Thus, making use of the Hölder condition of \(\partial^{m}\bar\partial^{n}u_N\) we obtain for some \(c>0\)
\begin{align*}
\abs{R_{\varkappa}(e^{-t}\zeta ,\psi_{N,t}(\zeta))}&\leq  \sum_{m+n=\varkappa}\frac{\varkappa\abs{ v_{N,t}(\zeta)}^{m+n}}{m!n!} \hnorm{\partial^{m}\bar\partial^{n}u_N}_{\cA_N',\alpha}\int_0^1\frac{x^\alpha\abs{v_{N,t}(\zeta)}^\alpha}{(1-x)^{1-\varkappa}}\d x\\
&\leq  c\varkappa\abs{v_{N,t}(\zeta)}^{\varkappa+\alpha} \norm{u_N}_{C^{\varkappa,\alpha}(\mathcal A_N')}\sum_{m+n=\varkappa} \frac{1}{m!n!}\\
&=\mathcal O(N^{-\varkappa-\alpha}),
\end{align*}
since \(v_{N,t}(\zeta)=\mathcal O(N^{-1})\) uniformly for \(\zeta\in\mathbb T\) and \(\abs{t}\leq\delta_N\), while \(\tnorm{u_N}_{C^{\varkappa,\alpha}(\mathcal A_N')}=\mathcal O(1)\) by \cref{prop:u_Nt-is-AL}.
Hence, we observe 
\[u_N (\psi_{N,t}(\zeta))=\sum_{m+n =0}^\varkappa\frac{\partial^{m}\bar\partial^{n}u_N(e^{-t}\zeta) }{m!n!}v^m_{N,t}(\zeta)\,\overline{v^n_{N,t}(\zeta)}+\mathcal O(N^{-\varkappa-\alpha}).\]
Expanding further we get
\begin{align*}
v^m_{N,t}(\zeta)&=(\psi_{N,t}(\zeta)-e^{-t}\zeta)^m\\
&= e^{-mt}\zeta^m \biggl(e^{\sum_{j=1}^{\varkappa}N^{-j}{\eta_{N,t,j}(\zeta)}}-1\biggr)^m\\
&=e^{-mt}\zeta^m{\textstyle\biggl(\sum_{l=1}^\varkappa \frac{1}{l !} \Bigl( \sum_{j=1}^\varkappa N^{-j}{\eta_{N,t,j}(\zeta)}\Bigr)^l\biggr)^m}+\mathcal O(N^{-\varkappa-1})\\
&=e^{-mt}\zeta^m\sum_{j=m}^\varkappa N^{-j} \tilde F_{N,t,j,m}(\zeta)+O(N^{-\varkappa-1}).
\end{align*}
For \(m=0\), the expansion is understood with \(\tilde F_{N,t,0,0}=1\) and \(\tilde F_{N,t,j,0}=0\) for \(j\geq1\). For \(m\geq1\), the function \(\tilde F_{N,t,j,m}\) is a polynomial expression in \(\{\eta_{N,t,1},\dots,\eta_{N,t,j+1-m}\}\subset S_j\). In a similar way we can proceed for the conjugated term and get
\[\overline{v_{N,t}^n(\zeta)}=e^{-nt}\zeta^{-n}\sum_{j=n}^\varkappa N^{-j} \overline{\tilde F_{N,t,j,n}}(\zeta)+O(N^{-\varkappa-1}).\]
According to \cref{prop:u_Nt-is-AL} we know that \(u_{N,t,m,n}= \mathcal F_{L_N}[\partial^{m,n}u_N(e^{-t}\cdot)]\) is an \(\sL^{\varkappa-m-n,\alpha}\) family. Further, for any \(\beta\in(0,\alpha)\), we have
\begin{align*}
\partial^{m}\bar\partial^{n}u_N(e^{-t}\cdot)&=\mathcal F_{L_N}[\partial^{m}\bar\partial^{n}u_N(e^{-t}\cdot)]+\mathcal O\biggl (\norm{\partial ^{m}\bar\partial^{n}u_N(e^{-t}\cdot)}_{C^{\varkappa-m-n,\alpha}(\mathbb T)}  N^{m+n-\varkappa-\beta}\biggr)\\
&=u_{N,t,m,n} +\mathcal O(N^{m+n-\varkappa-\beta})
\end{align*}
while the asymptotic error is uniform in \(t\) due to \cref{theo:Bernstein0}, since 
\[\norm{\partial ^{m,n}u_N(e^{-t}\cdot)}_{C^{\varkappa-m-n,\alpha}(\mathbb T)}\leq \mathcal O\Bigl( e^{\alpha\abs{t}}\norm{u_N}_{C^{\varkappa,\alpha}(\mathcal A_N')}\Bigr)=\mathcal O(1)  .\]
Combining these results yields
\[u_N(\psi_{N,t}(\zeta))=u_{N,t,0,0}(\zeta)+\sum_{j=1}^\varkappa N^{-j}F_{N,t,j}(\zeta)+\mathcal O(N^{-\varkappa-\beta}),\]
where \(F_{N,t,j}\) is a polynomial expression in
\[\left.\left\{e^{-t},\zeta,\overline{\zeta},u_{N,t,m,n}(\zeta),\eta_{N,t,k}(\zeta),\overline{\eta_{N,t,k}}(\zeta)\,\right|\,k,m+n\in\{1,\dots,j\}\right\}\subset S_j.\]

We need to do similar computations for \(N^{-j}\nu_{N,j}\circ \psi_{N,t}\) and the \(\log\)-term in \eqref{eq: app flow2}. Let us begin with the former expression. By applying Taylor's theorem we obtain
\[N^{-j}\nu_{N,j}(\psi_{N,t}(\zeta))=N^{-j}\sum_{k=0}^{\varkappa-j}\frac{\nu_{N,j}^{(k)}(e^{-t}\zeta)}{k!} v_{N,t}^k(\zeta)+\mathcal{O}(N^{-\varkappa-\beta}),\]
while the asymptotic error results from estimating the corresponding remainder analogously to \(R_\varkappa\).

Now let \(\tilde F_{N,t,j,k}\) denote the same polynomial expression as before, then we get
\begin{align*}
N^{-j} \nu_{N,j}&(\psi_{N,t}(\zeta))\\
&= N^{-j}\nu_{N,j}(e^{-t}\zeta)+ N^{-j}\sum_{k=1}^{\varkappa-j}\frac{\nu_{N,j}^{(k)}(e^{-t}\zeta)}{k!} \zeta^ke^{-kt}\sum_{l=k}^{\varkappa-j} N^{-l} \tilde F_{N,t,l,k}(\zeta)+\mathcal O(N^{-\varkappa-\beta})\\
&=N^{-j} \nu_{N,j}(e^{-t}\zeta)+ \sum_{l=j+1}^{\varkappa} N^{-l} G_{N,t,j,l}(\zeta)+\mathcal O(N^{-\varkappa-\beta}),
\end{align*}
where \(G_{N,t,j,l}\) is a polynomial expression in
\[\left.\left\{ e^{-t},\zeta,\nu_{N,j}^{(k)}(e^{-t}\zeta),\eta_{N,t,k}(\zeta)\,\right|\, 1\leq k \leq l-j \right\}\subset S_l.\]
It remains to decompose the \(\log\)-term of \eqref{eq: app flow2} in an analogous way. We observe first
\begin{align*}
\psi'_{N,t}(\zeta)&={\textstyle\left(1+\zeta\sum_{j=1}^{\varkappa} N^{-j} \eta'_{N,t,j}(\zeta)\right)}e^{-t +\sum_{j=1}^{\varkappa} N^{-j} \eta_{N,t,j}(\zeta) },\\
\partial_t\psi_{N,t}(\zeta)&=\zeta\left({\textstyle -1+\sum_{j=1}^{\varkappa}N^{-j}\partial_t\eta_{N,t,j}(\zeta)}\right)e^{-t +\sum_{j=1}^{\varkappa} N^{-j} \eta_{N,t,j}(\zeta) }.
\end{align*}
The \(\log\)-term can now be written as
\begin{multline*}
\log \Re \big(-\overline \zeta \partial_t\psi_{N,t}(\zeta)\overline{\psi'_{N,t}(\zeta)}\big)=-2t+2\sum_{j=1}^{\varkappa} N^{-j}\Re \eta_{N,t,j}(\zeta)\\
{} +\log \Re \Bigl[{\textstyle\left( 1-\sum_{j=1}^{\varkappa}N^{-j}\partial_t\eta_{N,t,j}(\zeta)\right)\left(1+\overline{\zeta}\sum_{j=1}^{\varkappa} N^{-j} \overline{\eta'_{N,t,j}}(\zeta)\right)}\Bigr].
\end{multline*}
Notice that the term \(-2t\) cancels with the term \(2t\) in \eqref{eq: app flow2}.
By employing another Taylor expansion, the remaining logarithm can be rewritten as
\begin{align*}
\log \Re \bigl[\dots\bigr]&=\log \left({\textstyle 1+\sum_{j=1}^{\varkappa} N^{-j}\Re \tilde H_{N,t,j}(\zeta)}\right)+\mathcal O(N^{-\varkappa-1})\\
&= \sum_{j=1}^{\varkappa} 
N^{-j}H_{N,t,j}(\zeta)+\mathcal O(N^{-\varkappa-1}),
\end{align*}
where \(\tilde H_{N,t,j}\) and \(H_{N,t,j}\) are  polynomial expressions in
\[\left.\left\{\zeta,\bar\zeta,\Re \eta'_{N,t,k},\Im\eta'_{N,t,k},\Re \partial_t\eta_{N,t,k},\Im\partial_t\eta_{N,t,k}\, \right|\, 1\leq k\leq j\right\}\subset S_j.\]
The claimed result follows now by setting
\begin{align*}
Q_{N,t,j} &\coloneqq \frac12 F_{N,t,j}+\frac 12H_{N,t,j}+\eta_{N,t,j}+{\textstyle\sum_{k=0}^{j-1}G_{N,t,k,j}},\qquad j=1,\dots,\varkappa,
\end{align*}
while \(Q_{N,t,0}\coloneqq \tfrac 12 u_{N,t,0,0} \).
This completes the proof.
\end{proof}

\begin{corollary}\label{col: flow goal}
Let \(\varkappa\in \mathbb Z_{\geq 0}\) and \(\alpha\in(0,1)\). Let \(\mathscr D\) be a Jordan domain with \(C^{\varkappa+1,\alpha}\)-smooth boundary, and let \(\omega\) be a non-negative \(C^{\varkappa,\alpha}\)-smooth weight function defined in a neighborhood of \(\overline{\mathscr D}\), with \(\omega>0\) on \(\partial\mathscr D\). Then, for each \(j=0,\dots,\varkappa\), there exist
\[\eta_{N,t,j}\in \sL^{\varkappa+1-j,\alpha}\qquad\text{and}\qquad\nu_{N,j}\in \sL^{\varkappa-j,\alpha},\]
such that, for every \(A>0\) and \(\beta\in(0,\alpha)\), the density \(\Upsilon_{N,t}\) satisfies uniformly
\begin{equation}\label{eq:coro-flow}
\Upsilon_{N,t}(\zeta)=\mfrac{1}{\pi}(N+1) e^{-2(N+1)t} (1+\mathcal O(N^{-\varkappa-\beta})),
\qquad\zeta\in\mathbb T,\,\abs{t}\leq \mfrac{A\log N}{N}.
\end{equation}
Moreover, \(\eta_{N,t,j}\) and \(\nu_{N,j}\) contain only Fourier coefficients corresponding to non-positive indices, \(\eta_{N,0,j}\equiv 0\), and
\[\nu_{N,j}=\mathcal F_{L_N}[\nu_j]\]
for some \(\nu_j\in C^{\varkappa-j,\alpha}(\mathbb T)\).
\end{corollary}

\begin{proof}
Let \(Q_j\) denote the limiting functions of the Laurent families \(Q_{N,t,j}\) constructed in \cref{theo: asymp flow decomp}. Then we set \(\eta_{N,t,0}\coloneqq -t\) and successively define
\begin{align*}
\nu_{N,j}& \coloneqq -\mathcal F_{L_N}[\Herg[\Re Q_j]],&& j=0,\dots, \varkappa,\\
\eta_{N,t,j+1} &\coloneqq \Herg[\Re Q_{N,0,j}](e^{-t}\cdot)-\Herg[\Re Q_{N,t,j}],&&j=0,\dots,\varkappa-1.
\end{align*}
We then have \(\eta_{N,0,j}\equiv 0\) and \( \nu_{N,j},\eta_{N,t,j+1}\in \sL^{\varkappa-j,\alpha}\), since by construction \(Q_{N,t,j}\in\sL^{\varkappa-j,\alpha}\). Note that the functions \(\eta_{N,t,j+1}\) and \(\nu_{N,j}\) are well-defined, since each \(Q_{N,t,j}\) only depends on the previous data \(\eta_{N,t,k}\) for \(k=0,\dots,j\) and \(\nu_{N,\ell}\) for  \(\ell=0,\dots,j-1\). Using \cref{theo:Fourier-Master} we observe for \(j<\varkappa\)
\begin{align*}
\Re\left(\eta_{N,t,j+1}(\zeta)+\nu_{N,j}(e^{-t}\zeta)+Q_{N,t,j}(\zeta)\right)&=\Re\left(\left(\Herg[\Re Q_{N,0,j}]-\mathcal F_{L_N}[\Herg[\Re Q_j]]\right)(e^{-t}\zeta)\right)\\
&=\mathcal O(N^{j-\varkappa-\beta}).
\end{align*}
and using \eqref{eq:AL1} together with \eqref{eq:AL2}, we get
\[\Re\left(\nu_{N,\varkappa}(e^{-t}\zeta)+Q_{N,t,\varkappa}(\zeta)\right)=\Re\left(\nu_{N,\varkappa}(\zeta)+Q_{N,0,\varkappa}(\zeta)\right)+\mathcal O(N^{-\beta})=\mathcal O(N^{-\beta}).\]
We conclude that \(\Pi_{N,t}(\zeta)=\mathcal O(N^{-\varkappa-\beta})\) for every \(\beta\in (0,\alpha)\), and hence \eqref{eq:coro-flow} follows directly from \eqref{eq: app flow}.
\end{proof}

\section{Quasi-Orthogonality and Strong Asymptotics}\label{sec:strong-asymptotics}
Now that we have constructed the model function \(F_N\) and the foliation, we turn to the orthogonal polynomial \(P_N\). Using the density estimate and the disintegration formula, we first prove that \(F_N\) is approximately normalized and quasi-orthogonal on a shrinking boundary layer. Since \(F_N\) is not a polynomial, we then construct a degree-\(N\) Faber polynomial \(B_N\) that approximates \(F_N\) and inherits suitable quasi-orthogonality estimates on the whole domain. These estimates allow us to compare \(B_N\) with \(P_N\) in \(L^2(\mathscr D,\omega)\) and subsequently derive the pointwise asymptotic expansions in the exterior and in a shrinking boundary layer. We begin by establishing two polynomial norm estimates that will be used both in the quasi-orthogonality argument and in the derivation of the pointwise expansions.

\begin{lemma}\label{lem: Bernstein-Walsh alt}
Let \(\mathscr D\) be a Jordan domain with \(C^{1,\alpha}\)-smooth boundary, and \(\omega\) a continuous non-negative function on \(\overline{\mathscr D}\) such that \(\omega>0\) on \(\partial\mathscr D\). Then there exists \(c>0\) such that for any polynomial \(q\)  we have
\begin{equation}\label{eq:lem poly est1}
\max_{z\in \overline{\mathscr D}} \abs{q(z)}\leq c (\deg(q)+1)\norm{q}_{L^2(\mathscr D,\omega)},
\end{equation}
as well as
\begin{equation}\label{eq:lem poly est2}
\int _{\partial \mathscr D} \abs{q(\zeta)}^2 \abs{\d\zeta}\leq c (\deg(q)+1) \norm{q}^2_{L^2(\mathscr D,\omega)}.
\end{equation}
\end{lemma}
\begin{proof}
The estimates \eqref{eq:lem poly est1} and \eqref{eq:lem poly est2} are well-known in the special case \(\omega\equiv1\), see \cite{suetin1966ordercomparison} and \cite[Eqs.~(2.29)--(2.30)]{suetin1974polynomials}. Let \(\psi\colon\mathbb D\to\mathscr D\) be a conformal map. By boundary regularity and the positivity of \(\omega\) on \(\partial\mathscr D\), there exist \(\rho\in(0,1)\) and \(c>0\) such that \(\omega(\psi(z))\tabs{\psi'(z)}^2\geq c\) for \(\rho\leq\tabs{z}\leq1\). Expanding \(q\circ\psi\) into its Taylor series and changing variables gives
\[\int_{\mathscr D}\abs{q}^2\d\A\leq\frac{\norm{\psi'}_{L^\infty(\mathbb D)}^2}{c(1-\rho^2)}\int_{\mathscr D}\abs{q}^2\omega\,\d\A.\]
Combining this inequality with the unweighted estimates completes the proof.
\end{proof}

\subsection{Quasi-Orthogonality Near the Boundary} \label{sec:quasi-ortho}
It turns out that the foliated domain \(\mathscr E_N=\phi_N^{-1}(\mathcal E_N)\) is too large to obtain a useful bound on the \(L^2\)-norm of \(F_N\).  Instead, we introduce the smaller domain
\begin{equation}\label{eq:def-E_N'}
\mathscr E_N'\coloneqq\phi_N^{-1}(\mathcal E_N'),
\qquad
\mathcal E_N'\coloneqq\bigcup_{t\in(-\delta_N',\delta_N)}\psi_{N,t}(\mathbb T),
\end{equation}
where
\begin{equation}\label{eq:def-delta'}
\delta_N' = \frac{(\log N)^{\varkappa+3+\alpha}}{N^{\varkappa+1+\alpha}}.
\end{equation}
Essentially, we are only interested in the curves for \(t\geq 0\), however, in order to include the boundary \(\partial \mathscr D\) we need to consider \(t\) also for small negative values.
The choice of \(\delta_N'\) precisely ensures that \(\partial\mathscr D\subset\mathscr E_N'\) while keeping \(\mathscr E_N'\) sufficiently small. This property is useful in the proofs of \cref{prop: quasi ortho,prop: quasi ortho Faber} when treating the sets \(\mathscr D\cap\mathscr E_N'\) and \(\mathscr E_N'\setminus\mathscr D\).

\begin{lemma}\label{lem:inclusion-of-domains}
Let \(\mathscr E_N'\) and \(\delta_N'\) be as in \eqref{eq:def-E_N'} and \eqref{eq:def-delta'}, and set \(\mathscr A_N \coloneqq \mathscr D_{R_N}\setminus \mathscr D\) for \(\mathscr D_{R_N}\coloneqq \mathrm{Int}(\phi^{-1}(\mathbb T_{R_N}))\) and \(R_N = e ^{2 \delta_N'}\). Then, for all sufficiently large \(N\), we have \(\partial\mathscr D\subset\mathscr E_N'\) and \(\mathscr E_N'\setminus\mathscr D\subset\mathscr A_N\).
\end{lemma}

\begin{proof}
As before, let \(\delta_N = \frac{A\log N}{N}\) for some \(A>0\). Further, denote by \(\eta_{N,t,j}\in \sL^{\varkappa+1-j,\alpha}\) the components of the foliation \(\psi_{N,t}\) as defined in \eqref{eq:def foliation mapping}. We recall that \(\eta_{N,0,j}\equiv 0\), and therefore we have
\[\eta_{N,t,j}(\zeta) = \int_0^t \eta^{[1]}_{N,x,j}(\zeta )\d x = t \int _0^1 \eta_{N,tx,j}^{[1]}(\zeta)\d x = \mathcal O(t),\qquad \abs{t}\leq \delta_N,\]
where \(\eta_{N,t,j}^{[1]}(\zeta)\coloneqq\partial_t \eta_{N,t,j}(\zeta)\). Note that we have used the uniform boundedness of the admissible Laurent families. Consequently it follows
\[\psi_{N,t}(\zeta) = e^{-t}\zeta+ \mathcal O(\abs tN^{-1}),\qquad \abs{t}\leq \delta_N.\]
Using the uniform boundedness of \((\phi_N^{-1})'\) on the corresponding annular region, Taylor's theorem provides for \(\abs{t}\leq	 \delta_N\)
\[\phi_N^{-1}(\psi_{N,t}(\zeta)) = \phi_N^{-1}(e^{-t}\zeta)+\mathcal O(\abs tN^{-1}).\]
Further, applying the Hölder estimate we observe 
\begin{align*}
\phi_N^{-1}(e^{-t}\zeta)&=\phi_N^{-1}(\zeta) -t\zeta(\phi_N^{-1})'(\zeta) - \zeta\int_0^t e^{-x} (\phi_N^{-1})'(e^{-x}\zeta)- (\phi_N^{-1})'(\zeta)\d x\\
&=\phi_N^{-1}(\zeta) -t\zeta(\phi_N^{-1})'(\zeta)+ \mathcal O\biggl(\norm{(\phi_N^{-1})'}_{C^{0,\alpha/2}(\cA_N)}\abs{t}^{1+\alpha/2}\biggr),
\end{align*}
where \(\cA_N= \{z\,|\,\tabs{z}\in(e^{-\delta_N},e^{\delta_N})\}\). Note that \cref{theo:Fourier-Master} ensures that the Hölder norm is uniformly bounded, hence we conclude
\begin{align*}
\phi_N^{-1}(\psi_{N,t}(\zeta)) &= \phi_N^{-1}(e^{-t}\zeta)+\mathcal O(\abs tN^{-1})\\
&=\phi_N^{-1}(\zeta) -t\zeta(\phi_N^{-1})'(\zeta)+ \mathcal O(\abs t^{1+\alpha/2}+\abs tN^{-1})\\
&=\phi^{-1}(\zeta) -t\zeta(\phi^{-1})'(\zeta)\\
&\quad+ \mathcal O\biggl(\abs{t}^{1+\alpha/2}+\abs{t}N^{-1}+\abs t\frac{(\log N)^{\varkappa+1+\alpha}}{N^{\varkappa+\alpha}}+\frac{(\log N)^{\varkappa+2+\alpha}}{N^{\varkappa+1+\alpha}}\biggr).
\end{align*}
We identify \(\zeta(\phi^{-1})'(\zeta)\) as the outward normal vector of \(\phi^{-1}\) at \(\zeta\in\mathbb T\). Moreover, for \(t\in\{\delta_N,-\delta_N'\}\), the asymptotic error is negligible compared with \(\abs{t\zeta(\phi^{-1})'(\zeta)}\). Indeed, \((\phi^{-1})'\) does not vanish on \(\mathbb T\) and is therefore uniformly bounded away from zero there. Consequently, for sufficiently large \(N\)
\[\phi_N^{-1}(\psi_{N,\delta_N}(\mathbb T))\subset \mathscr D,\qquad \phi_N^{-1}(\psi_{N,-\delta_N'}(\mathbb T))\subset \mathscr D^c.\]
Thus, the two boundary curves lie on opposite sides of \(\partial\mathscr D\). By continuity of the foliation, it follows that \(\partial\mathscr D\subset\mathscr E_N'\). Similarly, we obtain \(\mathscr E_N'\setminus\mathscr D\subset\mathscr A_N\).
\end{proof}

\begin{proposition}\label{prop: quasi ortho}
Let \(A\geq \frac{\varkappa+\alpha}{2}\) and let \(\mathscr E_N'\) be as in \eqref{eq:def-E_N'}. Then for every \(\beta \in (0,\alpha)\) we have \(\tnorm{ F_N}_{L^2(\mathscr D\cap\mathscr E_N',\omega)}=1+\mathcal O(N^{-\varkappa-\beta})\) and
\[\int_{\mathscr D\cap\mathscr E_N'} q(z)\overline{F_N(z)}\omega(z)\d\A(z)=\mathcal{O}\left(N^{-\varkappa-\beta}\norm{ q}_{L^2( \mathscr D,\omega)}\right),\quad q\in \mathrm{Pol}_{N-1},\ N\to\infty.\]
\end{proposition}

\begin{proof}
We follow the ideas of \cite[Prop. 5.1.1]{hedenmalm2024riemann}. 
We begin by analyzing \(\tnorm{F_N}_{L^2(\mathscr D\cap \mathscr E_N',\omega)}\) by splitting the norm
\[\norm{F_N}^2_{L^2(\mathscr D\cap \mathscr E_N',\omega)}=\norm{F_N}^2_{L^2(\mathscr E_N',\omega)}-\norm{F_N}^2_{L^2(\mathscr E_N'\setminus \mathscr D,\omega)}.\]
Using the disintegration identity \eqref{eq: disintegration} and the isometric property of \(\Lambda_N\) we observe for any \(\beta\in(0,\alpha)\)
\begin{align*}
\norm{F_N}_{L^2(\mathscr E_N',\omega)}^2&= \int_{\mathscr E_N'}\abs{F_N(z)}^2\omega(z)\,\d \A(z)\\
&=\int_{-\delta_N'}^{\delta_N}\int_\mathbb T\Upsilon_{N,t}(\zeta)\abs{\d\zeta}\d t\\
&=2(N+1) \int_{-\delta_N'}^{\delta_N} e^{-2(N+1)t}\d t\left( 1 +\mathcal O(N^{-\varkappa-\beta})\right)\\
&=\left(e^{\frac{2(N+1)(\log N)^{\varkappa+3+\alpha}}{N^{\varkappa+1+\alpha}}}- e^{-\frac{2A(N+1)\log N}{N}}\right)\left(1+\mathcal O(N^{-\varkappa-\beta})\right)\\
&=1+\mathcal O(N^{-\varkappa-\beta}),
\end{align*}
where we have used \(A\geq \frac{\varkappa+\alpha}{2}\) and that due to \cref{col: flow goal} the density of the flow satisfies \(\Upsilon_{N,t}(\zeta)=\pi^{-1}(N+1) e^{-2(N+1)t}(1+\mathcal O(N^{-\varkappa-\beta}))\). Further, let \(\mathscr A_N = \mathscr D_{R_N}\setminus \mathscr D\) for \(R_N = e ^{2 \delta_N'}\), then  \cref{lem:inclusion-of-domains} ensures that for sufficiently large \(N\) we have \(\mathscr E_N'\setminus\mathscr D\subset \mathscr A_N\) and \(\mathrm{Area}(\mathscr A_N) = \mathcal O(N^{-\varkappa-1-\beta})\). In addition, \(F_N\) satisfies \(F_N=\mathcal O(\sqrt{N})\) uniformly on \(\mathscr A_N\) as \(N\to\infty\). We conclude \(\tnorm{F_N}^2_{L^2(\mathscr E_N'\setminus \mathscr D,\omega)}=\mathcal O(N^{-\varkappa-\beta})\), which proves the first statement of the proposition.

Turning now to the orthogonality condition, we again begin by splitting
\begin{align*}
\int_{\mathscr  D\cap \mathscr E_N'}q(z)\overline{F_N(z)}\omega(z)\,\d\A(z)&=\int_{\mathscr E_N'}q(z)\overline{F_N(z)}\omega(z)\,\d\A(z)-\int_{\mathscr E_N'\setminus \mathscr D}q(z)\overline{F_N(z)}\omega(z)\,\d\A(z)\\
& \eqqcolon I_1 - I_2.
\end{align*}
Let us focus on \(I_2\) first. Clearly, 
\[\bigg|\int_{\mathscr E_N'\setminus\mathscr D}q(z)\overline {F_N(z)} \omega(z)\,\d\A(z)\bigg|\leq \int_{\mathscr E_N'\setminus\mathscr D}\abs{q(z)\overline {F_N(z)} \omega(z)}\d\A(z).\]
Let \(\Gamma_t = \phi^{-1}(\mathbb T_t)\) and \(c_\infty=\tnorm{(\phi^{-1})'}_{L^\infty(\mathbb T)}\), then for some \(c>0\) 
\begin{align*}
\int_{\mathscr E_N'\setminus\mathscr D}\abs{q\overline {F_N} \omega}\d\A&\leq c \sqrt{N}\int _{\mathscr A_N} \abs{q} \d\A\\
&=\frac{c\sqrt{N}}{\pi}\int_1^{R_N} \int_{\Gamma_{t}} \frac{\abs{q(z)}}{\abs{\phi'(z)}}\abs{\d z}\d t\\
& \leq \frac{c c_\infty\sqrt{N}}{\pi}\int_1 ^{R_N} \norm{q}_{L^2(\Gamma_{t})}\d t,
\end{align*}
where we have used the Cauchy-Schwarz inequality together with \(\tnorm{1}_{L^2(\Gamma_t)}=\mathcal O(1)\) as \(N\to\infty\). Further, we can rewrite \(q(\phi^{-1}(z))=\sum_{k=-\infty}^{N-1}a_k z^k\) and observe for some \(c'>0\)
\begin{align*}
\int_{\Gamma_t} \abs{q(z)}^2 \abs{\d z}&\leq c_\infty t \int _\mathbb T \abs{q(\phi^{-1}(tz))}^2 \abs{\d z}\\
&= 2\pi c_\infty \sum_{k = -\infty}^{N-1}\abs{a_k}^2t^{2k+1}\leq c' \norm{q}_{L^2(\partial \mathscr D)}^2,
\end{align*}
since \(t^{2k+1}\leq (R_N)^{2N-1}=\mathcal O(1)\). Combining previous considerations with \cref{lem: Bernstein-Walsh alt} and \(R_N-1=\mathcal O(N^{-\varkappa-1-\beta})\) yields for some \(c''>0\)
\begin{equation}\label{eq:prf-quasi-ortho-F0}
\int_{\mathscr E_N'\setminus\mathscr D}\abs{q\overline {F_N} \omega}\d\A\leq \frac{cc_\infty\sqrt{N}}{\pi}\int_1 ^{R_N} \norm{q}_{L^2(\Gamma_{t})} \d t\leq c''\norm{q}_{L^2(\mathscr D,\omega)} N^{-\varkappa-\beta},
\end{equation}
and thus \(I_2 = \mathcal O(\tnorm{q}_{L^2(\mathscr D,\omega)} N^{-\varkappa-\beta})\).

We turn now to \(I_1\). Using the disintegration formula \eqref{eq: disintegration}, we get
\begin{align*}
I_1&=\int_{-\delta_N'}^{\delta_N}\int_\mathbb T\frac{\Lambda_N^{-1}[q](\psi_{N,t}(\zeta))}{f_N(\psi_{N,t}(\zeta))}\Upsilon_{N,t}(\zeta)\abs{\d\zeta}\d t\\
&=\int_{-\delta_N'}^{\delta_N}\int_\mathbb T\frac{\Lambda_N^{-1}[q](\psi_{N,t}(\zeta))}{f_N(\psi_{N,t}(\zeta))}\left(\Upsilon_{N,t}(\zeta)-\pi^{-1}(N+1)e^{-2(N+1)t}\right)\abs{\d\zeta}\d t
\end{align*}
where we have used that \(\Lambda_N^{-1}[q]\) is holomorphic in a neighborhood of \(\mathbb D^c\) and vanishes at infinity. Further, by construction \(f_N\) does  not have any roots and it admits a finite non-zero value at infinity. Therefore, by the Cauchy integral formula
\[\frac{1}{2\pi}\int_\mathbb T\frac{\Lambda_N^{-1}[q](\psi_{N,t}(\zeta))}{f_N(\psi_{N,t}(\zeta))}\abs{\d\zeta}=\lim_{z\to \infty}\frac{\Lambda_N^{-1}[q](z)}{f_N(z)}=0.\]
For \(G_{N,t}=\Upsilon_{N,t}-\pi^{-1}(N+1) e^{-2(N+1)t}\) we have \(G_{N,t}(\zeta)=\Upsilon_{N,t}(\zeta)\cdot \mathcal O(N^{-\varkappa-\beta})\) since
\begin{align*}
\frac{G_{N,t}(\zeta)}{\Upsilon_{N,t}(\zeta)}&=1-\frac{\pi^{-1}(N+1)e^{-2(N+1)t}}{\Upsilon_{N,t}(\zeta)}=1-\frac{1}{1+\mathcal O(N^{-\varkappa-\beta})}=\mathcal O(N^{-\varkappa-\beta}).
\end{align*}
Thus, using again \eqref{eq: disintegration} we have for a suitable constant \(C>0\) and \(Q_N\coloneqq \Bigl|\frac{\Lambda_N^{-1}[q]}{f_N}\Bigr|f_N\)
\begin{align*}
\bigg|\int_{\mathscr E_N'}q(z)\overline{F_N(z)}\omega(z)\,\d\A(z)\bigg| &= \bigg|\int_{-\delta_N'}^{\delta_N}\int_\mathbb T\frac{\Lambda_N^{-1}[q](\psi_{N,t}(\zeta))}{f_N(\psi_{N,t}(\zeta))}G_{N,t}(\zeta)\abs{\d\zeta}\d t\bigg|\\
&\leq CN^{-\varkappa-\beta} \int_{-\delta_N'}^{\delta_N}\int_\mathbb T  \frac{Q_N(\psi_{N,t}(\zeta))}{f_N(\psi_{N,t}(\zeta))}\Upsilon_{N,t}(\zeta)\abs{\d\zeta}\d t\\
&=   C N^{-\varkappa-\beta}\int_{\mathscr E_N'}\Lambda_N[Q_N](z) \overline{F_N(z)}\omega(z)\,\d\A(z)\\
&= C N^{-\varkappa-\beta}\int_{\mathscr E_N'} \abs{\Lambda_N^{-1}[q](\phi_N(z))} \frac{\abs{F_N(z)}^2}{\abs{f_N(\phi_N(z))}}\omega(z)\, \d \A(z)  \\
&= C N^{-\varkappa-\beta}\int_{\mathscr E_N'} \frac{\abs{q(z)}}{\abs{\phi'_N(z) \phi_N^N(z)}} \frac{\abs{F_N(z)}^2}{\abs{f_N(\phi_N(z))}} \omega(z)\,\d \A(z)  \\
&= C N^{-\varkappa-\beta}\int_{\mathscr E_N'} \abs{q(z) F_N(z)}\omega(z)\, \d \A(z).
\end{align*}
We can further split the remaining integral and employ \eqref{eq:prf-quasi-ortho-F0} together with the \(L^2\)-estimate \(\norm{F_N}_{L^2(\mathscr D\cap \mathscr E_N',\omega)}=1+\mathcal O(N^{-\varkappa-\beta})\) to obtain
\begin{align*}
\int_{\mathscr E_N'} \abs{q F_N}\omega \d \A&=\int_{\mathscr D\cap\mathscr E_N'} \abs{q F_N}\omega \d \A+\int_{\mathscr E_N'\setminus\mathscr D} \abs{q(z) F_N(z)}\omega(z) \d \A(z)\\
&\leq \norm{q}_{L^2(\mathscr D\cap \mathscr E_N', \omega)}\norm{F_N}_{L^2(\mathscr D\cap \mathscr E_N', \omega)}+\mathcal O\left(\norm{q}_{L^2(\mathscr D,\omega)} N^{-\varkappa-\beta}\right)\\
&\leq \mathcal O \left(\norm{q}_{L^2(\mathscr D,\omega)}\right).
\end{align*}
This completes the proof.
\end{proof}

We take this opportunity to point out a consequence of the precise setup of \cref{prop: quasi ortho}. The condition \(A\geq(\varkappa+\alpha)/2\) reflects that the foliation has to extend sufficiently far into \(\mathscr D\) to obtain the \(L^2\)-estimate and quasi-orthogonality with the desired error rate. Controlling the flow throughout the resulting boundary layer requires the underlying Laurent families to remain stable there. This explains why we have considered Laurent polynomials of degree \(\mathcal O(L_N)\), as this degree scale allows the foliation to penetrate sufficiently far into \(\mathscr D\) while preserving the required estimates.

\subsection{Faber Polynomialization and Global Quasi-Orthogonality}\label{sec:faber}
We aim to find an appropriate polynomial approximation of \(F_N\) on \(\mathscr E_N'\). 
In contrast to \cite{hedenmalm2024riemann} we do not use Hörmander's \(\bar\partial\)-estimate. Instead we follow the ideas of \cite{suetin1974polynomials} and  apply an approximation scheme based on Faber polynomials.

We recall the model function 
\[F_N=\sqrt{N+1} \phi_N'\phi_N^N \exp \left({\textstyle \sum_{j=0}^\varkappa N^{-j} \nu_{N,j}\circ \phi_N}\right),\]
where \(\nu_{N,j} = \mathcal F_{L_N}[\nu_j]\) for some \(\nu_{j}\in C^{\varkappa-j,\alpha}(\mathbb T)\) admitting only  Fourier coefficients corresponding to non-positive indices. By \(B_N\) we denote the Faber polynomials induced by \(F_N\). This is the unique polynomial of degree \(N\) satisfying, for \(z\in \mathscr D_{N,1}\) and \(\rho >1\),
\begin{equation}\label{eq: def Faber2}
B_N(z) =  \mfrac{1}{2\pi i}\int_{\Gamma_{N,\rho}}\frac{F_N(\zeta)}{\zeta -z}\d\zeta,
\end{equation}
where \(\Gamma_{N,\rho}= \phi_N^{-1}(\mathbb T_\rho)\). In other words, \(B_{N}\) can be seen as the polynomial part of \(F_N\).

\begin{lemma}\label{lem:Faber-L2-approx}
For every \(A,s>0\) we have
\begin{equation}\label{eq:Faber-L2-approx1}
F_N(z)-B_N(z)=\mathcal O(N^{-s}),\qquad e^{-\frac{A\log N}{N}}\leq \abs{\phi_N(z)}. 
\end{equation}
Moreover, we have the \(L^2\)-estimate
\begin{equation}\label{eq:Faber-L2-approx2}
\norm{F_N-B_N}_{L^2(\mathscr E_N',\omega)}=\mathcal O\left(N^{-s}\right).
\end{equation}
\end{lemma}

\begin{proof}
We notice that \eqref{eq:Faber-L2-approx2} follows from \eqref{eq:Faber-L2-approx1}. Let \(\delta_N=\frac{A\log N}{N}\), then for sufficiently large \(N\), we have 
\[\abs{\psi_{N,\delta_N}(\zeta)}= e^{-\delta_N}+\mathcal O(N^{-1}),\qquad \zeta\in \mathbb T.\]
Consequently, we get \(\mathscr E_N'\subset \mathscr D_{N,r_N'}^c\) for \(r_N' = e^{-\frac{A'\log N}{N}}\) and any \(A'>A\). Therefore, it is sufficient to prove \eqref{eq:Faber-L2-approx1}.

The difference \(F_N-B_N\) can be captured in terms of the integral
\[F_N(z)-B_N(z) =- \mfrac{1}{2\pi i}\int_{\Gamma_{N,r_N}}\frac{F_N(\zeta)}{\zeta-z}\,\d \zeta,\]
where \(\Gamma_{N,r_N} = \phi_N^{-1}(\mathbb T_{r_N})\) for \(r_N=e^{-\frac{\sigma \log N}{N}}\) and any \(\sigma>A\).
We will assume for the moment that \(\tabs{\zeta-z}^{-1}=\mathcal O(N)\) uniformly for \(e^{-\frac{A\log N}{N}}\leq \tabs{\phi_N(z)}\) and \(\zeta \in \Gamma_{N,r_N}\). 
Then using that \(\nu_{N,j}\) is uniformly bounded, we obtain for some \(c>0\)
\[
\abs{F_N(z)-B_N(z)}\leq \frac{c}{2\pi} N^{3/2}\int _{\Gamma_{N,r_N}} \abs{\phi'_N(\zeta)\phi^N_N(\zeta)}\abs{\d \zeta}=cN^{3/2} r_N^{N+1} \leq c  N^{3/2-\sigma}.
\]
Thus, by choosing \(\sigma\) sufficiently large \eqref{eq:Faber-L2-approx1} follows. 

We turn now to estimating \(\abs{\zeta-z}\). The general strategy follows ideas presented in the proof of \cref{lem:min-curve}.
Let \(\rho_N= e^{-\frac{A\log N}{N}}\), and choose minimizers \(z_N\) and \(\zeta_N\) of the minimization problem
\[\min_{\substack{\abs{\phi_N( z)}\geq \rho_N\\\abs{\phi_N( \zeta)}=r_N}}\abs{ \zeta - z}.\]
Then \(\abs{\phi_N(z_N)}=\rho_N\) and the line segment \(\gamma_N =[\zeta_N,z_N]\) is contained in \(\mathscr D_{N,r_N}^c\). We observe
\[\abs{\phi_N(\zeta_N)-\phi_N(z_N)}=\abs{\int_{\gamma_N} \phi'_N(w)\,\d w}\leq \abs{\zeta_N-z_N}\max_{w\in \mathscr D_{N,r_N}^c}\abs{\phi_N'(w)}.\]
Using the maximum principle together with \cref{coro:uni-conv-Fourier} we get for some \(c>0\)
\[\max_{w\in\mathscr D_{N,r_N}^c}\abs{\phi_N'(w)}=\max_{u\in \mathbb T_{r_N}}\abs{\phi_N'(\phi_N^{-1}(u))}=\left(\min_{u\in\mathbb T} (\phi_N^{-1})'(r_Nu)\right)^{-1}<c,\]
where we have used that \(\phi^{-1}\in C^{\varkappa+1,\alpha}(\mathbb T)\) and that the derivative \((\phi^{-1})'\) does not vanish.
Putting things together yields
\[\abs{\zeta-z}\geq \abs{\zeta_N-z_N}\geq c^{-1} \abs{\phi_N(\zeta_N)-\phi_N(z_N)}\geq c^{-1}(\rho_N-r_N),\]
and consequently \(\abs{\zeta-z}^{-1}=\mathcal O(\frac{N}{\log N})\).
\end{proof}

It turns out that the Faber polynomials \(B_N\) not only approximate \(F_N\) efficiently but also inherit its quasi-orthogonality. Moreover, this property holds throughout \(\mathscr D\), rather than only on \(\mathscr D\cap\mathscr E_N'\). The following result may therefore be viewed as an analogue of \cref{prop: quasi ortho}.

\begin{proposition}\label{prop: quasi ortho Faber}
For any \(\beta\in(0,\alpha)\), we have \(\tnorm{ B_N}_{L^2(\mathscr D,\omega)}=1+\mathcal O(N^{-\varkappa-\beta})\) and 
\[\int_{\mathscr D} q(z)\overline{B_N(z)}\omega(z)\,\d\A(z)=\mathcal{O}\left(N^{-\varkappa-\beta}\norm{ q}_{L^2( \mathscr D,\omega)}\right),\quad q\in \mathrm{Pol}_{N-1},\ N\to\infty.\]
\end{proposition}

\begin{proof}
As before, let \(A\) denote the implicit constant determining the width of \(\mathscr E_N'\). We take \(A\) sufficiently large to ensure that \cref{prop: quasi ortho} and the subsequent estimates hold. Note that the Faber polynomials \(B_N\) are independent of the particular choice of \(A\).

Let us first consider \(\tnorm{B_N}_{L^2(\mathscr D,\omega)}\). Splitting the integral gives
\[\norm{B_N}_{L^2(\mathscr D,\omega)}^2= \norm{B_N}_{L^2(\mathscr D\setminus \mathscr E_N',\omega)}^2+\norm{B_N}_{L^2(\mathscr D \cap \mathscr E_N',\omega)}^2\eqqcolon I_1+I_2.\]
Since \(\partial\mathscr D\subset\mathscr E_N'\), the set \(\mathscr D\setminus\mathscr E_N'\) is bounded away from \(\partial\mathscr D\).
Using \cref{prop: quasi ortho} and \cref{lem:Faber-L2-approx} we immediately  get \(I_2 = 1+\mathcal O(N^{-\varkappa-\beta})\). 

To estimate \(I_1\) we will bound \(B_N\) using its contour integral representation.  We set \(\delta_N=\frac{A\log N}{N}\), \(r_N=e^{-\delta_N/2}\) and \(r_N'=e^{-\delta_N/3}\). For sufficiently large \(N\), it holds \(\mathscr D\setminus \mathscr E_N'\subset \mathscr D_{N,r_N}\). For \(z\in \mathscr D_{N,r_N}\) and some \(c,c'>0\) we have
\[\abs{B_N(z)} \leq \frac{\sqrt{N+1}\,c}{2\pi}\int_{\Gamma_{N,r_N'}}\frac{\abs{\phi_N'(\zeta)\phi_N^N(\zeta)}}{\abs{\zeta-z}}\abs{\d\zeta}\leq c' N^{3/2} e^{-\frac{(N+1)\delta_N}{3}}=\mathcal O(N^{3/2-A/3}),\]
where we have used that \(\nu_{N,j}\) are uniformly bounded on the given contours and that \(\tabs{\zeta-z}^{-1}=\mathcal O(N)\) as shown in the proof of \cref{lem:Faber-L2-approx}. By taking \(A\) sufficiently large, we get \(B_N(z) = \mathcal O(N^{-2(\varkappa+\beta)})\) on \(z\in \mathscr D_{N,r_N}\) and hence \(I_{1}=\mathcal O(N^{-2(\varkappa+\beta)})\). Combining the estimates for \(I_1\) and \(I_2\) yields \(\tnorm{B_N}_{L^2(\mathscr D,\omega)} =1+\mathcal O(N^{-\varkappa-\beta})\).

Let us now turn to the second statement. Using the results above as well as \cref{prop: quasi ortho} together with \cref{lem:Faber-L2-approx} we get
\begin{align*}
\bigg |\int_\mathscr D q \overline{B_N}\omega \d \A\bigg|&=\bigg |\int_{\mathscr D\setminus \mathscr E_N'}q \overline {B_N} \omega \d\A+ \int_{\mathscr D\cap \mathscr E'_N}q \overline{B_N} \omega \d \A\bigg |\\
&\hspace{-1.5cm}\leq \norm{ q}_{L^2(\mathscr D,\omega)}\left(\norm{B_N}_{L^2(\mathscr D\setminus\mathscr E_N',\omega)}+\norm{ F_N-B_N}_{L^2(\mathscr D\cap\mathscr E_N',\omega)}\right)+\bigg|\int_{\mathscr D\cap\mathscr E_N'}q\overline {F_N} \omega\d\A\bigg|\\
&\hspace{-1.5cm}=\mathcal O\Bigl(\norm{ q}_{L^2(\mathscr D,\omega)} N^{-\varkappa-\beta}\Bigr),
\end{align*}
where we have used \(\norm{B_N}_{L^2(\mathscr D\setminus\mathscr E_N',\omega)}=\sqrt{I_1} = \mathcal O(N^{-\varkappa-\beta})\).
\end{proof}

\subsection{Pointwise Asymptotics of the Orthogonal Polynomials}\label{sec:asymp expansion}
Following the considerations of \cite{hedenmalm2024riemann} we turn to the central section of the paper, where we prove our main result, namely, the asymptotics of the planar orthogonal polynomials. By construction, the functions \(\nu_j\) are holomorphic on \(\mathbb D^e\), admit \(C^{\varkappa-j,\alpha}\) boundary values on \(\mathbb T\), and satisfy \(\nu_j(\infty)\in\mathbb R\). Together with the normalization \(\phi_N'(\infty)=\phi'(\infty)>0\) we notice that the leading coefficient of \(B_N\) is positive. Indeed, as \(z\to\infty\),
\begin{align*}
B_N(z)&= F_N(z)+ \mfrac{1}{2\pi i}\int_{\partial \mathscr D}\frac{F_N(\zeta)}{\zeta-z}\,\d\zeta\\
&= \sqrt{N+1}\phi_N'(z)\phi_N^N(z)\exp\left({\textstyle \sum_{j=0}^{\varkappa}N^{-j}\nu_{N,j}(\phi_N(z))}\right)+\mathcal O(z^{-1})\\
&=\sqrt{N+1} (\phi'(\infty))^{N+1}\exp\left({\textstyle \sum_{j=0}^{\varkappa}N^{-j}\nu_j(\infty)}\right) z^N + \mathcal O(z^{N-1}).
\end{align*}

Let \(\mathcal I\) denote the identity operator on \(L^2(\mathscr D,\omega)\) and \(\mathcal S_N\) the orthogonal projection onto \(\mathrm{Pol}^2_{N-1}(\mathscr D,\omega)\). We define
\[Q_N\coloneqq \mathcal S_NB_N,\qquad P_N^*\coloneqq (\mathcal I-\mathcal S_N)B_N=B_N-Q_N.\]
Since \(Q_N\) has degree at most \(N-1\), the polynomials \(P_N^*\) and \(B_N\) have the same leading coefficient. Moreover, \(P_N^*\) is orthogonal to \(\mathrm{Pol}^2_{N-1}(\mathscr D,\omega)\) and is therefore a constant multiple of the orthogonal polynomial \(P_N\). Since both polynomials have positive leading coefficients, there exists \(c_N>0\) such that
\[P_N^*=c_NP_N.\]
Let \(\scalar{\cdot}{\cdot}_{\mathscr D,\omega}\) denote the inner product on \(L^2(\mathscr D,\omega)\). Using the self-adjointness of \(\mathcal S_N\) and \cref{prop: quasi ortho Faber}, we obtain, for any \(\beta\in(0,\alpha)\),
\begin{equation}\label{eq:norm Q_N}
\norm{Q_N}_{L^2(\mathscr D,\omega)}=\frac{\scalar{\mathcal S_N B_N}{Q_N}_{\mathscr D,\omega}}{\norm{Q_N}_{L^2(\mathscr D,\omega)}}=\frac{\scalar{ B_N}{Q_N}_{\mathscr D,\omega}}{\norm{Q_N}_{L^2(\mathscr D,\omega)}}=\mathcal O(N^{-\varkappa-\beta}),
\end{equation}
assuming \(\tnorm{Q_N}_{L^2(\mathscr D,\omega)}\neq 0\). If the norm of \(Q_N\) vanishes, then the estimate is trivially satisfied.
Consequently,
\[\norm{c_NP_N-B_N}_{L^2(\mathscr D,\omega)}
=\norm{P_N^*-B_N}_{L^2(\mathscr D,\omega)}
=\norm{Q_N}_{L^2(\mathscr D,\omega)}
=\mathcal O(N^{-\varkappa-\beta}).\]
Furthermore, \cref{prop: quasi ortho Faber} gives \(\tnorm{B_N}_{L^2(\mathscr D,\omega)}^2=1+\mathcal O(N^{-\varkappa-\beta})\). Since \(B_N=P_N^*+Q_N\),
where the two terms on the right side are orthogonal, the Pythagorean identity and the normalization
\(\tnorm{P_N}_{L^2(\mathscr D,\omega)}=1\) yield
\[c_N^2
=\norm{P_N^*}_{L^2(\mathscr D,\omega)}^2
=\norm{B_N}_{L^2(\mathscr D,\omega)}^2-\norm{Q_N}_{L^2(\mathscr D,\omega)}^2
=1+\mathcal O(N^{-\varkappa-\beta}).\]
Hence, \(c_N=1+\mathcal O(N^{-\varkappa-\beta})\). Fix some \(A>0\) and set \(r_N=e^{-\frac{A\log N}{N}}\). Employing \eqref{eq:Faber-L2-approx1}, we get uniformly for \(z\in\mathscr D_{N,r_N}^c\)
\begin{align*}
P_N(z)
&=\frac{1}{c_N}\left(B_N(z)-Q_N(z)\right)\\
&=(F_N(z)-Q_N(z)+\mathcal O(N^{-A-\varkappa-\beta}))(1+\mathcal O(N^{-\varkappa-\beta})).
\end{align*}
Since \(\abs{\phi_N'}\) is uniformly bounded away from zero and the functions \(\nu_{N,j}\circ\phi_N\) are uniformly bounded, it follows that
\[\abs{F_N(z)}\geq c\sqrt{N+1}\abs{\phi_N^N(z)}\geq c \sqrt{N+1}N^{-A},\qquad z\in \mathscr D_{N,r_N}^c ,\]
for some \(c>0\). Consequently, for \(z\in\mathscr D_{N,r_N}^c\)
\begin{equation}\label{eq:decomp P_N}
P_N(z)=F_N(z)\Bigl(1+\mathcal O\Bigl(N^{-1/2}\abs{Q_N(z)\phi_N^{-N}(z)}+N^{-\varkappa-\beta}\Bigr)\Bigr)\Bigl(1+\mathcal O(N^{-\varkappa-\beta})\Bigr).
\end{equation}

\begin{theorem}\label{theo:Expansion1}
Let \(\varkappa\in \mathbb Z_{\geq 0}\), \(\alpha\in(0,1)\), \(\mathscr D\) be a \(C^{\varkappa+1,\alpha}\)-smooth Jordan domain, and \(\omega\) a \(C^{\varkappa,\alpha}\)-smooth (non-negative) weight defined in a neighborhood of \(\overline{\mathscr D}\) such that \(\omega>0\) on \(\partial\mathscr D\).
Then there exist bounded functions \(\nu_j\), holomorphic on \(\mathbb D^e\), real-valued at infinity, and with boundary values in \(C^{\varkappa-j,\alpha}(\mathbb T)\), such that for any \(\beta\in(0,\alpha)\) and \(\gamma \in [0,\varkappa+\beta)\) it holds uniformly for \(z\) with \(\dist(z,\mathscr D)\geq N^{-2\gamma}\)
\[
P_N(z) = \sqrt{N+1}\phi'(z)\phi^N(z)\exp\left({\textstyle \sum_{j=0}^\varkappa N^{-j} \nu_{j}(\phi(z))}\right) \left(1+\mathcal O(N^{\gamma-\varkappa-\beta})\right).
\]
\end{theorem}
\begin{proof}
Let  \(K_N = \left\{z\in\mathbb C\,|\, \dist(z,\mathscr D)\geq N^{-2\gamma}\right\}\) and \(\beta\in(0,\alpha)\) arbitrary. The polynomial \(Q_N\) is of degree at most \(N-1\), hence the function \(Q_N/\phi_N^N\) is finite valued at infinity. By applying the Cauchy integral formula we observe for \(z\in\mathscr D^e\)
\begin{equation}\label{eq:prf-EXP1-1}
\left(\frac{Q_N(z)}{\phi_N^N(z)}\right)^2 =- \frac{1}{2\pi i} \int_{\partial\mathscr D} \left(\frac{Q_N(\zeta)}{\phi_N^N(\zeta)}\right)^2\frac{\d \zeta}{\zeta-z}
\end{equation}
Then \cref{theo:uni-approx-phi} provides 
\begin{equation}\label{eq:prf-EXP1-2}
\abs{\phi_N(z)}^{-N}=\left(\abs{\phi(z)}+\mathcal O(N^{-\varkappa-1-\beta})\right)^{-N}=\abs{\phi(z)}^{-N}\left( 1+\mathcal O(N^{-\varkappa-\beta})\right)
\end{equation}
uniformly for \(z\in \mathscr D^{c}\). Combining \eqref{eq:prf-EXP1-1}, \eqref{eq:prf-EXP1-2} with \(\abs{\phi(z)}^{-N}\leq 1\) for \(z\in K_N\), we get
\begin{align*}
\abs{\frac{Q_N(z)}{\phi_N^N(z)}}^2&\leq \frac{1}{2\pi} \int_{\partial\mathscr D} \frac{\abs{Q_N(\zeta)}^2}{\abs{\zeta-z}} \abs{\d \zeta}\,\left(1+\mathcal O(N^{-\varkappa-\beta})\right)\\
&\leq \int_{\partial\mathscr D} \abs{Q_N(\zeta)}^2\abs{\d \zeta}\mathcal O \left(\frac{1+ N^{-\varkappa-\beta}}{\dist(z,\mathscr D)}\right).
\end{align*}
Further, using \eqref{eq:lem poly est2} and \eqref{eq:norm Q_N} we obtain for \(z\in K_N\)
\[\abs{\frac{Q_N(z)}{\phi_N^N(z)}}^2\leq  N\norm{Q_N}_{L^2(\mathscr D,\omega)}^2 \mathcal O \left(\frac{1+ N^{-\varkappa-\beta}}{\dist(z,\mathscr D)}\right) \leq \mathcal O\left(N^{1+2(\gamma-\varkappa-\beta)}\right)\]
We conclude
\[\abs{\frac{Q_N(z)}{\phi_N^N(z)}}=\mathcal O(N^{1/2+\gamma-\varkappa-\beta}),\qquad z\in K_N.\]
Combining this with \eqref{eq:decomp P_N} yields
\[P_N(z) = F_N(z) (1+\mathcal O(N^{\gamma-\varkappa-\beta})),\qquad z\in K_N.\]

Next, \cref{theo:Fourier-Master,theo:uni-approx-phi} allow us to replace \(\phi_N^N\) and \(\phi_N'\) by \(\phi^N\) and \(\phi'\), respectively, with a relative error of order \(\mathcal O(N^{-\varkappa-\beta})\), uniformly on \(\mathscr D^c\). Thus, for \(z\in\mathscr D^c\) we have
\[F_N(z)=\sqrt{N+1}\phi'(z)\phi^N(z)\exp\Bigl({\textstyle \sum_{j=0}^\varkappa N^{-j} \nu_{N,j}(\phi_N(z))}\Bigr)\Bigl(1+\mathcal O(N^{-\varkappa-\beta})\Bigr).\]
It remains to prove \(\nu_{N,j}(\phi_N(z))=\nu_{j}(\phi(z))+\mathcal O(N^{j-\varkappa-\beta})\) for \(z\in\mathscr D^c\). Note that using the maximum principle it is sufficient to consider \(z\in \partial \mathscr D\). We begin by observing
\begin{align*}
\max_{z\in\partial \mathscr D}\abs{\nu_{N,j}(\phi_N(z))-\nu_j(\phi(z))}&\leq \max_{z\in\partial \mathscr D} \abs{\nu_{N,j}(\phi(z))-\nu_{N,j}(\phi_N(z))}+\norm{\nu_j -\nu_{N,j}}_{L^\infty(\mathbb T)}\\
&\leq \max_{z\in\partial \mathscr D} \abs{\nu_{N,j}(\phi(z))-\nu_{N,j}(\phi_N(z))}+\mathcal O(N^{j-\varkappa-\beta}),
\end{align*}
where we have used \cref{theo:Bernstein0}.
Now let \(r_N = e^{-\frac{\log N}{N}}\), then \(\phi_N(\partial \mathscr D)\subset \mathbb D_{r_N}^c\) for large \(N\). Hence for each \(z\) we can find a circular arc \(\gamma_z\subset \mathbb D_{r_N}^c\) connecting \(\phi_N(z)\) and \(\phi(z)\), where the length of \(\gamma_z\) is bounded by \(\frac{\pi}{2}\abs{\phi_N(z)-\phi(z)}\). We conclude using the maximum principle and \cref{theo:conf-inverse}
\begin{align*}
\abs{\nu_{N,j}(\phi_N(z))-\nu_{N,j}(\phi(z))}&\leq\mfrac{\pi}{2}\norm{\nu_{N,j}'}_{L^\infty(\mathbb T_{r_N})}\abs{\phi_N(z)-\phi(z)}\\
&=\norm{\nu_{N,j}'(r_N\cdot)}_{L^\infty(\mathbb T)}\mathcal O(N^{-\varkappa-1-\beta}),\qquad z\in\partial \mathscr D.
\end{align*}
Further, we recall the relation \(D_\tau \nu_{N,j}(z) = iz \nu_{N,j}'(z)\), as displayed in \eqref{eq:polar-to-wirtinger}, then the Bernstein inequality \eqref{eq:bernstein-ineq1} and \cref{theo:Fourier-Master} yield
\[\norm{\nu_{N,j}'(r_N\cdot)}_{L^\infty(\mathbb T)}= r_N^{-1} \norm{D_\tau \nu_{N,j}(r_N\cdot)}_{L^\infty(\mathbb T)}\leq r_N^{-1} L_N \norm{\nu_{N,j}(r_N\cdot)}_{L^\infty(\mathbb T)}\leq \mathcal O(N).\]
We conclude that \(\nu_{N,j}(\phi_N(z))=\nu_{j}(\phi(z))+\mathcal O(N^{j-\varkappa-\beta})\) and therefore it holds
\begin{equation}\label{eq:prf-Main1-F_N-on-Dc}
F_N(z)=\sqrt{N+1}\phi'(z)\phi^N(z)\exp\Bigl({\textstyle \sum_{j=0}^\varkappa N^{-j} \nu_{j}(\phi(z))}\Bigr)\Bigl(1+\mathcal O(N^{-\varkappa-\beta})\Bigr),
\end{equation}
uniformly for \(z\in\mathscr D^c\). Combining \eqref{eq:decomp P_N} and \eqref{eq:prf-Main1-F_N-on-Dc} we establish the desired result.
\end{proof}

\begin{theorem}\label{theo:Expansion2}
Let \(\varkappa\in \mathbb Z_{\geq 0}\), \(\alpha\in(0,1)\), \(\mathscr D\) be a \(C^{\varkappa+1,\alpha}\)-smooth Jordan domain, and \(\omega\) a \(C^{\varkappa,\alpha}\)-smooth (non-negative) weight defined in a neighborhood of \(\overline{\mathscr D}\) such that \(\omega>0\) on \(\partial\mathscr D\).
Then there exist bounded functions \(\nu_j\), holomorphic on \(\mathbb D^e\), real-valued at infinity, and with boundary values in \(C^{\varkappa-j,\alpha}(\mathbb T)\), such that for any \(\beta\in(0,\alpha)\) it holds 
\[P_N(z) = \sqrt{N+1}\phi'(z)\phi^N(z)\exp\left({\textstyle \sum_{j=0}^\varkappa N^{-j} \nu_{j}(\phi(z))}\right) \left(1+\mathcal O(N^{1/2-\varkappa-\beta})\right),\]
uniformly for \(z\in\mathscr D^c\).
Moreover, let \(\nu_{N,j}=\mathcal F_{L_N}[\nu_j]\)  and \(\phi_N = (\mathcal F_{L_N}[\phi^{-1}])^{-1}\), then
for and any \(R>0\) it holds 
\[P_N(z) = \sqrt{N+1}\phi_N'(z)\phi_N^N(z)\exp\Bigl({\textstyle \sum_{j=0}^\varkappa N^{-j} \nu_{N,j}(\phi_N(z))}\Bigr) \Bigl(1+\mathcal O(N^{1/2+R-\varkappa-\beta})\Bigr)\]
uniformly for \(\dist(z,\mathscr D^c)\leq \frac{R\log N}{\norm{\phi'}_{L^\infty(\partial\mathscr D)} N}\).
\end{theorem}

\begin{proof}
We notice that the first asymptotic formula immediately follows from the second one by applying \eqref{eq:prf-Main1-F_N-on-Dc}. Therefore, it is sufficient to establish the second estimate.

Let \(\varepsilon>0\) be small and  \(r_N = e^{-\frac{(1+\varepsilon) R \log N}{ N}}\), then for large \(N\) \cref{lem:min-curve} gives
\begin{equation}\label{eq:prf-EXP2-1}
K_N\coloneqq \left\{z\,\left|\, \dist(z,\mathscr D^c)\leq \mfrac{R\log N}{\norm{\phi'}_{L^\infty(\partial\mathscr D)} N}\right\}\right.\subset \mathscr D_{N,r_N}^c.
\end{equation}
The polynomial \(Q_N\) has degree at most \(N-1\), hence \(\frac{Q_N}{\phi_N^N}\) vanishes at infinity. Using the maximum principle we observe
\[\max_{z\in K_N}\abs{\frac{Q_N(z)}{\phi_N^N(z)}}\leq \max_{z\in \mathscr D_{N,r_N}^c}\abs{\frac{Q_N(z)}{\phi_N^N(z)}}=\max_{\abs{\phi_N(z)}=r_N}
r_N^{-N}\abs{Q_N(z)}\leq \max_{z\in \overline{\mathscr D}} N^{(1+\varepsilon)R}\abs{Q_N(z)}.\]
 Now \eqref{eq:lem poly est1} together with   \eqref{eq:norm Q_N}  yields for any \(\gamma\in(0,\alpha)\) 
\[\max_{z\in \overline{\mathscr D}}\abs{Q_N(z)}=\mathcal O(N \norm{Q_N}_{L^2(\mathscr D,\omega)})=\mathcal O(N^{1-\varkappa-\gamma}).\]
Thus, if we choose \(\gamma\) and \(\varepsilon\) such that \(\gamma-\varepsilon R=\beta\) then it follows 
\[\frac{Q_N(z)}{\phi_N^N(z)}=\mathcal O\bigl(N^{(1+\varepsilon) R+1 -\varkappa-\gamma}\bigr)=\mathcal O\bigl(N^{R+1-\varkappa-\beta}\bigr),\qquad z\in   K_N,\]
which in combination with \eqref{eq:decomp P_N} provides the second asymptotic formula.
\end{proof}

\begin{proof}[Proof of \cref{thm:Main1,thm:Main2}] 
We observe that \cref{thm:Main2} follows directly from \cref{theo:Expansion2}. To establish \cref{thm:Main1}, we distinguish two cases. If \(K\cap \partial \mathscr D \neq \emptyset\), the claimed error bound follows directly from \cref{theo:Expansion2}. If \(K\cap \partial \mathscr D = \emptyset\), we apply \cref{theo:Expansion1} with some \(\beta'\in (\beta,\alpha)\) and \(\gamma \in(0,\beta'-\beta)\). For all sufficiently large \(N\), the set \(K\) is contained in \(\{z:\dist(z,\mathscr D)\geq N^{-2\gamma}\}\). The resulting error term is \(\mathcal O(N^{\gamma-\varkappa-\beta'})\), and hence also \(\mathcal O(N^{-\varkappa-\beta})\).
\end{proof}

\begin{proof}[Proof of \cref{coro:main}]
The key observation is, that the asymptotic expansions in \cref{theo:Expansion1,theo:Expansion2} do not admit any roots, wherever they are valid. Therefore, \cref{theo:Expansion1} directly implies the first part of \cref{coro:main}, regarding the case \(\varkappa+\alpha\leq 1/2\). Turning to \(\varkappa+\alpha > 1/2\), we first notice that for large \(N\), all roots of \(P_N\) have to be inside of \(\overline{\mathscr D}\), due to \cref{theo:Expansion2}. We define \(d_N = \min\{\dist(z,\mathscr D^c)\,|\, P_N(z)=0\}\) and set \(\Omega_{N,R}\) to be the set of all \(z\) satisfying \(\dist(z,\mathscr D^c)\leq \frac{R\log N}{\norm{\phi'}_{L^\infty(\partial\mathscr D)} N}\). Then for any \(R< 1/2-\varkappa-\beta\) and large \(N\), the orthogonal polynomial \(P_N\) does not have any roots in \(\Omega_{N,R}\). Consequently, 
\[\liminf_{N\to \infty} \frac{d_N N}{\log N}\geq \frac{R}{\norm{\phi'}_{L^\infty(\partial\mathscr D)}},\]
for all \(R\leq 1/2-\varkappa-\beta\), which gives the desired result.
\end{proof}

\appendix
\crefalias{section}{appendix}
\addtocontents{toc}{\protect\addvspace{1em}}

\section{Proof of Theorem 2.1}\label{sec:App-add-proofs}
The uniform estimate \eqref{eq:theo-Bernstein-uni1} is a classical result within Fourier approximation theory. It follows directly from \cite[Thm.~13.6]{zygmund1959trigseries} and \cite[Rem.~(e), p.~120]{zygmund1959trigseries}. Hence it remains to prove \eqref{eq:theo-Bernstein-uni2}.

We first consider the case \(\varkappa=0\). Let \(K_N(\zeta)=\sum_{k=-N}^N\zeta^k\) for \(\zeta\in\mathbb T\) denote the \(N\)-th Dirichlet kernel. Then the kernel is bounded by \(\abs{K_N(\zeta)}\leq 2\abs{1-\zeta}^{-1}\) and it recovers the Fourier truncation of \(f\) in terms of the convolution
\begin{equation}\label{eq:prf-Bernstein-convo}
\mathcal F_N[f](\zeta)=\mfrac{1}{2\pi}\int_{\mathbb T} f(\zeta\overline{w})K_N(w)\abs{\d w},\qquad \zeta\in\mathbb T.
\end{equation}
Since \(\frac{1}{2\pi}\int_{\mathbb T}K_N(w)\abs{\d w}=1\), it follows
\begin{equation}\label{eq:prf-Bernstein-0}
\mathcal F_N[f](\zeta)=f(\zeta) +\mfrac{1}{2\pi}\int_\mathbb T (f(\zeta\overline w)-f(\zeta))K_N(w)\abs{\d w},\qquad \zeta\in\mathbb T,
\end{equation}
and therefore for some \(c_\alpha>0\)
\[\norm{\mathcal F_N[f]}_{L^\infty(\mathbb T)} \leq \norm{f}_{L^\infty(\mathbb T)}+\frac{\hnorm{f}_{\mathbb T,\alpha}}{\pi} \int_{\mathbb T}\abs{1-w}^{\alpha-1}\abs{\d w} \leq c_\alpha\norm{f}_{C^{0,\alpha}(\mathbb T)}.\]
For \(\zeta,\eta\in\mathbb T\), we consider
\begin{equation}\label{eq:prf-Bernstein--1}
\begin{split}
\abs{f(\zeta\overline{w})-f(\zeta)-f(\eta\overline{w})+f(\eta)}&\leq 2\hnorm{f}_{\mathbb T,\alpha}\min\{\abs{1-w}^\alpha,\abs{\zeta-\eta}^\alpha\}\\
&\leq 2\hnorm{f}_{\mathbb T,\alpha}\abs{1-w}^{\alpha-\beta}\abs{\zeta-\eta}^\beta,
\end{split}
\end{equation}
where we have used the interpolation \(\min(a,b)\leq a^{1-\gamma}b^\gamma\) for \(\gamma=\beta/\alpha\in(0,1)\). Then, for some \(c_{\alpha,\beta}>0\) it follows from \eqref{eq:prf-Bernstein-0}
\begin{align*}
\abs{\mathcal F_N[f](\zeta)-\mathcal F_N[f](\eta)}
&\leq \hnorm{f}_{\mathbb T,\alpha}\abs{\zeta-\eta}^\alpha+\frac{\hnorm{f}_{\mathbb T,\alpha}}{\pi}\abs{\zeta-\eta}^\beta\int_{\mathbb T}\abs{1-w}^{\alpha-\beta-1}\abs{\d w}\\
&\leq c_{\alpha,\beta}\hnorm{f}_{\mathbb T,\alpha}\abs{\zeta-\eta}^\beta,
\end{align*}
where the integral is finite since \(\alpha-\beta>0\). We conclude
\[\norm{\mathcal F_N[f]}_{C^{0,\beta}(\mathbb T)}
\leq c_{\alpha,\beta}\norm{f}_{C^{0,\alpha}(\mathbb T)}.\]
For general \(\varkappa\), we use
\(D_\tau^j \left(\mathcal F_N[f]\right)=\mathcal F_N[D_\tau^jf]\) for \(0\leq j\leq \varkappa\) and apply the preceding estimate to the derivatives of \(f\).
This gives
\[\norm{\mathcal F_N[f]}_{C^{\varkappa,\beta}(\mathbb T)}
\leq\sum_{j=0}^\varkappa\norm{\mathcal F_N[D_\tau^j f]}_{C^{0,\beta}(\mathbb T)}\leq 2c_{\alpha,\beta}\norm{f}_{C^{\varkappa,\alpha}(\mathbb T)},\]
where the constant is independent of \(f\) and \(N\). This completes the proof.

\section{Hölder Regularity of the Exterior Herglotz Transform}\label{app:herglotz}
The proof of the following proposition adapts ideas from \cite[Ch.~3.4]{privalov1956randeigenschaften}, where similar results are established for Cauchy-type integrals. Related results can further be found in \cite[Ch.~1.5]{Gakhov1966}.

\begin{proposition}
Let \(f\in C^{\varkappa,\alpha}(\mathbb T)\), where \(\varkappa\in\mathbb Z_{\geq 0}\) and \(\alpha\in(0,1)\). Then, for \(0\leq j\leq \varkappa\), the derivative \(\Herg[f]^{(j)}\) extends continuously to \(\mathbb T\). Moreover, \(\Herg[f]\in C^{\varkappa,\alpha}(\mathbb T)\) and
\[\norm{\Herg[f]}_{C^{\varkappa,\alpha}(\mathbb T)}
\leq C_{\varkappa,\alpha}\norm{f}_{C^{\varkappa,\alpha}(\mathbb T)},\]
where \(C_{\varkappa,\alpha}>0\) is independent of \(f\).
\end{proposition}

\begin{proof}
We first consider the case \(\varkappa=0\). Since \(\frac{1}{2\pi}\int_{\mathbb T}\frac{z+\zeta}{z-\zeta}\abs{\d\zeta}=1\) for \(\abs{z}>1\), we may write, for any \(\eta\in\mathbb T\),
\begin{equation}\label{eq:prf-Herg1}
\Herg[f](z)=f(\eta)+\mfrac{1}{2\pi}\int_{\mathbb T} \left(f(\zeta)-f(\eta)\right)\frac{z+\zeta}{z-\zeta}\abs{\d\zeta}.
\end{equation}
The integral remains absolutely convergent for \(z=\eta\), since for some \(c_\alpha>0\)
\begin{equation}\label{eq:prf-Herg2}
\mfrac{1}{2\pi}\int_{\mathbb T}\abs{f(\zeta)-f(\eta)} \frac{\abs{\eta+\zeta}}{\abs{\eta-\zeta}}\abs{\d\zeta}
\leq \frac{\hnorm{f}_{\mathbb T,\alpha}}{\pi} \int_{\mathbb T}\abs{\eta-\zeta}^{\alpha-1}\abs{\d\zeta}
\leq c_\alpha\hnorm{f}_{\mathbb T,\alpha}.
\end{equation}
We therefore define
\begin{equation}\label{eq:prop herglotz def H on T}
G(\eta)\coloneqq f(\eta)+\mfrac{1}{2\pi}\int_{\mathbb T} \left(f(\zeta)-f(\eta)\right) \frac{\eta+\zeta}{\eta-\zeta}\abs{\d\zeta},\qquad \eta\in\mathbb T.
\end{equation}
In particular,
\(\tnorm{G}_{L^\infty(\mathbb T)} \leq c_\alpha\tnorm{f}_{C^{0,\alpha}(\mathbb T)}\) if we additionally assume \(c_\alpha\geq 1\). For \(\eta\in\mathbb T\), define the non-tangential domain
\[ T(\eta)\coloneqq \left\{z\in\mathbb D^e\,|\, \abs{z-\eta}\leq2(\abs{z}-1)\right\}. \]
If \(z\in T(\eta)\) and  \(\delta\coloneqq \abs{z-\eta}\), then \(\abs{z-\zeta}\geq\abs{z}-1\geq\delta/2\) for every \(\zeta\in\mathbb T\). Then we see
\[\frac{z+\zeta}{z-\zeta}-\frac{\eta+\zeta}{\eta-\zeta}=\frac{2\zeta}{z-\zeta}-\frac{2\zeta}{\eta-\zeta}=\frac{2\zeta(\eta-z)}{(z-\zeta)(\eta-\zeta)},\]
and obtain by using \eqref{eq:prf-Herg1}
\begin{align*}
\abs{\Herg[f](z)-G(\eta)}
&\leq\frac{1}{2\pi}\int_\mathbb T \abs{f(\zeta)-f(\eta)}\abs{\frac{z+\zeta}{z-\zeta}-\frac{\eta+\zeta}{\eta-\zeta}}\abs{\d\zeta}\\
&\leq\mfrac{\hnorm{f}_{\mathbb T,\alpha}\delta}{\pi}\int_{\mathbb T}\frac{\abs{\d\zeta}}{\abs{z-\zeta}\abs{\eta-\zeta}^{1-\alpha}}.
\end{align*}
We split \(\mathbb T\) into the sets \(S_1\) and \(S_2\) defined by  \(\abs{\zeta-\eta}\leq2\delta\) and \(\abs{\zeta-\eta}>2\delta\), respectively, and observe
\[\abs{z-\zeta}\geq \begin{cases} \frac{\delta}{2},&\zeta \in S_1,\\ \frac{\abs{\zeta-\eta}}2,& \zeta \in S_2.
\end{cases}\]
Splitting the remaining integral along \(S_1\) and \(S_2\), and using the above estimates yields for some \(c_\alpha'>0\)
\[\int_{\mathbb T} \frac{\abs{\d\zeta}} {\abs{z-\zeta}\abs{\eta-\zeta}^{1-\alpha}}
\leq c_\alpha'\delta^{\alpha-1}.\]
We conclude
\begin{equation}\label{eq:prop herglotz nontangential}
\abs{\Herg[f](z)-G(\eta)}
\leq c_\alpha'\hnorm{f}_{\mathbb T,\alpha} \abs{z-\eta}^\alpha,\qquad z\in T(\eta).
\end{equation}
Note that to this point we have shown that \(\Herg[f]\) admits \(G\) as non-tangential limit.

Now let \(\eta,\mu\in\mathbb T\) and \(\sigma\coloneqq \abs{\eta-\mu}\). The point \(z=(1+\sigma)\eta\) belongs to \(T(\eta)\cap T(\mu)\), while \(\abs{z-\eta}\leq\sigma\) and \(\abs{z-\mu}\leq2\sigma\). Thus,
\[\abs{G(\eta)-G(\mu)}\leq \abs{G(\eta)-\Herg[f](z)}+\abs{\Herg[f](z)-G(\mu)}
\leq 3c_\alpha'\hnorm{f}_{\mathbb T,\alpha} \abs{\eta-\mu}^\alpha.\]
In particular, we have shown that \(G\) is \(\alpha\)-Hölder continuous on \(\mathbb T\).
Using \eqref{eq:prf-Herg2} and \eqref{eq:prop herglotz def H on T} we can bound the \(L^\infty\)-norm of \(G\) and obtain for some \(C_{0,\alpha}>0\) 
\begin{equation}\label{eq:prop herglotz holder estimate}
\norm{G}_{C^{0,\alpha}(\mathbb T)}
\leq C_{0,\alpha}\norm{f}_{C^{0,\alpha}(\mathbb T)}.
\end{equation}

To verify continuity from the exterior, let \(z\to\eta\) and set \(\mu=z/\abs{z}\). Then \(z\in T(\mu)\) as well as \(\max\{\abs{z-\mu},\abs{\mu-\eta}\}\leq\abs{z-\eta}\). Hence, by \eqref{eq:prop herglotz nontangential} and \eqref{eq:prop herglotz holder estimate}, \(\Herg[f](z)\to G(\eta)\). Thus, \(\Herg[f]\) extends continuously onto the circle and coincides there with \(G\), proving the result for \(\varkappa=0\).

Let us turn to a general \(\varkappa\). We notice that  tangential differentiation commutes with the
Herglotz transform, that is \( D_\tau(\Herg[f])=\Herg[D_\tau f]\).
Using \eqref{eq:polar-to-wirtinger} we get for \(j=1,\dots,\varkappa\) 
\[\Herg[f]^{(j)}(\zeta) =(-i\zeta^{-1} D_\tau)^j \Herg[f](\zeta)=\zeta^{-j}\sum_{m=1}^ja_{m,j}  \Herg[D^m_\tau f](\zeta), \]
for suitable coefficients \(a_{m,j}\). Applying the case \(\varkappa=0\) to the tangential derivatives of \(f\), we conclude that \(\Herg[f]^{(j)}\) extends continuously onto \(\mathbb T\). Furthermore, we have
\[\norm{\Herg[f]}_{C^{\varkappa,\alpha}(\mathbb T)}
\leq \sum_{j=0}^\varkappa \norm{D_\tau^j\Herg[f]}_{C^{0,\alpha}(\mathbb T)}\leq C_{0,\alpha}\sum_{j=0}^\varkappa \norm{D_\tau^jf}_{C^{0,\alpha}(\mathbb T)}\leq  C_{\varkappa,\alpha}\norm{f}_{C^{\varkappa,\alpha}(\mathbb T)},\]
for some appropriate \(C_{\varkappa,\alpha}>0.\)
\end{proof}

\section*{Acknowledgments}
I am grateful to Aron Wennman for his valuable feedback, which greatly improved the quality and presentation of this paper. I also thank my office mates and close friends, Sampad Lahiry and Max van Horssen, for keeping me company during the long writing process.

The author is supported by the Odysseus Grant G0DDD23N from the Research Foundation Flanders (FWO).

\section*{Declaration of AI Usage}
All main ideas and constructions are due to the author. AI was used to assist in proving local statements or improving existing proofs, primarily by suggesting standard arguments. Moreover, AI was used to proofread the paper and improve its language.


\begin{thebibliography}{99}

\bibitem{Akemann2021Gegenbauer}
G.~Akemann, T.~Nagao, I.~Parra, and G.~Vernizzi,
\emph{Gegenbauer and other planar orthogonal polynomials on an ellipse in
the complex plane},
Constr. Approx. 53 (2021), no.~3, 441--478.

\bibitem{Ameur2011fluctuations}
Y.~Ameur, H.~Hedenmalm, and N.~Makarov,
\emph{Fluctuations of eigenvalues of random normal matrices},
Duke Math. J. 159 (2011), no.~1, 31--81.

\bibitem{BeckermannStylianopoulos2018Grunsky}
B.~Beckermann and N.~Stylianopoulos,
\emph{Bergman orthogonal polynomials and the Grunsky matrix},
Constr. Approx. 47 (2018), no.~2, 211--235.

\bibitem{carleman1922classic}
T.~Carleman,
\emph{\"Uber die Approximation analytischer Funktionen durch lineare
Aggregate von vorgegebenen Potenzen},
Ark. Mat. Astr. Fys. 17 (1922).

\bibitem{cronvall2025direct}
J.~Cronvall and A.~Wennman,
\emph{A direct approach to soft and hard edge universality for random
normal matrices},
preprint (2025), arXiv:2511.18628.

\bibitem{DeVoreLorentz1993constapprox}
R.~A.~DeVore and G.~G.~Lorentz,
\emph{Constructive approximation},
Grundlehren der mathematischen Wissenschaften, vol.~303,
Springer-Verlag, Berlin, 1993.

\bibitem{DragnevMinaDiaz2010Carleman}
P.~Dragnev and E.~Mi\~na-D\'iaz,
\emph{Asymptotic behavior and zero distribution of Carleman orthogonal
polynomials},
J. Approx. Theory 162 (2010), no.~11, 1982--2003.

\bibitem{gaier1987complex}
D.~Gaier,
\emph{Lectures on complex approximation},
translated from German by R.~McLaughlin,
Birkh\"auser Boston, Boston, MA, 1987.

\bibitem{Gakhov1966}
F.~D.~Gakhov,
\emph{Boundary value problems},
I.~N.~Sneddon (ed.),
International Series of Monographs in Pure and Applied Mathematics,
vol.~85, Pergamon Press, 1966.

\bibitem{GustafssonEtAl2009Archipelago}
B.~Gustafsson, M.~Putinar, E.~B.~Saff, and N.~Stylianopoulos,
\emph{Bergman polynomials on an archipelago: estimates, zeros and shape
reconstruction},
Adv. Math. 222 (2009), no.~4, 1405--1460.

\bibitem{Hedenmalm2024softRHP}
H.~Hedenmalm,
\emph{Soft Riemann--Hilbert problems and planar orthogonal polynomials},
Comm. Pure Appl. Math. 77 (2024), no.~4, 2413--2451.

\bibitem{Hedenmalm2013Coulomb}
H.~Hedenmalm and N.~Makarov,
\emph{Coulomb gas ensembles and Laplacian growth},
Proc. Lond. Math. Soc. (3) 106 (2013), no.~4, 859--907.

\bibitem{Hedenmalm2020spectral}
H.~Hedenmalm and A.~Wennman,
\emph{Off-spectral analysis of Bergman kernels},
Comm. Math. Phys. 373 (2020), no.~3, 1049--1083.

\bibitem{hedenmalm2017planar}
\bysame,
\emph{Planar orthogonal polynomials and boundary universality in the
random normal matrix model},
Acta Math. 227 (2021), no.~2, 309--406.

\bibitem{hedenmalm2024riemann}
\bysame,
\emph{Riemann--Hilbert hierarchies for hard edge planar orthogonal
polynomials},
Amer. J. Math. 146 (2024), no.~2, 371--403.

\bibitem{korovkin1947asymptotic}
P.~P.~Korovkin,
\emph{The asymptotic representation of polynomials orthogonal over
a region},
Dokl. Akad. Nauk SSSR (N.S.) 58 (1947), 1883--1885.

\bibitem{Lubinsky2010universality}
D.~S.~Lubinsky,
\emph{Universality type limits for Bergman orthogonal polynomials},
Comput. Methods Funct. Theory 10 (2010), no.~1, 135--154.

\bibitem{MinaDiaz2008AsymptoticIntegral}
E.~Mi\~na-D\'iaz,
\emph{An asymptotic integral representation for Carleman orthogonal
polynomials},
Int. Math. Res. Not. IMRN (2008), 38~pp.

\bibitem{MinaDiaz2023Bergman}
\bysame,
\emph{Orthogonal polynomials in weighted Bergman spaces},
J. Approx. Theory 296 (2023), Paper No.~105972, 22~pp.

\bibitem{plonka2018numerical}
G.~Plonka, D.~Potts, G.~Steidl, and M.~Tasche,
\emph{Numerical Fourier analysis},
Springer, 2018.

\bibitem{pommerenke1975univalent}
C.~Pommerenke,
\emph{Univalent functions},
Studia Mathematica/Mathematische Lehrb\"ucher, vol.~25,
Vandenhoeck \& Ruprecht, 1975;
with a chapter on quadratic differentials by G.~Jensen.

\bibitem{pommerenke2013boundary}
\bysame,
\emph{Boundary behaviour of conformal maps},
Springer Science \& Business Media, 2013.

\bibitem{privalov1956randeigenschaften}
I.~I.~Privalov,
\emph{Randeigenschaften analytischer Funktionen},
Deutscher Verlag der Wissenschaften, 1956.

\bibitem{Riesz1914}
M.~Riesz,
\emph{Eine trigonometrische Interpolationsformel und einige Ungleichungen
f\"ur Polynome},
Jahresber. Deutsch. Math.-Verein. 23 (1914), 354--368.

\bibitem{Stechkin1948}
S.~B.~Stechkin,
\emph{A generalization of some inequalities of S.~N.~Bernstein},
Dokl. Akad. Nauk SSSR 60 (1948), no.~9, 1511--1514 (Russian).

\bibitem{Stylianopoulos2013StrongAsymptotics}
N.~Stylianopoulos,
\emph{Strong asymptotics for Bergman polynomials over domains with
corners and applications},
Constr. Approx. 38 (2013), no.~1, 59--100.

\bibitem{suetin1966ordercomparison}
P.~K.~Suetin,
\emph{Order comparison of various norms of polynomials in a complex region},
Ural. Gos. Univ. Mat. Zap. 5 (1966), no.~4, 91--100 (Russian).

\bibitem{Suetin1969AreallyOrthogonal}
P.~K.~Suetin,
\emph{Areally orthogonal polynomials and Bieberbach polynomials},
Soviet Math. Dokl. 10 (1969), 1123--1126;
English translation of Dokl. Akad. Nauk SSSR 188 (1969), no.~2, 294--296.

\bibitem{suetin1974polynomials}
\bysame,
\emph{Polynomials orthogonal over a region and Bieberbach polynomials},
American Mathematical Society, 1974.


\bibitem{Zabrodin2006Dyson}
A.~Zabrodin and P.~Wiegmann,
\emph{Large-$N$ expansion for the 2D Dyson gas},
J. Phys. A 39 (2006), no.~28, 8933--8963.

\bibitem{zygmund1959trigseries}
A.~Zygmund,
\emph{Trigonometric series}, 2nd ed.,
Cambridge Mathematical Library, 1959.


\end{thebibliography}
\end{document}